\documentclass[opre,nonblindrev]{informs3} 

\OneAndAHalfSpacedXI

\usepackage{graphicx}
\usepackage{natbib}
 \bibpunct[, ]{(}{)}{,}{a}{}{,}%
 \def\bibfont{\small}%
 \def\bibsep{\smallskipamount}%
\usepackage{color}
\usepackage[linesnumbered,ruled,vlined]{algorithm2e}
\usepackage{algorithmic}
\definecolor{clemson-orange}{RGB}{234,106,32}
\definecolor{chicago-maroon}{RGB}{128,0,0}
\definecolor{northwestern-purple}{RGB}{82,0,99}
\definecolor{lawngreen}{RGB}{0,250,154}
\definecolor{pink}{RGB}{255,0,128}
\definecolor{northwestern-purple}{RGB}{82,0,99}
\definecolor{OliveGreen}{RGB}{	186, 184, 108}
\usepackage{dsfont}
\usepackage{booktabs}
\usepackage{thm-restate}
\usepackage{multirow}
\usepackage{mathrsfs}
\usepackage{caption}
\usepackage[bookmarks,colorlinks,breaklinks]{hyperref}  
\hypersetup{linkcolor=northwestern-purple,citecolor=northwestern-purple,filecolor=northwestern-purple,urlcolor=northwestern-purple}
\usepackage[nameinlink]{cleveref}
\crefname{assumption}{Assumption}{Assumptions}
\crefname{lemma}{Lemma}{Lemmas}
\crefname{theorem}{Theorem}{Theorems}
\crefname{corollary}{Corollary}{Corollaries}
\crefname{prop}{Proposition}{Propositions}
\crefname{claim}{Claim}{Claims}
\crefname{procedure}{Procedure}{Procedures}
\crefname{algorithm}{Algorithm}{Algorithms}
\crefname{figure}{Figure}{Figures}
\crefname{remark}{Remark}{Remarks}
\crefname{section}{Section}{Sections}
\crefname{procedure}{Procedure}{Procedures}
\crefname{proposition}{Proposition}{Propositions}
\crefname{example}{Example}{Examples}
\crefname{equation}{}{}
\crefname{enumi}{}{}
\crefname{table}{Table}{Tables}
\crefname{definition}{Definition}{Definitions}
\crefname{appendix}{Appendix}{Appendices}

\TheoremsNumberedThrough     
\ECRepeatTheorems

\EquationsNumberedThrough    

\usepackage[english]{babel}

\usepackage{etoolbox}
\newcommand{\zerodisplayskips}{%
  \setlength{\abovedisplayskip}{5pt}%
  \setlength{\belowdisplayskip}{5pt}%
  \setlength{\abovedisplayshortskip}{5pt}%
  \setlength{\belowdisplayshortskip}{5pt}}
\appto{\normalsize}{\zerodisplayskips}
\appto{\small}{\zerodisplayskips}
\appto{\footnotesize}{\zerodisplayskips}

\newcommand*{\red}{\textcolor{black}}

\begin{document}


\RUNAUTHOR{El Housni et al.}

\RUNTITLE{Matching Drivers to Riders}

\TITLE{Matching Drivers to Riders: A Two-stage Robust Approach}

\ARTICLEAUTHORS{%
 \AUTHOR{Omar El Housni}
 \AFF{School of Operations Research and Information Engineering, Cornell Tech, New York, NY 10044, \EMAIL{oe46@cornell.edu}}
 \AUTHOR{Vineet Goyal}
 \AFF{Industrial Engineering and Operations Research, Columbia University, New York, NY 10027, \EMAIL{vg2277@columbia.edu}}
 \AUTHOR{Oussama Hanguir}
 \AFF{Industrial Engineering and Operations Research, Columbia University, New York, NY 10027, \EMAIL{oh2204@columbia.edu}}
 \AUTHOR{Clifford Stein}
 \AFF{Industrial Engineering and Operations Research, Columbia University, New York, NY 10027, \EMAIL{cs2035@columbia.edu}}
}
\ABSTRACT{%
Matching demand (riders) to supply (drivers) efficiently is a fundamental problem for ride-sharing platforms who need to match the riders (almost) as soon as the request arrives with only partial knowledge about future ride requests. A myopic approach that computes an optimal matching for current requests ignoring future uncertainty can be highly sub-optimal. In this paper, we consider a two-stage robust optimization framework for this matching problem where future demand uncertainty is modeled using a set of demand scenarios (specified explicitly or implicitly). The goal is to match the current request to drivers (in the first stage) so that the cost of first-stage matching and the worst-case cost over all scenarios for the second-stage matching is minimized. We show that the two-stage robust matching is NP-hard under various cost functions and present constant approximation algorithms for different settings of our two-stage problem. Furthermore, we test our algorithms on real-life taxi data from the city of Shenzhen and show that they substantially improve upon myopic solutions and reduce  the maximum wait time of the second-stage riders.
 
}%



\maketitle

%


\section{Introduction}\label{sec:intro}
Matching demand (riders) with supply (drivers) is a fundamental problem for ride-hailing platforms such as Uber, Lyft and DiDi, who continually need to match drivers to current riders efficiently with only  partial knowledge of future ride requests.   
A common approach in practice is batched matching: instead of matching each request sequentially as it arrives, aggregate the requests for a small amount of time (typically one to two minutes) and match all the requests to available drivers in one batch \citep{Uber,Lyft,zhang2017taxi}.  However, computing this batch matching myopically without considering future requests can lead to a highly sub-optimal outcome for some subsequent riders. 
Motivated by this shortcoming, and by the possibility of using historical data to hedge against future uncertainty, we study a two-stage framework for the matching problem where the future demand uncertainty is modeled as a set of scenarios that are specified explicitly or implicitly. The goal is to compute a matching between the available drivers and current batch of riders such that the total worst-case cost of first-stage and second-stage matching is minimized. More specifically, we consider an adversarial model of uncertainty where the adversary observes the first-stage matching of our algorithms and presents a worst-case scenario from the list of specified scenarios in the second stage. 
We primarily focus on the case where the first-stage cost is the average weight of the first-stage matching, and the second-stage cost is the highest edge weight in the second-stage matching. This is motivated by the goal of computing a low-cost first-stage matching while also minimizing the waiting time for any ride in the worst-case scenario in the second stage. We also consider other metrics for the total cost and present related results.

Two-stage robust optimization is a popular model for hedging against uncertainty \citep{gupta2010thresholded,el2017beyond,el2024lp}. Several combinatorial optimization problems have been studied in this model, including Set Cover and Capacity Planning \citep{dhamdhere2005pay,feige2007robust}, Facility Location \citep{baron2011facility,el2021power} and  Network flow \citep{atamturk2007two}.
Two-stage matching problems with uncertainty, however, have not been studied extensively. They have been considered in the stochastic setting with uncertainty over the edges \citep{katriel2008commitment,escoffier2010two}, or recently in  adversarial and stochastic settings for maximizing supply efficiency or maximizing market efficiency in a two-stage matching \citep{feng2023two}.
\cite{matuschke2018maintaining} considered a two-stage version of the uni-chromatic problem (where there is no distinction between servers and clients). Their model can be seen as online min-cost matching with recourse while our model focuses on the worst-case performance with respect to an uncertainty set. 
\red{A more detailed discussion on  related work is presented in Section \ref{sec:relatedwork}.}

\red{In this paper, we contribute to the field of two-stage matching with a focus on a robust optimization approach, specifically addressing situations where an adversary is limited to selecting from a set of predefined scenarios for the second stage. Our work explores the challenges and solutions in scenarios constrained by robust optimization principles.} We study the hardness of approximation of our two-stage problem under different cost functions and present constant approximation algorithms in several settings for both the implicit and explicit models of uncertainty. Furthermore, we test our algorithms on real-life taxi data from the city of Shenzhen and show that they significantly improve upon classical greedy solutions.

\subsection{Results and Contributions}
\textbf{Problem definition.} We consider the following \textit{Two-stage Robust Matching Problem}. We are given a set of drivers $D$, a set of first-stage riders $R_1$, a universe of potential second-stage riders $R_2$ and a set of second-stage scenarios ${\mathcal S} \subseteq {P}(R_2)$\footnote{$\mathcal{P}(R_2)$ is the power set of $R_2$, the set of all subsets of $R_2$.}. We are  given a metric distance $d$ on $V = R_1 \cup R_2 \cup D$. The goal is to find a subset of drivers $D_1 \subseteq D$ ($|D_1| = |R_1|$) to match all the first-stage riders $R_1$ such that the sum of cost of first-stage matching and worst-case cost of second-stage matching (between $D\setminus D_1$ and the riders in the second-stage scenario) is minimized. More specifically, 
 \vspace{1mm}
\[ \min\limits_{D_1 \subset D}\Big\{ cost_1(D_1,R_1) + \max\limits_{S \in \mathcal{S}} cost_2(D\setminus D_1, S)\Big\}.\]

The first-stage decision is denoted $D_1$ and its cost is $cost_1(D_1,R_1)$. Similarly, $cost_2(D\setminus D_1, S)$ is the second-stage cost for scenario $S$, and $\max \{cost_2(D\setminus D_1, S) \; | \; S \in \mathcal{S}\}$ is the worst-case cost over all possible scenarios. Let $|R_1|= m$, $|R_2| = n$. We denote the objective function for a feasible solution $D_1$ by 
\[f(D_1) =  cost_1(D_1,R_1) + \max\limits_{S \in \mathcal{S}} cost_2(D\setminus D_1, S).\]

We assume that there are sufficiently many drivers to satisfy both first and second-stage demand. Given an optimal first-stage solution $D_1^*$, we denote 
\[OPT_1 = cost_1(D_1^*,R_1), \quad OPT_2 = \max \{cost_2(D\setminus D_1^*, S)\; | \; S \in \mathcal{S}\}, \quad OPT = OPT_1 + OPT_2.\]
 As we mention earlier, we primarily focus on the setting where the first-stage cost is the average weight of matching between $D_1$ and $R_1$, and the second-stage cost is the bottleneck matching cost between $D\setminus D_1$ and $S$.\footnote{The bottleneck matching problem is to find a maximum matching that minimizes the length of the longest edge.} We refer to this variant as the \textit{Two-Stage Robust Matching Bottleneck Problem} (\textbf{TSRMB}). We also consider several other cost variants and present results in Section~\ref{sec:variants}. Formally, let $M_1$ be the minimum weight perfect matching between $R_1$ and $D_1$, and given a scenario $S$, let $M_2^S$ be the bottleneck matching between the scenario $S$ and the available drivers $D\setminus D_1$, then the cost functions for the TSRMB are:
\vspace{1mm}
\[
    cost_1(D_1,R_1)  =  \frac{1}{m} \sum\limits_{(i,j) \in M_1} d(i,j), \quad \mbox{ and } \quad
    cost_2(D\setminus D_1,S)  =  \max\limits_{(i,j) \in M_2^S} d(i,j).
\]

\begin{figure}
    \centering
    \includegraphics[scale=0.4]{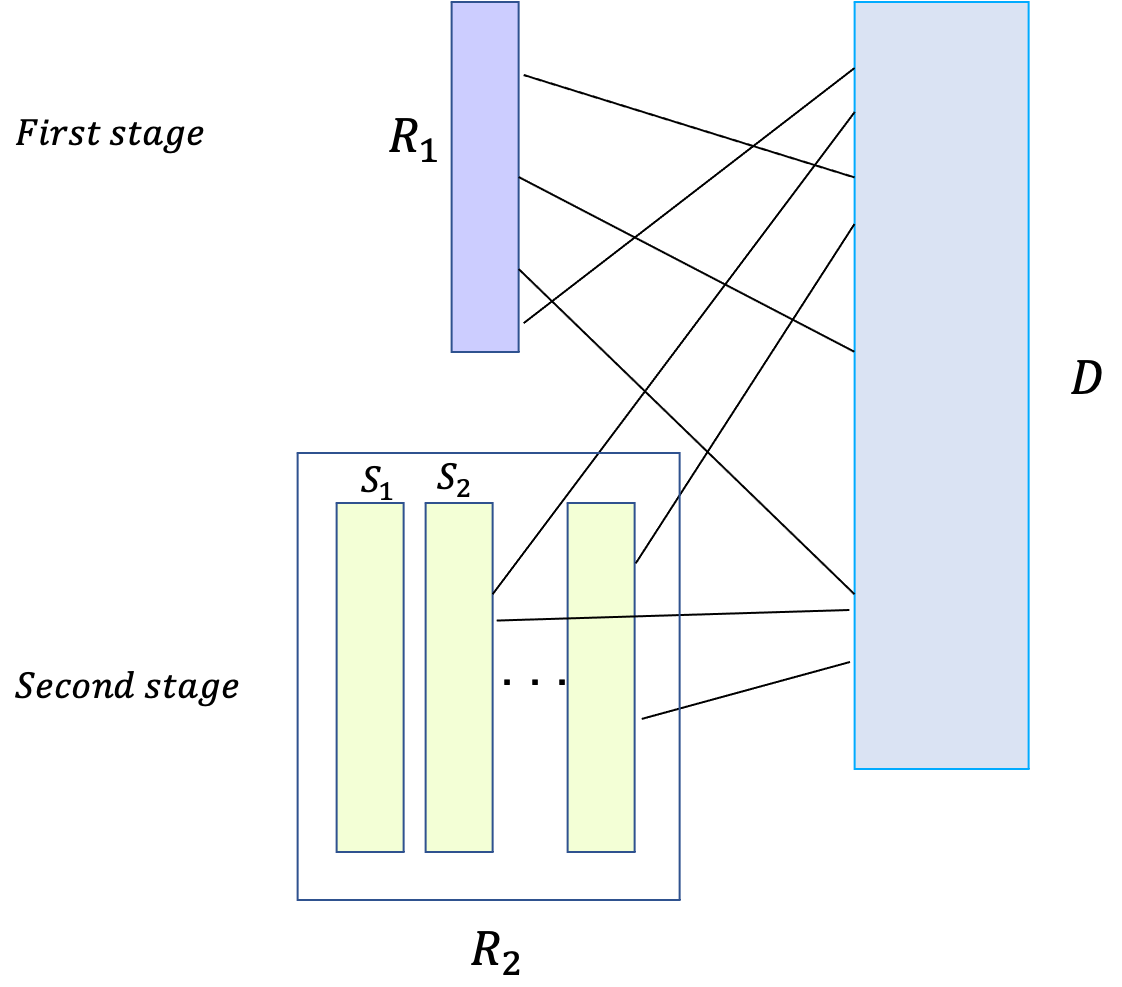}
    \caption{Bipartite graph of drivers and riders in our two-stage matching problem.}
    \label{fig:twostagemodel}
\end{figure}

 The difference between the first and second-stage metric is motivated by the fact that the \red{platform} has access to the current requests and can exactly compute the cost of matching these first-stage requests. On the other hand, to ensure the robustness of the solution over the second-stage uncertainty, we require for all second-stage assignments to have low waiting times by accounting for the maximum wait time in every scenario. Note that we choose the first-stage cost to be the average matching weight instead of the total weight for homogeneity reasons, so that first and second-stage costs have comparable magnitudes.

\vspace{2mm}
\noindent{\bf Scenario model of uncertainty}. Two common approaches to model uncertainty in robust optimization problems are to either explicitly enumerate all the realizations of the uncertain parameters or to specify them implicitly by a set of constraints over the uncertain parameters. In this paper, we consider both models.  In the \textit{explicit model}, the uncertainty set is given by a 
list of scenarios: $\mathcal{S} = \{S_1, \ldots, S_p\}$.  
In the \textit{implicit model}, we consider a universe of second-stage riders, denoted as $R_2$. The uncertainty set in this model is defined by a set of constraints that must be satisfied by any scenario. Specifically, scenarios in the implicit model correspond to the extreme points of the polytope that represents the uncertainty set. A common model of uncertainty in this context is the budget of uncertainty set, where we are given a universe of second-stage riders, $R_2$, and any subset of size at most $k$ can be considered a scenario. Thus, the set of scenarios is defined as:
\begin{equation} \label{eq:budget}
    \mathcal{S} = \{ S \subseteq R_2 \mid |S| \leq k \},
\end{equation}
where $k$ is a given parameter. Note that the total number of possible scenarios is exponential in $k$, but they are specified implicitly. This model is widely used in Robust Optimization literature, commonly referred to as the budget of uncertainty or cardinality constraint set \citep{bertsimas2004price,housni2018optimality}. The parameter \(k\), known as the budget parameter, provides the flexibility to adjust the level of conservatism in terms of the total demand. In particular, $k$ can be interpreted as the total aggregated demand in the worst case. In an adversarial model,  any \(k\) riders from the given universe might materialize. This uncertainty set also captures elementary models in robust optimization, such as box uncertainty (\(k=|R_2|\)), where the worst-case scenario encompasses the entire universe, and simplex uncertainty (\(k=1\)), where scenarios are formed by individual vertices. {\color{black} While the implicit model can, in principle, accommodate more general uncertainty sets, the current paper focuses primarily on the budgeted uncertainty model. In particular, 
Section~\ref{sec:explicit} focuses on the explicit model of uncertainty. 
Section~\ref{sec:implicit} presents one result  that applies to general uncertainty sets. Then, the rest of the paper (Section~\ref{sec:general-surplus}) focuses on the budget of uncertainty.}

\vspace{2mm}
\noindent\textbf{Hardness.} 
We show that TSRMB is NP-hard even for two scenarios. {\color{black} Furthermore, for three scenarios, TSRMB is NP-hard to approximate within a factor better than $3-\epsilon$ for any $\epsilon > 0$, unless NP=P} (see Theorem \ref{thm:hardness_robust}). For the case of implicit model of uncertainty, we show in Theorem~\ref{thm:hardness2} that under the budget of uncertainty set \eqref{eq:budget}, TSRMB is NP-hard to approximate within a factor better than $3 - \epsilon$ for any $\epsilon > 0$, unless NP=P. In general, the number of scenarios in a budget of uncertainty set can be exponentially large, but  even when the number of scenarios is small, specifically  for $k=1$, {\color{black} we show that TSRMB is NP-hard to approximate within a factor better than $3-\epsilon$ for any $\epsilon > 0$, unless NP=P}  (see Theorem \ref{thm:hardness3}). Given these  hardness results, we focus on designing approximation algorithms for the TSRMB problem.

A natural candidate to address two-stage problems is the greedy approach that minimizes only the first-stage cost without considering the uncertainty in the second stage. However, we show that this myopic approach can be bad, namely $\Omega(m) \cdot OPT$.

\vspace{2mm}
\noindent
\textbf{Approximation algorithms.}   We first consider the case of a small number of explicit scenarios. This model is motivated by the desire to use historical data from past riders as our list of explicit scenarios. Our main result in this case  is a constant approximation algorithm for TSRMB with two scenarios (Theorem \ref{2scenarios}). We  further generalize the ideas of this algorithm to show a constant approximation for TSRMB with a fixed number of scenarios (Theorem \ref{thm:pscenarios}). Our approximation does not depend on the number of first-stage riders or the size of scenarios but scales with  the number of scenarios. \red{In particular, in Theorem \ref{2scenarios}, we show  an algorithm that yields a 5-approximation to the TSRMB problem with 2 scenarios. In Theorem \ref{thm:pscenarios}, we show an algorithm that yields a $O(p^{1.59})$-approximation to the TSRMB with $p$ explicit scenarios. }



The main idea in our algorithms is to reduce the TSRMB problem with multiple scenarios to an instance  with a single \textit{representative scenario} while losing only a small factor. We then solve the single scenario instance (which can be done exactly in polynomial time) and recover a constant-factor approximation for our original problem. The challenge in constructing a single representative scenario is to find the right trade-off between  effectively capturing the demand of all second-stage riders and keeping the cost of this scenario close to the optimal cost of the original instance.

For the implicit model of uncertainty, even though the problem can be described with a polynomial size input, the  scenarios can be  exponentially many, which makes even the evaluation of the total cost of a feasible solution challenging and not necessarily achievable in polynomial time. \red{In fact, our proof of Theorem \ref{thm:hardness2}, establishes that computing the objective function for a given first-stage solution is NP-hard and can not be approximated within a factor better than $3-\epsilon$ for any $\epsilon>0$ unless P=NP.} Our analysis for the implicit model of uncertainty depends on the imbalance between supply and demand. In fact, when the number of drivers is very large compared to riders, the problem is less interesting in practice. However, the problem becomes interesting when the supply and demand are comparable. In this case, drivers might need to be shared between different scenarios. This leads us to define the notion of surplus $\ell = |D| - |R_1| - k$, \red{where $k$ is the maximum size of a second-stage scenario. The surplus represents the maximum number of drivers that we can afford not to use in a solution.}

We first consider the case where the surplus is zero in Section \ref{sec:implicit}. Our result in this section holds for any model of uncertainty, either implicit or explicit. In Theorem \ref{thm:implicitnosuplusproof}, we show that if the surplus is equal to zero, (in this case all the drivers need to be used),  using a scenario with the maximum size as a representative scenario and solving the singe scenario instance gives a 3-approximation to TSRMB. 
{\color{black} Additionally, we establish that this approximation is tight, i.e., even with zero surplus, it is NP-hard to approximate TSRMB within a factor better than $3-\epsilon$ for any $\epsilon > 0$, unless NP=P. }

The problem becomes significantly more challenging even with a small surplus. In Section \ref{sec:general-surplus}, we consider the implicit model of uncertainty described by a budget of uncertainty set. We show that under a reasonable assumption on the size of scenarios,  there is a constant approximation to the TSRMB in the regime when the surplus $\ell$ is smaller than $k$ (the maximum size of a scenario). \red{Specifically, in Theorem \ref{thm:exponential}, we show   an algorithm that yields a $17$-approximation to the TSRMB problem under a budget of uncertainty, when $\ell < k$ and $k \leq \sqrt{\frac{n}{2}}$.
This result is quite involved and requires several different new ideas and techniques to overcome the exponential number of scenarios.}
The algorithm in  Theorem \ref{thm:exponential}  finds a clustering of drivers and riders that yields a simplified instance of TSRMB which can be solved within a constant factor. We show that we can cluster the riders into a ball (riders close to each others) and a set of \textit{outliers} (riders far from each others) and apply some of our ideas from the analysis of two scenario on these two sets. Finally, since the  evaluation problem is challenging because of the  exponentially many scenarios, our algorithm constructs a set of a polynomial number of proxy scenarios on which we can evaluate any feasible solution within a constant approximation.

\color{black}

We also address the case of arbitrary surplus where each scenario has only a single rider (i.e., the uncertainty set is a budget of uncertainty with $k=1$). While this case has only polynomially many scenarios, each of size 1, it is still NP-hard to approximate within a factor better than $3-\epsilon$, as shown in Theorem~\ref{thm:hardness3}. We present an algorithm that provides a $3$-approximation for this case (Theorem~\ref{thm:koneimplicit}). Therefore, our result matches the hardness of approximation and closes the gap for the case of a budget of uncertainty with $k=1$.

\color{black}

\red{We summarize our approximation guarantees and hardness results in Table~\ref{table:results}. The first column of the table specifies the uncertainty model. The second column presents the guarantees of our approximation algorithms. The third column provides our hardness lower bounds. Note that for some models of uncertainty, the approximation guarantee is not tight, leaving open questions for further research to close the gap.}

\begin{table}[h]
\centering
\begin{tabular}{|c|c|c|}
\hline
\textbf{Model of Uncertainty} & \textbf{Approximation Guarantee} & \textbf{Lower Bound} \\ \hline
1 scenario & Solvable in polynomial time & -- \\ \hline
2 scenarios & 5 & NP-hard \\ \hline
$p$ scenarios & $O(p^{1.59})$ & \color{black} $3-\epsilon$ \\ \hline
General model without surplus & 3 & $3-\epsilon$ \\ \hline
Budget of uncertainty with small surplus   & 17 & $3-\epsilon$ \\ \hline
\color{black}Budget of uncertainty with \( k = 1 \) & \color{black} 3 & \color{black} $3-\epsilon$\\ \hline
\end{tabular}
\caption{Summary of our approximation guarantees and hardness lower bounds for TSRMB. }
\label{table:results}
\end{table}

\noindent
\textbf{Experimental study.}
{\color{black}
We implement our algorithms and test them on real-life taxi data from the city of Shenzhen \citep{cheng2019stl}. Our experimental results show that our two-scenario algorithm significantly improves upon the greedy algorithm in both in-sample and out-of-sample settings. We also benchmark against the optimal solution of TSRMB computed using an integer program and against the offline optimal out-of sample solution and show that our algorithm performs competitively.
Moreover, the algorithm achieves a lower second-stage bottleneck cost without increasing the overall matching weight, indicating reduced maximum wait times without sacrificing average performance. 
We further evaluate the algorithm over multiple stages to mimic a rolling-horizon deployment and find that it consistently outperforms the greedy baseline by improving the maximum wait time as well as modest gains in average wait time.  See Section~\ref{sec:experiments} for more details.}

\vspace{3mm}
\noindent
\textbf{Extensions and variants.} While the majority of the paper studies the TSRBM problem, we also initiate the study of several other cost functions
 for two-stage matching problems both for adversarial and stochastic second-stage scenarios.  In particular, we consider the Two-Stage Stochastic Matching Bottleneck (TSSMB), where the first-stage cost is the average weight of the matching, and the second-stage cost is the expectation of the bottleneck matching cost over all scenarios.  We also consider the Two-Stage Robust Matching problem (TSRM), where the first and second-stage costs correspond both to the total weight of the matchings.  Finally, we consider the Two-Stage Robust Bottleneck Bottleneck problem (TSRBB), where the first and second-stage costs both correspond to the bottleneck matching cost. We study the hardness of these variants, and make a first attempt to present approximation algorithms under specific settings.


\vspace{2mm}
{\color{black}
\noindent
\textbf{Other adversarial models.} 
The TSRMB model studied in this paper assumes an \emph{adaptive} adversary, meaning the adversary observes the realization of the algorithm’s first-stage decision \( D_1 \), then selects a second-stage scenario \( S \in \mathcal{S} \) that maximizes the total cost. This models the worst-case, fully informed adversary and aligns with standard formulations in robust optimization.
A natural variation of this model is one where the adversary is \emph{oblivious}, that is, the adversary knows the algorithm used to generate the first-stage decision but not the specific realization of \( D_1 \). In this setting, the adversary must select a scenario \( S \in \mathcal{S} \) before observing the outcome of any randomness in the algorithm. If the algorithm is deterministic, the oblivious adversary can simulate it and compute \( D_1 \), making the adaptive and oblivious models equivalent. However, when the algorithm is randomized and selects \( D_1 \sim \mathcal{D}_1 \) from a distribution over feasible decisions, the models differ: the adaptive adversary reacts to the actual realization of \( D_1 \), while the oblivious adversary can only react to its distribution.
To illustrate, consider the expected second-stage cost under an oblivious adversary: \( \max_{S \in \mathcal{S}} \mathbb{E}_{D_1 \sim \mathcal{D}_1}[cost_2( D \setminus D_1, S)] \). In contrast, under an adaptive adversary, the expected second-stage cost is \( \mathbb{E}_{D_1 \sim \mathcal{D}_1}[\max_{S \in \mathcal{S}} cost_2(D \setminus D_1, S)] \). Since the maximum of expectations is generally smaller than the expectation of maxima, the oblivious setting can yield strictly lower expected cost. On the other hand, randomization does not help in the adaptive setting: the adversary observes the realization of \( D_1 \) and can respond in the worst possible way, rendering randomization ineffective.

This distinction between models reveals an interesting game-theoretic structure in the oblivious case. The problem can be viewed as a two-player minimax game, where the algorithm selects a distribution over first-stage decisions to minimize expected cost, and the adversary selects a scenario to maximize it. By Von Neumann’s minimax theorem, a game value exists, and the optimal randomized strategy can be approximated efficiently in certain cases. In particular, when the number of scenarios is polynomial, the \emph{Multiplicative Weights Update Method}~\citep{arora2012multiplicative,plotkin1995fast} provides a principled framework to nearly-approximate the optimal randomized strategy using repeated best-response dynamics.
While we do not explore this variant further in this paper, it is a meaningful and computationally tractable direction. Our focus in this work is  the adaptive adversary setting, which represents the standard worst-case model in robust optimization and gives rise to the complexity and algorithmic challenges we address.
}

Finally, we wish to emphasize that the TSRMB problem assumes knowledge of the uncertainty set $\cal S$. If we consider a model operating under complete adversarial conditions (i.e., without a predefined uncertainty set defining scenarios), the problem lacks a bounded approximation. In such scenarios, an adversary could select a second stage  that significantly inflates the full cost. We have included a thorough discussion in Appendix \ref{newyear}, comparing our robust model to a fully adversarial model, and provided an illustrative example to elucidate this point. {\color{black} Unlike the rest of the paper, which uses approximation guarantees relative to an optimal robust solution, this fully adversarial model is analyzed using competitive analysis, where the performance is measured against the hindsight optimum. }

\subsection{Outline} 
The paper is organized as follows. We review relevant literature in Section \ref{sec:relatedwork}. In Section \ref{sec:preliminaries}, we introduce some preliminary results on the hardness of TSRMB. We study the performance of the greedy approach and finally present a subroutine to solve the deterministic TSRMB with one scenario. In Section \ref{sec:explicit}, we study TSRMB with explicit scenarios. 
In Section \ref{sec:implicit}, we consider the case of general uncertainty set wihout surplus and we study the case of
implicit scenarios in Section~\ref{sec:general-surplus}.   We present our numerical experiments on a set of real-life taxi data from the city of Shenzhen in Section \ref{sec:experiments}. Finally, Section \ref{sec:variants} explores other variants of the two-stage robust matching problem with different cost functions.

\section{Related Work}\label{sec:relatedwork}

\textit{Online bipartite matching}. Finding a maximum cardinality bipartite matching
is one of the classical problems in algorithmic graph theory and  combinatorial optimization as it arises naturally in several applications such as resource allocation, scheduling, and online advertising. The online version of this
problem has received a considerable amount of attention over the years  \citep{mehta2013online}. In this setting, we are given a known set of \textit{servers}
while a set of \textit{clients} arrive online and upon arrival, each client can be
matched to a server irrevocably. The online matching problem was first studied by \cite{karp1990optimal} in the adversarial model where the graph
is unknown; when a client arrives it reveals its incident edges. \cite{karp1990optimal} and \cite{birnbaum2008line}
proved that the simple randomized RANKING algorithm achieves $(1-1/e)$ competitive ratio and this
factor is the best possible performance. Since then, many online variants have been studied in great depth (see survey of \cite{mehta2013online}).
This includes problems the study of problems like AdWords by \cite{buchbinder2007online, devanur2009adwords} and \cite{mehta2007adwords}, vertex-weighted matching by \cite{aggarwal2011online} and \cite{devanur2013randomized}, edge-weighted matching by \cite{haeupler2011online} and \cite{korula2009algorithms},
stochastic matching by \cite{feldman2009online,manshadi2012online, mehta2014online} and \cite{feldman2016online}, random vertex arrival by \cite{goel2008online, karande2011online} and \cite{jaillet2014online}, and batch arrivals by \cite{lee2017maximum, zhang2017taxi} and \cite{feng2020batching}. 

\vspace{2mm}

\noindent \textit{Online minimum weight matching.} In the \textit{online bipartite metric matching problem}, servers
and clients correspond to points from a metric space. Upon arrival, each client must be
matched to a server irrevocably, at a cost equal to their distance. The objective is to find the minimum weight maximum cardinality matching. For general metric spaces, \cite{khuller1994line} and \cite{kalyanasundaram1993online}
proved that there is a tight bound of $(2n - 1)$ on the competitiveness factor of deterministic online
algorithms, where $n$ is the number of servers. In the random arrival model, a natural
question is whether randomization could help obtain an exponential improvement
for general metric spaces.
\cite{meyerson2006randomized} and \cite{bansal2007log} provided poly-logarithmic competitive randomized algorithms for the problem. Recently, \cite{raghvendra2016robust} presented a
$O(\log{n})$-competitive algorithm in the random arrival model.

\vspace{2mm}

\noindent \textit{Two-stage stochastic combinatorial optimization.} Within two-stage stochastic optimization, matching has been studied under various models and different objectives. \cite{kong2006factor} introduce the stochastic two-stage maximum matching problem. They prove that the problem is NP-hard when the number of scenarios is an input of the problem and provide $1/2$-approximation algorithm. \cite{escoffier2010two} further study this problem, strengthen the hardness results, and slightly improve the approximation ratio. \cite{katriel2008commitment} study two stochastic minimum weight maximum matching problems in bipartite graphs. In their two variants, the uncertainty is respectively on the second-stage cost of the edges and on the set of vertices to be matched. \cite{feng2020batching} study $K$-stage variants of vertex weighted
bipartite b-matching and AdWords problems, where online vertices arrive in $K$ batches. More recently,  \cite{feng2021two} initiate the study and present online competitive algorithms for vertex-weighted two-stage stochastic matching as well as two-stage joint matching and pricing \red{with application to ride hailing \cite{feng2023two}}.

\vspace{2mm}

\noindent \textit{Two-stage robust combinatorial optimization.} Within two-stage robust optimization, matchings have not been studied extensively.
\cite{matuschke2018maintaining} proposed a two-stage
robust model for minimum weight matching with recourse. 
In the first stage, a perfect matching between
$2n$ given nodes must be selected; in the second stage $2k$ new nodes are introduced. The goal is to produce
$\alpha$-competitive matchings at the end of both stages, and such that the number of edges removed from the first
stage matching is at most $\beta k$. Our model for TSRMB is different in 3 main aspects: 1) In our model, the second-stage vertices come from an uncertainty set whereas in their model the only information given is the number of second-stage vertices. 2) We do not allow any recourse and our first-stage matching is irrevocable. 3) Our second-stage cost is the bottleneck weight instead of the total weight. In general, a bottleneck optimization problem on a graph with edge costs is the problem of finding a subgraph of a certain kind that minimizes the maximum edge cost in the subgraph. The bottleneck objective contrasts with the more common objective of minimizing the sum of edge costs. Several Bottleneck problems have been considered, e.g. Shortest Path Problem by \cite{kaibel2006bottleneck} and \cite{bose2004approximatin}, Spanning Tree and Maximum Cardinality Matching by \cite{gabow1988algorithms}, and TSP problems by \cite{garfinkel1978bottleneck} (see \cite{hochbaum1986unified} for a compilation of graph bottleneck problems).

\red{In terms of application, our work relates to the broader literature that focuses on developing optimization models and algorithms for ride-sharing applications. Notable references in this area include recent works by \cite{bertsimas2019online,feng2021two}, among others.}

\section{Preliminaries}\label{sec:preliminaries}
In this section, we study the hardness of approximation for TSRMB. We also examine the challenges with  the natural greedy approach for solving TSRMB. We finally present a subroutine to solve the single scenario case that we will use later on in our general algorithms.

\subsection{NP-hardness}\label{sec:hardness}

We demonstrate that the TSRMB problem is NP-hard under both implicit and explicit models of uncertainty. In the explicit model, Theorem \ref{thm:hardness_robust} establishes the NP-hardness of TSRMB even for two scenarios. Furthermore, {\color{black}we prove that approximating TSRMB within a factor better than $3-\epsilon$, for any $\epsilon>0$, is NP-hard for three scenarios.}

In the implicit model of uncertainty, Theorem \ref{thm:hardness2} asserts the NP-hardness of approximating TSRMB within a factor better than $3-\epsilon$. This result holds even   when the uncertainty set is defined by the budget of uncertainty \eqref{eq:budget}. 
{\color{black}
Furthermore, we show that TSRMB under the budgeted uncertainty set \eqref{eq:budget} is $\Sigma_2^P$-hard. Specifically, we prove that deciding whether the optimal value of TSRMB is equal to a given threshold, is $\Sigma_2^P$-complete. This result is significant because $\Sigma_2^P$-completeness implies that, unless the polynomial hierarchy collapses, there does not exist a compact Integer Programming (IP) formulation for TSRMB with budgeted uncertainty~\citep{grune2023completeness}.
}
It is important to note that the number of scenarios in a budget of uncertainty set, with a general parameter $k$, can be exponentially large. We also show in Theorem \ref{thm:hardness3} that, even for $k=1$ (where number of scenarios is polynomial), {\color{black}approximating TSRMB within a factor better than $3-\epsilon$ is NP-hard.} 

The proof of Theorem \ref{thm:hardness_robust} employs a reduction from the 3-Dimensional Matching Problem. The first result in Theorem \ref{thm:hardness2} uses a reduction from the Clique Problem and {\color{black} the $\Sigma_2^P$-completeness result uses a reduction from the Clique Interdiction Problem.} The proof of Theorem \ref{thm:hardness3}  uses a reduction from the Set Cover Problem. Detailed proofs of these theorems are provided in Appendix~\ref{appendix:hardness}.

\red{Note that in the explicit model with a polynomial number of scenarios, it is clear that the problem is in NP. However, in the implicit model, even though the problem can be described with a polynomial size input, there could be exponentially many scenarios and, in general, we cannot compute the total cost function in polynomial time. In fact,  the proof of Theorem \ref{thm:hardness2}, establishes that computing the objective function for a given first-stage solution is NP-hard and can not be approximated within a factor better than $3-\epsilon$ for any $\epsilon>0$ unless P=NP.}


\begin{theorem}\label{thm:hardness_robust}
In the explicit model of uncertainty, TSRMB is NP-hard even when the number of scenarios is  two. {\color{black}Furthermore, when the number of scenarios is three, there is no $(3-\epsilon)$-approximation algorithm for any fixed $\epsilon > 0$, unless P = NP.}
\end{theorem}

\begin{theorem}
\label{thm:hardness2}
In the implicit model of uncertainty, under the budget of uncertainty set \eqref{eq:budget},  there is no $(3-\epsilon)$-approximation algorithm for TSRMB for any fixed $\epsilon > 0$, unless $P = NP$. {\color{black}Furthermore, TSRMB  under the budget of uncertainty set \eqref{eq:budget} is $\Sigma_2^P$-hard.}
\end{theorem}

\color{black}
\begin{theorem}
\label{thm:hardness3}
In the implicit model of uncertainty, under the budget of uncertainty set \eqref{eq:budget} even when $k=1$, there is no $(3-\epsilon)$-approximation algorithm for TSRMB for any fixed $\epsilon > 0$, unless $P = NP$.
\end{theorem}

\color{black}
\subsection{Greedy Approach}
A natural greedy approach is to choose the optimal matching for the first-stage riders $R_1$ without considering the uncertainty in the second stage in any way.
 We show via a counterexample that this greedy approach could lead to a bad solution for TSRMB with a total cost that scales linearly with $m$ (cardinality of  $R_1$) while $OPT$ is a constant, even when there is only one scenario.

\noindent
{\bf Counterexample.}
Consider the line example depicted in Figure \ref{fig:greedy}, where we have $m$ first-stage riders and $m+1$ drivers that alternate on a line with distances $1$ and $1-\epsilon$. There is only one second-stage rider at the right endpoint of the line. A greedy matching would minimize the first-stage cost by matching the first-stage riders using the dashed edges, with an average weight of $1-\epsilon$. When the second-stage scenario is revealed, the rider can only be matched with the farthest driver for a cost of $1+ (2-\epsilon)m$. Therefore the total cost of the greedy approach is $(2-\epsilon)(m+1)$, while the optimal cost is clearly equal to $2$. This example shows that the cost of the greedy algorithm for TSRMB could be far away from the optimal cost with an approximation ratio that scales with the dimension of the problem. The same observation generalizes to any number of scenarios by simply duplicating the second-stage rider. Therefore any attempt to have a good approximation to the TSRMB needs to consider the second-stage riders. In particular, we have the following lemma.

\begin{figure}[!t]
    \centering
    \includegraphics[scale=0.45]{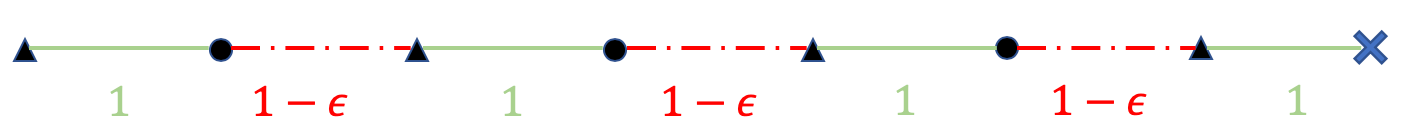}
    \caption{Example instance on the line. Riders in the first stage are depicted by black dots and
drivers are indicated as black triangles. The second-stage rider is depicted as a blue cross. First and second-stage optimum are depicted
by solid green edges. 
}
    \label{fig:greedy}
\end{figure}



\begin{lemma}
The cost of the Greedy algorithm can be $\Omega(m)\cdot~OPT$.
\end{lemma}

{\color{black}
\subsection{Trade-off between the First and Second Stages}}

Another natural question in the context of our two-stage model is the necessary trade-off between the two stages. In particular, consider the single-stage matching problem that ignores the second stage; i.e., the goal is to match the riders $R_1$ with drivers from $D$ with the minimum cost. The problem is given by
\begin{equation}
    \label{first-stage}
    \min_{D_1 \subset D} \left\{ \text{cost}_1(D_1, R_1) \right\}.
\end{equation}
The question is whether there exists an approximately optimal solution for Problem \eqref{first-stage} that is also approximately optimal for our two-stage robust matching problem. We show that the answer to this question is negative by presenting an instance where there exists no solution $D_1$ that is near-optimal for both problems. Any near-optimal solution for the single-stage problem, i.e., within a constant factor from the optimal one, can be arbitrarily bad for the two-stage model. In particular, first-stage optimality has to be sacrificed entirely in order to get a robust near-optimal solution for the two-stage model in the worst case. Consider the following example, which is a simple modification of the previous example in Figure \ref{fig:greedy}.

\noindent
\red{
\textbf{Example.}
Consider the line example depicted in Figure \ref{fig:tradeoff}, where we have $m$ first-stage riders and $m+1$ drivers that alternate on a line with distances $1$ and $\epsilon$. There is only one second-stage rider at the right endpoint of the line. An optimal solution for the single-stage Problem \eqref{first-stage} would minimize the cost by matching the first-stage riders using the dashed edges, each dashed edge has weight $\epsilon$, so the average weight is $\epsilon$. Therefore, the optimal objective value of \eqref{first-stage} is equal to $\epsilon$. Note that any other feasible solution for \eqref{first-stage} needs to use at least one solid green edge, and therefore one of the edges has a weight of at least $1$, resulting in an average weight in this matching of at least $1/m$. By choosing $\epsilon$ arbitrarily small, the gap between the objective value of any feasible solution and the optimal solution for \eqref{first-stage} can be arbitrarily bad. It is sufficient to choose $\epsilon=1/m^2$; in that case, the gap is $\Omega(m)$. Hence, in this example there is {\color{black}no near-optimal} solution within a constant factor from the optimal one.
Note that using the optimal solution of \eqref{first-stage} as the first-stage solution in the two-stage model, the second-stage rider can only be matched with the farthest driver for a cost of $1 + (1+\epsilon)m$. Therefore, {\color{black}the total cost of this feasible solution in the two-stage model is $(1+\epsilon)(m+1)$, while the optimal total cost of the two-stage model is clearly equal to $2$.} We conclude through this example that an optimal solution for the single first-stage problem can be arbitrarily bad for the two-stage robust problem, and in order to get an optimal or {\color{black}even near-optimal} solution for the the two-stage model, the {\color{black}entire} first-stage optimality has to be sacrificed.
\begin{figure}[h]
    \centering
    \includegraphics[scale=0.45]{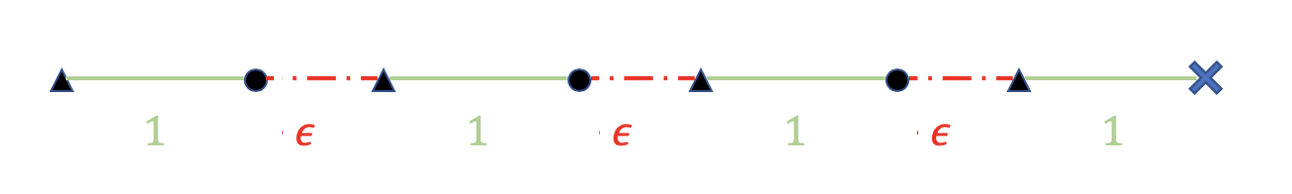}
    \caption{Example instance on the line. Riders in the first stage are depicted by black dots, and
    drivers are indicated as black triangles. The second-stage rider is depicted as a blue cross. First and second-stage optima are depicted
    by solid green edges.}
    \label{fig:tradeoff}
\end{figure}
}

\subsection{Single Scenario}
The {\em deterministic} version of the TSRMB problem, i.e., when there is only a single scenario in the second stage, can be solved exactly in polynomial time.  This is a simple preliminary result which we  need for the general case. Denote $S$ a single second-stage scenario. The  instance $(R_1,S,D)$ of TSRMB is then simply given by
 $$ \min\limits_{D_1 \subset D}\Big\{ cost_1(D_1,R_1) +  cost_2(D\setminus D_1, S)\Big\}. $$
Since the second-stage problem is a bottleneck problem, the value of the optimal second-stage cost $w$ is one of the edge weights between $D$ and $S$. We iterate over all possible values of $w$ (at most $|S|\cdot|D|$ values), delete all edges between $R_2$ and $D$ with weights strictly higher than $w$ and  set the weight of the remaining edges between $S$ and $D$ to zero. This reduces the problem to finding a minimum weight maximum cardinality matching. 
Below, are presented the details of our algorithm. 
We refer to it  as \textit{TSRMB-1-Scenario} (or Algorithm \ref{alg:singlescenario}) in the rest of this paper.


We define the bottleneck graph of $w$ to be $BOTTLENECKG(w) = (R_1\cup S \cup D, E_1 \cup E_2)$ where $E_2 = \{(i, j) \in D \times S, \ d(i,j) \leq w\}$ and $E_1=\{(i, j) \in D \times R_1\}$. Furthermore, we assume that there are $q$ edges $\{e_1, \ldots, e_q\}$ between $S$ and $D$ with weights  $w_1 \leq w_2 \leq \ldots \leq w_q.$

\begin{algorithm}
	\caption{TSRMB-1-Scenario($R_1, S, D)$}
	\label{alg:singlescenario}
    \begin{algorithmic}[1]
    \REQUIRE{First-stage riders $R_1$, scenario $S$ and drivers $D$.}
    \ENSURE{First-stage decision $D_1$.}
	\FOR{$i \in \{1, \ldots, q\}$}
	\STATE{$G_i$ :=  $BOTTLENECKG(w_i)$.}
	\STATE{Set all weights between $D$ and $S$ in $G_i$ to be 0.}
	\STATE{$M_i:=$ minimum weight maximum cardinality matching on $G_i$.}
	\IF{$R_1\cup S$ is not completely matched in $M_i$}
	    \STATE{$D_1^i:=\emptyset$}
	\ELSE
	    \STATE{$D_1^i:=$ first-stage drivers in $M_i$.}
	\ENDIF
	\ENDFOR
	\RETURN $D_1$ = $\argmin\limits_{ D_1^i :    1 \leq i \leq q} \Big\{ cost_1(D_1^i,R_1) +  cost_2(D\setminus D_1^i, S)\Big\}$.
	\end{algorithmic} 
\end{algorithm}
Note that when $D_1^i =\emptyset$, we adopt the convention that $cost_1(D_1^i,R_1)= \infty$. So, the $\argmin$ in the last step of Algorithm \ref{alg:singlescenario} is only taken over values of $i$ for which  $D_1^i \neq \emptyset$.

\begin{lemma}\label{lemma:singlescenario}
Algorithm \ref{alg:singlescenario} provides an exact solution to TSRMB with a single scenario.
\end{lemma}
\proof{\textit{Proof}.}
 Let $OPT_1$ and $OPT_2$ be the first and second-stage cost of an optimal solution, and $i \in \{1,\ldots,q\}$ such that $w_i = OPT_2$. In this case, $G_i$ contains all the edges of this optimal solution. By setting all the edges in $E_2$ to 0, we are able to compute a minimum weight maximum cardinality matching between $R_1 \cup S$ and $D$ that matches both $R_1$ and $S$ and minimizes the weight of the edges matching $R_1$. The first-stage cost of this matching is at most $OPT_1$, the second-stage cost is clearly at most $OPT_2$ because we only allowed edges with weight at most $OPT_2$ in $G_i$.
 \hfill \Halmos
 \endproof


\section{Explicit Scenarios}\label{sec:explicit}
In this section, we consider TSRMB under the explicit model of uncertainty where we have an explicit list of scenarios for the second-stage and we optimize over the worst-case scenario realization. We first present a constant factor approximation for TSRMB for the case of two scenarios. We then extend our result to the case of any fixed number of scenarios. However, the approximation factor scales with the number of scenarios $p$ as $O(p^{1.59})$. The idea of our algorithm is to reduce the instance of TSRMB with $p$ scenarios to an instance with only a single representative scenario by losing a small factor and then use 
Algorithm \ref{alg:singlescenario} to solve the single scenario instance. To illustrate the core ideas of our algorithm, we focus on the case of two scenarios first and  then extend it to a constant number of scenarios.
 
 \subsection{Two Scenarios}

Consider two scenarios $\mathcal{S} = \{ S_1, S_2\}$. 
First, we can assume without loss of generality that we know the exact value of $OPT_2$ which corresponds to one of the edges connecting second-stage riders $R_2$ to drivers $D$ (we can iterate over all the weights of second-stage edges).
 We construct  a representative scenario that serves as a proxy for $S_1$ and $S_2$ as follows. In the second stage, if a pair of riders $i \in S_1$ and $j\in S_2$ are served by the same driver in the optimal solution, then they should be close to each other. Therefore, we can consider a single representative rider for each such pair. While it is not easy to guess all such pairs, we can approximately compute the representative riders by solving a maximum matching on $S_1 \cup S_2$ with edges at most $2 OPT_2$. More formally,  let $G_I$ be the induced bipartite subgraph of $G$ on $S_1 \cup S_2$ containing only edges between $S_1$ and $S_2$ with weight less than or equal to $2OPT_2$. We compute a maximum cardinality matching $M$ between $S_1$ and $S_2$ in $G_I$, and construct a representative scenario containing $S_1$ as well as the unmatched riders of $S_2$. We solve the single scenario problem on this representative scenario using Algorithm \ref{alg:singlescenario} and return its optimal first-stage solution. We show in Theorem   \ref{2scenarios} that this solution leads to 5-approximation for our problem. Our algorithm is described below.

{\small \begin{algorithm}
	\caption{Two explicit scenarios.} 
	\label{alg:2scenarios}
	\begin{algorithmic}[1]
	\REQUIRE{First-stage riders $R_1$, two scenarios $S_1$ and $S_2$, drivers $D$ and value of $OPT_2$.}
	\ENSURE{First-stage decision $D_1$.}
	\STATE{Let $G_I$ be the induced subgraph of $G$ on $S_1 \cup S_2$ with only the edges between $S_1$ and $S_2$ of weights at most  $2OPT_2$ .}
	\STATE{Set $M:=$ maximum cardinality matching between $S_1$ and $S_2$ in ${G_I}$.}
	\STATE{Set $S_2^{Match}:= \{ r \in S_2 \ | \ \exists \ s \in S_1 \ \mbox{ s.t }(s,r) \in M \}$ and $S_2^{Unmatch} = S_2 \setminus S_2^{Match}$.}
	\RETURN $D_1:=$ TSRMB-1-Scenario($R_1, S_1 \cup S_2^{Unmatch}, D)$.
	\end{algorithmic} 
\end{algorithm}
}

\begin{theorem}\label{2scenarios}
Algorithm \ref{alg:2scenarios} yields a solution with total cost at most $OPT_1 + 5OPT_2$ for TSRMB with 2 scenarios.
\end{theorem}

Recall that $OPT_1$ and $OPT_2$ are respectively the first-stage and second-stage cost of an optimal solution for our TSRMB problem with two scenarios. The proof of Theorem \ref{2scenarios} relies on the following structural lemma where we show that the set  $D_1$ returned by Algorithm \ref{alg:2scenarios} yields a total cost at most $(OPT_1 + 3OPT_2)$ when evaluated only on the single representative scenario $S_1 \cup S_2^{Unmatch}$.

\begin{lemma}\label{lemma:2scenarios}
Let $D_1$ be the set of first-stage drivers returned by Algorithm \ref{alg:2scenarios}. Then,
{\small $$cost_1(D_1,R_1) + cost_2(D\setminus D_1, S_1 \cup S_2^{Unmatch}) \leq OPT_1 + 3OPT_2.$$
}
\end{lemma}
\proof{\textit{Proof}.}
To prove the lemma, it is sufficient to show the existence of a matching $M_a$ between $R_1 \cup S_1 \cup S_2^{Unmatch}$ and $D$ with a total cost at most $OPT_1 + 3OPT_2$. This would imply that the optimal solution $D_1$ of TSRMB-1-Scenario($R_1, S_1 \cup S_2^{Unmatch}, D)$ has a total cost at most $OPT_1 + 3OPT_2$ and concludes the proof of the lemma. We show the existence of $M_a$ by construction.
\begin{itemize}
    \item {\bf Step 1.} We first match $R_1$ with their mates in the optimal solution of TSRMB. Hence, the first-stage cost of our constructed matching $M_a$ is $OPT_1$. 
    
    \item {\bf Step 2.} Now, we focus on $S_2^{Unmatch}$. Let $S_2^{Unmatch} = S_{12} \cup S_{22}$ be a partition of $S_2^{Unmatch}$ where $S_{12}$ contains riders with a distance at most $2OPT_2$ from $S_1$ and $S_{22}$ contains riders with a distance strictly bigger than $2OPT_2$ from $S_1$, where the distance from a set is the minimum distance to any element of the set.  A rider in $S_{22}$ cannot share any driver with a rider from $S_1$ in the optimal solution of TSRMB, because otherwise, the distance between these riders will be at most $2OPT_2$ by using the triangle inequality. Therefore we can match $S_{22}$ to their mates in the optimal solution and add them to $M_a$, without using the optimal drivers of $S_1$. We pay at most $OPT_2$ for matching $S_{22}$.
    
    \item {\bf Step 3. }We still need to simultaneously match riders in $S_1$ and $S_{12}$ to finish the construction of $M_a$. Notice that some riders in $S_{12}$ might share their optimal drivers with riders in $S_1$. We can assume without loss of generality that all riders in $S_{12}$ share their optimal drivers with $S_1$ (otherwise we can match them to their optimal drivers without affecting $S_1$). Denote $S_{12} = \{r_1, \ldots, r_q\}$ and $S_{1} = \{s_1, \ldots, s_k\}$. For each $i\in [q]$ let's say $s_i \in S_1$ is the rider that shares its optimal driver with $r_i$. We show that $q \leq  |M|.$  In fact, every rider in $S_{12}$ shares its optimal driver with a different rider in $S_1$, and is  within a distance $2OPT_2$ from $S_1$. But since $S_{12}$ is not covered by the maximum cardinality matching $M$, this implies by the maximality of $M$ that there are $q$ other riders from $S_2^{Match}$ that are covered by $M$. Hence $q \leq |M|$. Finally, let $\{t_1, \ldots, t_q\} \subset S_2^{Match}$ be the mates of $ \{s_1, \ldots, s_q\}$ in $M$, i.e., $(s_i,t_i) \in M$ for all $i \in [q]$. Recall that $d(s_i,t_i)\leq 2OPT_2$ for all $i \in [q]$. In what follows, we describe how to match  $S_{12}$ and $S_1$:

    \begin{itemize}
    \item For $i \in [q]$, we match $r_i$ to its optimal driver and $s_i$ to the optimal driver of $t_i$. This is possible because the optimal driver of $t_i$ cannot be the same as the optimal driver of $r_i$ since both $r_i$ and $t_i$ are part of the same scenario $S_2$. Therefore, we pay a cost $OPT_2$ for the riders $r_i$ and a cost $3 OPT_2$ (follows from the triangle inequality) for the riders $s_i$ where $i \in [q]$.
    
    \item We still need to match $\{s_{q+1}, \ldots, s_k\}$. Consider a rider $s_j$ with $j \in \{q+1, \ldots, k\}$. If the optimal driver of $s_{j}$ is not shared with any $t_i \in \{t_1,\ldots, t_q\}$, then this optimal driver is still available and can be matched to $s_{j}$ with a cost at most $OPT_2$. If the optimal driver of $s_{j}$ is shared with some $t_i \in \{t_1, \ldots t_q\}$, then $s_{j}$ is also covered by $M$. Otherwise $M$ can be augmented by deleting $(s_i,t_i)$ and adding $(r_i, s_i)$ and $(s_{j}, t_i)$. Therefore $s_{j}$ is covered by $M$ and has a mate $\tilde{t}_{j} \in S_2^{Match} \setminus \{t_1, \ldots, t_q\}$. Furthermore, the driver assigned to $\tilde{t}_j$ is still available. We can then match $s_{j}$ to the optimal driver of $\tilde{t}_{j}$. Similarly if the optimal driver of some $s_{j'} \in \{s_{q+1}, \ldots, s_k\} \setminus \{s_{j}\}$ is shared with $\tilde{t}_{j}$, then $s_{j'}$ is covered by $M$. Otherwise $(r_i, s_i, t_i, s_{j}, \tilde{t}_{j}, s_{j'})$ is an augmenting path in $M$. Therefore $s_{j'}$ has a mate in $M$ and we can match $s_{j'}$ to the optimal driver of its mate. We keep extending these augmenting paths until all the riders in $\{s_{q+1}, \ldots, s_k\}$ are matched. Furthermore, the augmenting paths $(r_i, s_i, t_i, s_{j}, \tilde{t}_{j}, s_{j'} \ldots)$  starting from two different riders $r_i \in S_{12}$ are vertex disjoint. This ensures that every driver is used at most once. Again, by the triangle inequality, the edges that match $\{s_{q+1}, \ldots, s_k\}$ in our solution have weights less then $3 OPT_2$.
    \end{itemize}
\end{itemize}

Putting it all together, we have constructed a matching $M_a$ where the first-stage cost is exactly $OPT_1$ and the second-stage cost is at most $3 OPT_2$ since the edges used for matching $S_1 \cup S_2^{Unmatch}$ in $M_a$ have a weight at most $3OPT_2$. Therefore, the total cost of $M_a$ is at most $OPT_1 + 3OPT_2$.

\hfill \Halmos 
\endproof

\vspace{3mm}

\proof{\textit{Proof of Theorem \ref{2scenarios}}.}
Let $D_1$ be the set of drivers returned by Algorithm \ref{alg:2scenarios}. 
Lemma \ref{lemma:2scenarios} implies
\begin{equation}\label{eq:S1}
 cost_1(D_1,R_1)  + cost_2(D\setminus D_1, S_1)  \leq  \ OPT_1 + 3OPT_2   
\end{equation}
and 
\begin{equation*}
    cost_1(D_1,R_1) + cost_2(D\setminus D_1, S_2^{Unmatch})  \leq  \ OPT_1 + 3OPT_2.
\end{equation*}
We have $S_2=S_2^{Match} \cup S_2^{Unmatch} $. If the scenario $S_2$ is realized, we  use the drivers that were assigned to $S_1$ in the matching constructed in Lemma \ref{lemma:2scenarios} to match $S_2^{Match}$. This is possible with edges of weights at most $cost_2(D\setminus D_1, S_1) + 2OPT_2$ because by definition $S_2^{Match}$ are connected to $S_1$ within edges of weight at most $2 OPT_2$.  Therefore, 
$$ cost_2(D\setminus D_1, S_2) \leq \max\big\{cost_2(D\setminus D_1, S_2^{Unmatch}), \ cost_2(D\setminus D_1, S_1) + 2OPT_2\big \}
$$
and therefore
\begin{equation}
\label{eq:S2}
    cost_1(D_1,R_1) + cost_2(D\setminus D_1, S_2) \leq OPT_1 + 5OPT_2.
\end{equation}
From  \eqref{eq:S1} and \eqref{eq:S2}, we conclude that
\[  cost_1(D_1,R_1) + \max\limits_{S \in \{S_1,S_2\} } cost_2(D\setminus D_1, S) \leq OPT_1 + 5OPT_2. \] \hfill  \Halmos
\endproof
\subsection{$\bf p$ Explicit Scenarios}
We now consider the case of  explicit list of $p$ scenarios, i.e., $\mathcal{S} = \{ S_1, S_2, \ldots, S_p  \}$. Building upon the ideas from Algorithm \ref{alg:2scenarios}, we present  $O(p^{1.59})$-approximation to TSRMB with $p$ scenarios.  The idea of our algorithm is  to construct the  representative scenario recursively by processing  pairs of ``scenarios'' at each step. Hence, we need $O(\log_2 p)$ iterations to reduce the problem to an instance of a single scenario. At each iteration, we show that we only lose a multiplicative factor of $3$ so that the final approximation ratio is $O(3^{\log_2 p})=O(p^{1.59})$. We present details in Algorithm \ref{alg:pscenarios}. Theorem \ref{thm:pscenarios} states the theoretical guarantee.  Note that the approximation guarantee of our algorithm grows in a sub-quadratic manner with the number of scenarios $p$ and  it is  an interesting question if there exists an algorithm for TSRMB with an approximation guarantee that does not depend on the number of scenarios.

{\small \begin{algorithm} 
	\caption{$p$ explicit scenarios.} 
	\label{alg:pscenarios}
	\begin{algorithmic}[1]
	\REQUIRE{First-stage riders $R_1$,  scenarios $\{ S_1, S_2, \ldots, S_p  \}$, drivers $D$ and value of $OPT_2$.}
	\ENSURE{First-stage decision $D_1$.}
	\STATE{Initialize $\hat{S}_j:= S_j$ for $j = 1, \ldots, \ p$.}
	\FOR{$i = 1, \ldots, \ \log_2{p}$}
	    \FOR{$j = 1,2, \ldots, \ \frac{p}{2^{i}}$}
	    \STATE $ \sigma(j)=j+\frac{p}{2^{i}}$
	    \STATE{$M_{j}:=$ maximum cardinality matching between $\hat{S}_j$ and $\hat{S}_{\sigma(j)}$ with edges of weight at most $2\cdot 3^{i-1}\cdot OPT_2$.}
	    \STATE{$\hat{S}_{\sigma(j)}^{Match}:= \{ r \in \hat{S}_{\sigma(j)} \ | \ \exists \ s \in \hat{S}_j \ \mbox{ s.t }(s,r) \in M_{j} \}$.}
	    \STATE{$\hat{S}_{\sigma(j)}^{Unmatch}:=\hat{S}_{\sigma(j)} \setminus \hat{S}_{\sigma(j)}^{Match}$ }
	    \STATE{$\hat{S}_j = \hat{S_j} \cup \hat{S}_{\sigma(j)}^{Unmatch}$.}
	\ENDFOR
	\ENDFOR
	\RETURN $D_1:=$ TSRMB-1-Scenario($R_1, \hat{S_1}, D)$.
	\end{algorithmic} 
\end{algorithm}
}

\begin{theorem} \label{thm:pscenarios}
Algorithm \ref{alg:pscenarios} yields a solution with total cost of $O(p^{1.59})\cdot OPT$ for TSRMB with an explicit list of $p$ scenarios.
\end{theorem}

\proof{\textit{Proof.}
The algorithm reduces the number of considered ``scenarios'' by half in every iteration, until only one scenario remains. In iteration $i$, we have $\frac{p}{2^{i-1}}$ scenarios that we aggregate in $\frac{p}{2^{i}}$ pairs, namely $(\hat{S}_j$ , $\hat{S}_{\sigma(j)})$ for $j \in \{1, 2,  \ldots, \frac{p}{2^i} \}$. For each pair, we construct a single representative scenario which plays the role of the new $\hat{S}_j$ at the start of the next iteration $i+1$.

\begin{claim}\label{claim:iteration}
There exists a first-stage decision $D^*_1$, such that at every iteration $i \in \{1,\ldots, \log_2{p}\}$, we have for all $j \in \{1,2, \ldots ,\frac{p}{2^i} \}$:
\begin{enumerate}
    \item[(1)] $R_1$ can be matched to $D^*_1$ with a first-stage cost of $OPT_1$.
    \item[(2)] $\hat{S_j} \cup \hat{S}_{\sigma(j)}^{Unmatch}$ can be matched to $D \setminus D^*_1$ with a second-stage cost at most $3^{i} \cdot OPT_2$.
    \item[(3)] There exists a matching between $\hat{S}_{\sigma(j)}^{Match}$ and $\hat{S}_j$ with all edge weights at most $ 2\cdot 3^{i-1}\cdot OPT_2$.
\end{enumerate}
\end{claim}

\proof{\textit{Proof of Claim \ref{claim:iteration}}.}
Statement (3) follows from the definition of  $\hat{S}_{\sigma(j)}^{Match}$ in Algorithm \ref{alg:pscenarios}. Let's show (1) and (2) by induction over $i$. 
\begin{itemize}
    \item \textbf{Initialization:} for $i=1$, let's take any two scenarios $\hat{S}_j={S}_j$ and $\hat{S}_{\sigma(j)}={S}_{\sigma(j)} $. We know  that these two scenarios can be matched to drivers of the optimal solution in the original problem with a cost at most $OPT_2$. In the proof of Lemma \ref{lemma:2scenarios}, we show that if we use the optimal first-stage decision $D_1^*$ of the original problem, then we can match $\hat{S}_j$ and $\hat{S}_{\sigma(j)}^{Unmatch}$ simultaneously to $D\setminus D^*_1$ with a cost at most $3OPT_2.$ 
    \item \textbf{Maintenance.} Assume the claim is true for all values less than  $i \leq \log_2{p} - 1$. We show it is true for $i+1$. Since the claim is true for iteration $i$, we know that at the start of iteration $i+1$, for $j \in \{1, \ldots, \frac{p}{2^i} \}$, $\hat{S}_j$ can be matched to $D\setminus D^*_1$ with a cost at most $3^i \cdot OPT_2.$ We can therefore consider a new TSRMB problem with $\frac{p}{2^i}$ scenarios, where using $D_1^*$ as a first-stage decision ensures a second-stage optimal value at most $\widehat{OPT}_2 = 3^i \cdot OPT_2$. By the proof of Lemma \ref{lemma:2scenarios}, and by using $D_1^*$ as a first-stage decision in this problem, we ensure that for $j \in \{1, \ldots, \frac{p}{2^{i+1}}\}$, $\hat{S}_j$ and $\hat{S}^{Unmatch}_{\sigma(j)}$ can be simultaneously matched to $D\setminus D_1^*$ with a cost at most $3\widehat{OPT}_2 = 3^{i+1}\cdot OPT_2$. 
    \hfill \Halmos
\end{itemize}
\endproof

From Claim \ref{claim:iteration}, we have in  the last iteration $i = \log_2{p}$, \begin{itemize}
    \item $R_1$ can be matched to $D_1^*$ with a first-stage cost of $OPT_1$.
    \item $\hat{S}_1$ can be matched to $D \setminus D^*_1$ with a second-stage cost at most $3^{\log_2{p}}\cdot OPT_2$.
\end{itemize}
Computing the single scenario solution for $\hat{S}_1$ will therefore yield a first-stage decision $D_1$ that gives a total cost of at most $OPT_1 + 3^{\log_2{p}}\cdot OPT_2$ when the second stage is evaluated on the scenario $\hat{S}_1$. We now bound the cost of $D_1$ on the original scenarios $\{S_1, \ldots, \ S_p\}$. Consider a scenario $S \in \{S_1, \ldots, \ S_p\}$. The riders in $S \cap  \hat{S}_1$ can be matched to some drivers in $D\setminus D_1$ with a cost at most $OPT_1 + 3^{\log_2{p}}\cdot OPT_2$. As for other riders of $S \setminus \hat{S}_1$, they are not part of $\hat{S}_1$ because they have been matched and deleted at some  iteration $i < \log_2 p$. Consider riders $r$ in $S \setminus \hat{S}_1$ that were matched and deleted from a representative scenario at some iteration, then by statement (3) in Claim \ref{claim:iteration}, each  $r$ can be connected to a different rider in $\hat{S}_1 \setminus (\hat{S}_1 \cap S )$ within a path of  length at most 
$$\sum\limits_{t = 1}^{\log_2{p}} 2\cdot 3^{t-1} \cdot OPT_2= (3^{\log_2 p} -1)\cdot OPT_2.$$
 We know that $R_1$ and $\hat{S}_1$  can be matched respectively to $D_1$ and $D \setminus D_1$  with a total cost at most $OPT_1 + 3^{\log_2{p}}\cdot OPT_2$.
 Therefore, we can match $R_1$ and $S$  respectively to $D_1$ and $D \setminus D_1$ with a total cost at most
$$OPT_1 + 3^{\log_2{p}}\cdot OPT_2 + (3^{\log_2 p} -1)\cdot OPT_2 = O(3^{\log_2{p}})\cdot OPT = O(p^{\ln{3}/\ln{2}})\cdot OPT=  O(p^{1.59})\cdot OPT  .$$
Therefore, the worst-case total cost of the solution returned by Algorithm \ref{alg:pscenarios} is $O(p^{1.59})\cdot OPT$. \Halmos
\endproof

{\color{black}
\vspace{2mm}
\noindent
{\bf Remark.} We also observe that the structure of TSRMB with explicit scenarios reveals a connection to the chromatic $k$-supplier problem, which was recently posed as an open problem by~\cite{goyal2023tight}. We explore this connection in Appendix~\ref{apx:tamim}, where we show how to reduce chromatic $k$-supplier to an instance of TSRMB with $p$ scenarios.
}

\section{\red{General Model of Uncertainty Without Surplus}\label{sec:implicit}}
\red{In this section, we consider TSRMB under a general model of uncertainty. The uncertainty model may be implicit or explicit, and we denote the uncertainty set by $\mathcal{S}$. Our analysis in subsequent sections depends on the balance between supply (drivers) and demand (riders). Specifically, we introduce the notion of surplus $\ell$, defined as the excess in the number of available drivers for matching both first-stage riders and a second-stage scenario with the maximum size, i.e., 
$$ \ell = |D| - |R_1| - k, $$
where 
$$ k = \max \{ |S| : S \in \mathcal{S} \}. $$
Here, $k$ represents the maximum size of any scenario in the uncertainty set. If the surplus $\ell < 0$, TSRMB becomes infeasible, as there would not be enough drivers to satisfy the maximum possible demand in the second stage. Therefore, we assume $\ell \geq 0$. In this section, we analyze the case of zero surplus ($\ell=0$), meaning the supply is equal to the maximum demand. The case of general surplus is addressed in the following section.}

{\color{black}For the zero surplus case ($\ell=0$), we demonstrate a tight $3$-approximation for TSRMB (Theorem~\ref{thm:implicitnosuplusproof}). Specifically, we propose a simple algorithm that achieves a $3$-approximation for TSRMB under any model of uncertainty when there is zero surplus. Furthermore, this approximation is tight. In fact, in Theorem \ref{thm:hardness_robust}, we showed that TSRMB with an explicit model of uncertainty using three scenarios is NP-hard to approximate within a factor better than $3-\epsilon$ for any $\epsilon>0$. The instance used to prove this hardness result (see the proof of the second part of Theorem \ref{thm:hardness_robust} in Appendix \ref{appendix:hardness}) has a surplus equal to $0$, (we have three scenarios of equal size $k$ and $|D| = |R_1| + k$). Thus, TSRMB is NP-hard to approximate within a factor better than $3-\epsilon$ even when the surplus is $0$.  Therefore, our approximation guarantee is tight and effectively closes the gap for this case.}

\begin{theorem}\label{thm:implicitnosuplusproof}
    Algorithm \ref{robustnosurplusalgo} yields a solution with total cost at most $OPT_1 + 3OPT_2$ for TSRMB with a general model of uncertainty and no surplus. {\color{black}Furthermore, there exists no $(3-\epsilon)$-approximation algorithm for TSRMB with no surplus for any fixed $\epsilon > 0$, unless $P = NP$.}
\end{theorem}

\vspace{-3mm}
{\small \begin{algorithm}
    \caption{General Uncertainty with No Surplus.}
    \label{robustnosurplusalgo}
    \begin{algorithmic}[1]
    \REQUIRE{First-stage riders $R_1$, uncertainty set $\mathcal{S}$ for second-stage riders, drivers $D$.}
    \ENSURE{First-stage decision $D_1$.}
    \STATE{$S_1 :=$ a second-stage scenario of maximum size.}
    \STATE{$D_1 :=$ TSRMB-1-Scenario$(R_1, S_1, D)$.}
    \RETURN $D_1$.
    \end{algorithmic} 
\end{algorithm}}
The proposed algorithm for the $3$-approximation operates as follows. We solve a single scenario TSRMB for a scenario $S_1 \in \mathcal{S}$ with the maximum size $k$. The scenario $S_1$ is chosen arbitrarily among scenarios with maximum size. Then, we simply use the first-stage drivers from the obtained solution as our first first-stage drivers for our TSRMB problem. {\color{black}Note that, while in general solving TSRMB requires knowledge of the entire uncertainty set, in the no-surplus case we only need to know a single scenario with zero surplus.}
The complete algorithm for the no-surplus case is presented in Algorithm \ref{robustnosurplusalgo}, along with its worst-case guarantee in Theorem \ref{thm:implicitnosuplusproof}. The intuition behind the proof of Theorem \ref{thm:implicitnosuplusproof} is as follows. Given the absence of surplus, there exists a set of drivers of size $k$, where all the second-stage scenarios are matched to this set or its subsets in the optimal solution. Leveraging this, we establish that for any two scenarios from the uncertainty set, we can perfectly match the scenario of the smaller size with the other scenario, using only edges with weight at most $2 OPT_2$. Using this observation, we can bound the cost for other scenarios within through the {\color{black}triangle inequality}, thereby showing a $3$-approximation. The complete proof is presented below.} 


\proof{\textit{Proof of Theorem \ref{thm:implicitnosuplusproof}}.}
{\color{black}Fix any optimal solution $D^*_1$, and let $OPT_1$ and $OPT_2$ be its first-stage and second-stage costs, respectively.}
Let $f(D_1)$ be the total cost of the solution returned by Algorithm~\ref{robustnosurplusalgo}. We claim that $f(D_1) \leq OPT_1 + 3 OPT_2$. Let $S_1 \in {\cal S}$ be a second-stage scenario of maximum size $k$.
Since $D_1$ is the optimal solution  of TSRMB problem with the single scenario $S_1$, this implies that
\begin{equation} \label{eq:martil}
    cost_1(D_1,R_1) + cost_2(D\setminus D_1, S_1) \leq OPT_1 + OPT_2.
\end{equation}
 Since the surplus is zero, we have  $|D| = |R_1| + k$. 
 Let $D_2^*=D \setminus D_1^*$. 
 In particular, $D_2^*$ is the set of drivers used in an optimal solution to match any second-stage scenario of riders. We have $|D_2^*|=k$. Hence, there exists a perfect matching between $D_2^*$ and $S_1$ with bottleneck cost at most $OPT_2$, i.e., $cost_2(D_2^*, S_1) \leq OPT_2$. 
 Consider another scenario $S \in \mathcal{S}$.
We have $cost_2(D_2^*, S) \leq OPT_2$ and by definition of $S_1$,  $|S| \leq |S_1|=k$. So, there exists a matching between $S$ and $D_2^*$ with  size equals to $|S|$ and a bottleneck cost at most $OPT_2$. Note that the drivers from $D_2^*$ used in this matching are also matched to a subset of $S_1$. Therefore, 
by applying the triangle inequality, there exists a matching between $S$ and $S_1$, of size equal to $|S|$ and a bottleneck cost at most $2OPT_2$. Consequently, for any scenario $S \in \cal S$,
\begin{align*}
    cost_1(D_1,R_1) + cost_2(D\setminus D_1, S) &\leq cost_1(D_1,R_1) + cost_2(D\setminus D_1, S_1) + cost_2( S_1, S)\\ &\leq OPT_1+ OPT_2 + 2OPT_2 = OPT_1 + 3OPT_2,
\end{align*} 
where the first inequality follows from the {\color{black}triangle inequality}, the second one
follows from \eqref{eq:martil} and from the fact that we establish a matching between $S$ and $S_1$, of size $|S|$ and a bottleneck cost at most $2OPT_2$.  
{\color{black} Finally, the hardness of approximation result follows directly from the proof of the second part of Theorem \ref{thm:hardness_robust}, which is based on an instance with zero surplus..}
\hfill
\Halmos
\endproof

\vspace{3mm}
\noindent
{\bf Remark.}
    We note that when the surplus is strictly greater than 0, Algorithm \ref{robustnosurplusalgo} no longer yields a constant approximation and its worst-case performance can be as bad as $\Omega(m)$. In particular, consider the simple example
 in Figure \ref{fig:example_surplus}, with a budget of uncertainty set where  we have two second-stage riders and $k=1$. {\color{black} We have $m$ (even) first-stage riders and $m+2$ drivers. The first-stage riders are indexed from $0$ to $m-1$, where index $0$ refers to the top rider and index $m-1$ to the bottom rider. Riders with even indices have one upward edge of distance $1-\epsilon$ and one downward edge of distance $1$, while riders with odd indices have one upward edge of distance $1$ and one downward edge of distance $1-\epsilon$. The single-scenario solution for $S_1$ results in all first-stage riders being matched via downward edges. Hence, if $S_2$ is realized, the cost of matching $S_2$ to the closest available driver is $\Omega(m)$}. By symmetry, solving the single scenario problem for $S_2$ yields a $\Omega(m)$ bottleneck cost for $S_1$. This instance has a surplus $\ell=1$.
\begin{figure}[h]
    \centering
    \includegraphics[scale=0.25]{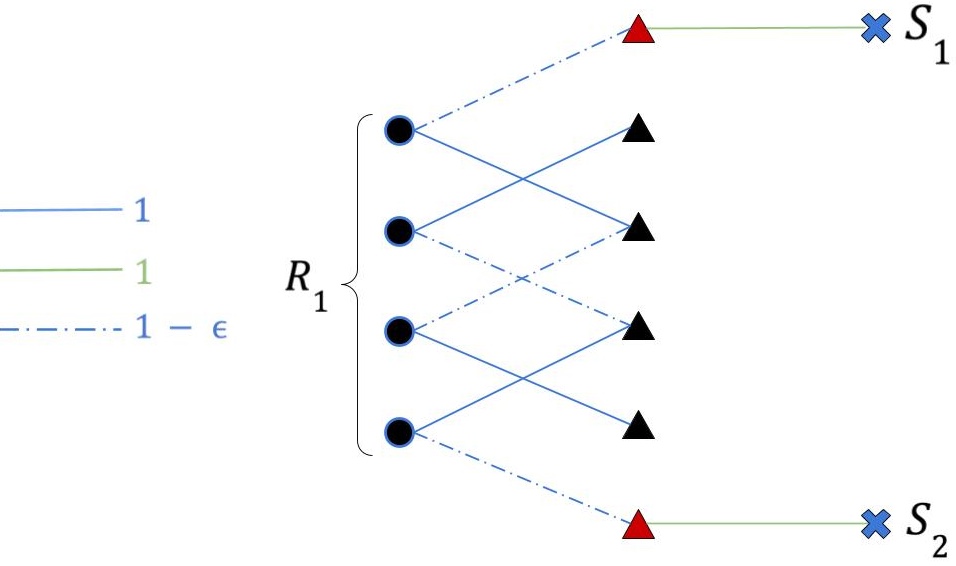}
    \caption{First-stage riders are depicted as black dots and
drivers as triangles. The two second-stage riders are depicted as blue crosses{\color{black}, and their respective optimal drivers as red triangles}. Second-stage optimum are depicted
as solid green edges. $\mathcal{S} = \{S_1, S_2\}$, $k=1$ and $\ell=1$.}
    \label{fig:example_surplus}
\end{figure}

\section{Implicit Scenarios: Budget of Uncertainty} \label{sec:general-surplus}

In this section, we consider an implicit model of uncertainty given by the budget of uncertainty set \eqref{eq:budget}. Recall the budget of uncertainty set $\cal S$ is given by $ \mathcal{S} = \{ S \subseteq R_2 \mid |S| \leq k \}.$ Here $k$ is the maximum size of a scenario and the surplus $\ell$ is given by $\ell= |D|-|R_1|-k$. As observed in the example of Figure \ref{fig:example_surplus}, the TSRMB problem becomes challenging  even with a unit surplus of drivers and  Algorithm \ref{robustnosurplusalgo} could be arbitrarily bad. Motivated by this, we focus on the case of a small surplus $\ell$ which we study in Section \ref{subsection:mathilde}. Then, in Section \ref{subsec:last}, we consider arbitrary surplus with $k=1$.

\subsection{Small Surplus} \label{subsection:mathilde}

We assume  that $\ell < k $, i.e., the excess in the total available drivers is smaller than the size  of any scenario. We present a constant approximation algorithm in this regime  for the implicit model of uncertainty where the size of scenarios is relatively small with respect to the size of the universe  ($k = O (\sqrt{n})$). This technical assumption is needed for our analysis but it is not too restrictive and still captures the regime where the number of scenarios can be exponential.
Our algorithm attempts to cluster the second-stage riders in different groups (a \textit{ball} and a set of \textit{outliers}) in order to reduce the number of possible worst-case configurations. We then solve a sequence of instances with representative riders from each group. In what follows, we present our construction for these groups of riders.

\vspace{3mm}
\noindent
{\bf Our construction}. First, we show that many of the riders are contained in a ball with radius $3 OPT_2$. The center of this ball $\delta$ can be found by enumerating over all drivers and selecting the one with the least maximum distance to its closest $k$ second-stage riders, i.e.,
\begin{equation} \label{eq:7ankor}
    \delta = \argmin\limits_{\delta' \in D} \max\limits_{r \in R_k(\delta')} d(\delta', r),
\end{equation}
where $R_k(\delta')$ is the set of the $k$ closest second-stage riders to $\delta'$. Formally, we have the following lemma for which we defer the proof to Appendix \ref{appendix:implicit}.

\begin{restatable}{lemma}{centerlemma}
\label{lemma:center}
Suppose $k \leq \sqrt{\frac{n}{2}}$ and $\ell < k$ and let $\delta$ be the driver given by \eqref{eq:7ankor}.
Then, the ball $\mathcal{B}$ centered at $\delta$ with radius $3OPT_2$ contains at least $n-\ell$ second-stage riders. Moreover, the distance between any of these riders and any rider in $R_k(\delta)$ is at most $4OPT_2$.  
\end{restatable}

\proof{\textit{Proof.}
Let $\delta$ be the driver given by \eqref{eq:7ankor}. We claim that the $k$ closest riders to $\delta$ are all within a distance at most $OPT_2$ from $\delta$. Consider $D_2^*$ to be the $k+\ell$ drivers left for the second stage in the optimal solution. Every driver in $D_2^*$ can be matched to a set of different second-stage riders over different scenarios. Let us rank the drivers in $D_2^*$ according to how many different second-stage riders they are matched to over all scenarios, in descending order. Formally, let $D_2^* = \{\delta_1, \delta_2, \ldots, \delta_{k+\ell}\}$ and let $R^*(\delta_i)$ be the second-stage riders that are matched to $\delta_i$ in the optimal solution in some scenario. \red{ Let us assume} $$|R^*(\delta_1)| \geq \ldots \geq |R^*(\delta_{k+\ell})|.$$ We claim that $|R^*(\delta_1)| \geq k$. In fact, we have $\sum\limits_{i=1}^{k+\ell} |R^*(\delta_i)| \geq n$ because every second-stage rider is matched to at least one driver in some scenario. Therefore 
$$|R^*(\delta_1)| \geq \frac{n}{k+\ell} \geq \frac{n}{2k} \geq k.$$ We know that all the second-stage riders in $R^*(\delta_1)$ are within a distance at most $OPT_2$ from $\delta_1$. Therefore $\max\limits_{r\in R_k(\delta_1)} d(\delta_1,r) \leq OPT_2$.
But we know that by definition of $\delta$,
\begin{equation*}
    \max\limits_{r\in R_k(\delta)} d(\delta,r) \leq \max\limits_{r\in R_k(\delta_1)} d(\delta_1,r) \leq OPT_2.
\end{equation*}

This proves that the $k$ closest second-stage riders to $\delta$ are within a distance at most $OPT_2$. Let $R(\delta)$ be the set of all second-stage riders that are within a distance at most $OPT_2$ from $\delta$. Recall that $R_k(\delta)$ is the set of the $k$ closest second-stage riders to $\delta$. In the optimal solution, the scenario $R_k(\delta)$ is matched to a set of at least \red{ $k-1$ other drivers} $\{\delta_{i_1}, \ldots \delta_{i_{k-1}}\} \subset D_2^* \setminus \{\delta\}$. We show a lower bound on the size of $R(\delta)$ and the number of riders matched to $ \{\delta_{i_1}, \ldots \delta_{i_{k-1}}\}$ over all scenarios in the optimal solution.

\begin{claim}\label{claim:size}
$\big|R(\delta) \red{\cup} \bigcup\limits_{j=1}^{k-1} R^*(\delta_{i_j})\big| \geq n-\ell$.
\end{claim}
\proof{\textit{Proof of Claim \ref{claim:size}}.}
Suppose the opposite, suppose that at least $\ell+1$ riders from $R_2$ are not in the union. Let $F$ be the set of these $\ell+1$ riders. Since $\ell+1 \leq k$, we can construct a scenario $S$ that includes $F$. In the optimal solution, and in particular, in the second-stage matching of $S$, at least one rider from $F$  needs to be matched to a driver from $\{\delta, \delta_{i_1}, \ldots \delta_{i_{k-1}}\}$. Otherwise there are only $\ell$ second-stage drivers left to match all of $F$.  Therefore there exists $r \in F$ such that either $r \in R(\delta)$ or there exists $j \in \{1,\ldots,k-1\}$ such that $r \in R^*(\delta_{i_j})$. This shows that $r \in R(\delta) \red{\cup} \bigcup\limits_{j=1}^{k-1} R^*(\delta_{i_j})$, which is a contradiction. Therefore, at most $\ell$ second-stage riders are not in the union. 
\hfill \Halmos
\endproof

\begin{claim}\label{claim:center}
For any rider  $r \in R(\delta) \red{\cup} \bigcup\limits_{j=1}^{k-1} R^*(\delta_{i_j})$, we have $
    d(r,\delta) \leq 3 OPT_2$.
\end{claim}

\proof{\textit{Proof of Claim \ref{claim:center}}.}
If $r \in R(\delta) $ then by definition we have $d(r,\delta) \leq OPT_2$. Now suppose $r \in R^*(\delta_{i_j})$ for $j \in [k-1]$. Let $r'$  be the rider from scenario $R_k(\delta)$ that was matched to $\delta_{i_j}$ in the optimal solution.
\begin{equation*}
    d(r, \delta) \leq d(r,\delta_{i_j}) + d(\delta_{i_j}, r') + d(r', \delta) \leq 3OPT_2. \hfill \Halmos
\end{equation*}
\endproof

From Claim \ref{claim:center}, we see that the ball centered at $\delta$, with radius $3OPT_2$, contains  at least $n-\ell$ second-stage riders in $R(\delta) \red{\cup} \bigcup\limits_{j=1}^{k-1} R^*(\delta_{i_j})$. This proves the first part of the lemma. The second part is proved in the next claim.

\begin{claim} \label{claim:tamim}
For  $r_1 \in R_k(\delta)$ and $r_2 \in   R_k(\delta) \red{\cup} \bigcup\limits_{j=1}^{k-1} R^*(\delta_{i_j})$, we have $d(r_1, r_2) \leq 4 OPT_2$.
\end{claim}

\proof{\textit{Proof of Claim \ref{claim:tamim}}.}
Let $r_1 \in R_k(\delta)$. If $r_2 \in R_k(\delta) $ then $d(r_1,r_2) \leq d(r_1, \delta) + d(\delta,r_2) \leq 2 OPT_2$.
If  $r_2 \in R^*(\delta_{i_j})$ for some $j$, and $r'$ is the rider from scenario $R_k(\delta)$ that was matched to $\delta_{i_j}$ 
\begin{equation*}
    d(r_1,r_2) \leq d(r_1, \delta) + d(\delta, r') + d(r', \delta_{i_j}) + d(\delta_{i_j},r_2) \leq 4OPT_2. \hfill \Halmos
\end{equation*}

\endproof
\noindent Combining the two last claims concludes the proof of Lemma \ref{lemma:center}.
\hfill \Halmos
\endproof

\vspace{3mm}

Now, let us focus on the rest of second-stage riders. We introduce the following definition. We say that a rider $r \in R_2$ is an \textit{outlier} if $ d(\delta, r) > 3 OPT_2 $. Denote $\{o_1, o_2, \ldots, o_{\ell} \}$  the farthest $\ell$ riders from $\delta$ with $d(\delta, o_1) \geq d(\delta, o_2)  \geq \ldots \geq d(\delta, o_{\ell})$. Note that by Lemma \ref{lemma:center}, the $n-\ell$ riders in $\mathcal{B}$ are not outliers and the only potential outliers could be in $ \{ o_1, o_2, \ldots, o_{\ell} \}$.  Let $j^*$ be the threshold such that $o_1, o_2, \ldots, o_{j^*}$ are outliers and $o_{j^*+1}, \ldots, o_{\ell}$ are not,
with the convention that ${j^*}=0$ if there is no outlier. There are $\ell+1$ possible values for $j^*$. We call each of these possibilities a \textit{configuration}. For $j = 0, \ldots, \ell$, let $C_j$ be the configuration corresponding to threshold candidate $j$. Note that $C_0$ is the configuration where there is no outlier and  $C_{j^*}$ is the correct configuration (See Figure \ref{fig:threshold}).

\begin{figure}[!h]
    \centering
    \includegraphics[scale=0.4]{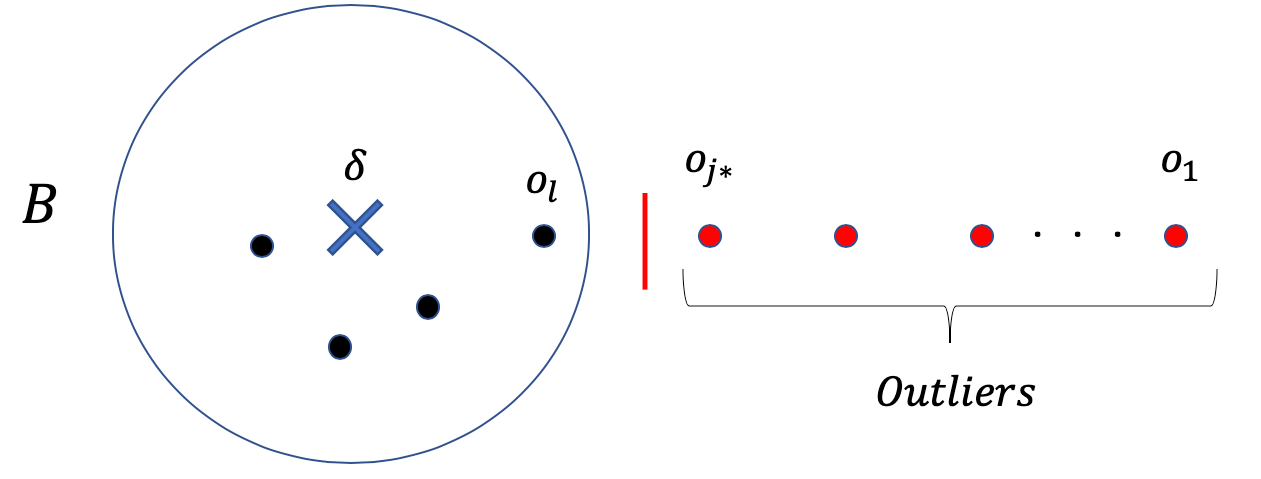}
    \caption{Configuration $C_{j^*}$}
    \label{fig:threshold}
\end{figure}

Now, we are ready to describe our algorithm. Recall that $R_k(\delta)$ are the closest $k$ second-stage riders to $\delta$. For the sake of simplicity, we denote $S_1=R_k(\delta)$ and $S_2=\{o_1 \ldots o_{\ell}\}$. Note that $S_2$ is a feasible scenario since $\ell <k$. For every configuration $C_j$, we form a representative scenario using $S_1$ and  $\{o_1 \ldots o_j\}$. We solve  TSRMB with this single representative scenario $S_1 \cup \{o_1 \ldots o_j\}$ and denote $D_1(j)$ the corresponding optimal solution, i.e.,
\begin{equation*}\label{c_j:configuration}
    D_1(j) = \mbox{TSRMB-1-Scenario}(R_1, S_1 \cup \{o_1 \ldots o_j\}, D). 
\end{equation*}

Since we cannot evaluate the cost of $D_1(j)$ on all scenarios because we can have exponentially many, we evaluate the cost of $D_1(j)$ on the two proxy scenarios $S_1$ and  $S_2$. We finally show that the candidate $D_1(j)$  with minimum cost over these two scenarios, gives a constant approximation to our original problem (See Theorem \ref{thm:exponential}). The details of our algorithm are summarized below.

\begin{algorithm}
	\caption{Implicit scenarios with small surplus and $k \leq \sqrt{\frac{n}{2}}$.} 
	\label{smallsurplus}
	\begin{algorithmic}[1]
	\REQUIRE{First-stage riders $R_1$, second-stage riders $R_2$, size $k$ and drivers $D$.}
	\ENSURE{First-stage decision $D_1$.}
	\STATE{Set $\delta:=$ driver given by \eqref{eq:7ankor}.}
	\STATE{Set $S_1$:= the closest $k$ second-stage riders to $\delta$.}
	\STATE{Set $S_2:= \{o_1, \ldots, o_{\ell}\}$} the farthest $\ell$ second-stage riders from $\delta$ ($o_1$ being the farthest).
	\FOR{ $j = 0, \ldots, \ell$}
	    \STATE{$D_1(j):=\mbox{TSRMB-1-Scenario}(R_1, S_1 \cup \{o_1 \ldots o_j\}, D)$.}
	\ENDFOR
	\RETURN $D_1 = \argmin\limits_{D_1(j): \ j \in \{0,\ldots, \ell\}} cost_1\big(D_1(j),R_1\big) + \max\limits_{S \in \{S_1,S_2\}} cost_2\big(D\setminus D_1(j), S\big)$.
	\end{algorithmic} 
\end{algorithm}

\begin{theorem}\label{thm:exponential}
Algorithm \ref{smallsurplus} yields a solution with total cost at most $3OPT_1 + 17 OPT_2$ for TSRMB with implicit scenarios when $k \leq \sqrt{\frac{n}{2}}$ and $\ell < k$.
\end{theorem}

\proof{\textit{Proof.}
We present here a sketch of the proof. The complete details and proofs of the claims appear in Appendix \ref{appendix:implicit}.  For all $j \in \{0, \ldots, \ell\}$, denote
\begin{align*}
    \Omega_j  &=  cost_1\big(D_1(j),R_1\big) \\
    \Delta_j  &=  cost_2\big(D \setminus D_1(j), S_1 \cup \{o_1, \ldots, o_j\}\big) \\
    \beta_j &= cost_1\big(D_1(j),R_1\big) + \max\limits_{S \in \{S_1,S_2\}} cost_2\big(D\setminus D_1(j), S\big).
\end{align*}
Recall $f$ the objective function of TSRMB. In particular,
$$f\big(D_1({j})\big) = cost_1\big(D_1({j}),R_1\big) + \max\limits_{S \in \mathcal{S}} cost_2\big(D\setminus D_1({j}), S\big).$$

Our proof is based on the following two claims. Claim \ref{lemma:ball} establishes a bound on the cost of $D_1(j^*)$ when evaluated on the proxy scenarios $S_1$ and $S_2$ and on all the scenarios in $\mathcal{S}$. Recall that $j^*$ is the threshold index  for the outliers as defined earlier in our construction. Claim \ref{lemma:betaj} bounds the cost of $f(D_1(j))$ for any $j$. The proofs of both claims are presented in Appendix \ref{appendix:implicit}.
 \begin{claim}
\label{lemma:ball} $ $
$\Omega_{j^*} + \Delta_{j^*} \leq OPT_1 + OPT_2.$
\quad and \quad
$f( D_1(j^*))\leq OPT_1 + 5 OPT_2.$
\end{claim}

\begin{claim}
\label{lemma:betaj}
For all $j \in \{0, \ldots, l\}$ we have, $ \beta_j \leq f(D_1(j)) \leq \max\{\beta_j + 4OPT_2, \ 3\beta_j + 2OPT_2\}.$
\end{claim}

Suppose Algorithm \ref{smallsurplus} returns $D_1(\tilde{j})$ for some $\tilde{j}$. From Claim \ref{lemma:betaj} and the minimality of $\beta_{\tilde{j}}$:
$$
    f\big(D_1(\tilde{j})\big)  \leq
    \max\{\beta_{\tilde{j}} + 4OPT_2, 3\beta_{\tilde{j}} + 2OPT_2\}
\leq \max\{\beta_{j^*} + 4OPT_2, 3\beta_{j^*} + 2OPT_2\}.
$$
From Claim \ref{lemma:ball} and Claim \ref{lemma:betaj},  we have  $\beta_{j^*} \leq f\big(D_1(j^*)\big) \leq OPT_1 + 5OPT_2$.  We conclude that,
 $$   f\big(D_1(\tilde{j})\big) \leq \max\big\{OPT_1 + 9OPT_2, 3OPT_1 + 17 OPT_2\big\} = 3OPT_1 + 17 OPT_2. $$
  \hfill \Halmos
\endproof

\color{black}
\subsection{Arbitrary Surplus with $k=1$} \label{subsec:last}

In this part, we consider TSRMB when the surplus can be arbitrary, and each of the second-stage scenarios has a single rider ($k=1$). We present a $3$-approximation algorithm for this case. Recall that, from Theorem~\ref{thm:hardness3}, TSRMB is NP-hard to approximate within a factor better than $3-\epsilon$ under the budget of uncertainty set even when $k=1$. Therefore, our result in this section matches the hardness of approximation and closes the gap for the case of $k=1$.

In our algorithm, we first guess $OPT_2$ by simple enumeration over all the edges between the riders in $R_2$ and the drivers in $D$. We can assume that we know $OPT_2$, since the number of scenarios is exactly $n$, and therefore, we can evaluate any feasible solution in polynomial time. Next, we build a maximal set $S$ from the second-stage riders $R_2$ that are at a distance greater than $2  OPT_2$ from one another. Specifically, we (arbitrary) pick a rider in $R_2$ and add it to $S$, remove all riders that are within a distance of less than $2  OPT_2$, then from the remaining riders in $R_2$, we (arbitrarily) pick another rider, add it to the set $S$, and again remove all riders that are within a distance less than $2  OPT_2$ from the newly added rider. This process continues until all riders in $R_2$ have been considered. At the end of this process, the riders in $S$ are at a distance greater than $2  OPT_2$ from one another, and for each rider in $R_2 \setminus S$, there exists a rider in $S$ within a distance of at most $2  OPT_2$.

Finally, using Algorithm~\ref{alg:singlescenario}, we solve TSRMB-1-Scenario($R_1, S, D$), where the scenario $S$ is as described above. The obtained solution, say $D_1$, is our final solution that provides the $3$-approximation for the TSRMB problem with $k=1$. These steps are summarized in Algorithm~\ref{alg:k1}.

We show in Theorem~\ref{thm:koneimplicit} that the solution returned by Algorithm~\ref{alg:k1} gives a $3$-approximation to the TSRMB problem. The proof is presented in Appendix~\ref{appx:kone}.

{\small
\begin{algorithm}
\color{black}
	\caption{Implicit scenarios with arbitrary surplus and $k=1$.}
	\label{alg:k1}
	\begin{algorithmic}[1]
	\REQUIRE{First-stage riders $R_1$, second-stage riders $R_2$, drivers $D$, and value of $OPT_2$.}
	\ENSURE{First-stage decision $D_1$.}
    \STATE{Construct greedily a maximal set of riders from $R_2$ such that they are at a distance greater than $2  OPT_2$ from one another, and denote this set by $S$.}
    \RETURN $D_1 :=$ TSRMB-1-Scenario$(R_1, S, D)$.
	\end{algorithmic}
\end{algorithm}
}

\begin{restatable}{theorem}{koneimplicit}\label{thm:koneimplicit}
Algorithm \ref{alg:k1} yields a solution with total cost at most $OPT_1 + 3 OPT_2$ for TSRMB in the case where the  budget of uncertainty has parameter $k = 1$.
\end{restatable}

\color{black}

\section{Numerical Experiments}\label{sec:experiments}
In this section, we present an empirical comparison of Algorithm \ref{alg:2scenarios} with the greedy algorithm. We use a taxi data set from the city of Shenzhen to create realistic instances of the TSRMB problem.
\subsection{Data}
The data is collected for a month in the city of Shenzhen \citep{cheng2019stl}\footnote{The raw trajectory record can be found in \url{https://github.com/cbdog94/STL}.}. This data contains the GPS records of taxis in Shenzhen. The details of the data set are summarized in Table \ref{data_details}, where the sample rate means the interval between two adjacent GPS records. A trajectory is constructed by following one taxi between a pick-up (``Occupied''  value change from 0 to 1) and a drop-off. A snapshot of the data is presented in Table \ref{table:snapshot}.

\vspace{5mm}

\begin{center}
 \begin{tabular}{||c c c c c ||} 
 \hline
 Size & \# Taxis & \# Trajectories & Sample rate & Avg trip time  \\ [0.5ex] 
 \hline\hline
 32.7 GB & 9,475 & 6,068,516 & 10-30 s & 863s  \\ 
 \hline
\end{tabular}
\captionof{table}{Details of the taxi trajectory data}
\label{data_details}
\end{center}

\vspace{5mm}

\begin{center}
 \begin{tabular}{||c c c c c c c||} 
 \hline
 Taxi ID & Time & Longitude & Latitude & Speed & Direction & Occupied \\ [0.5ex] 
 \hline\hline
 B97U79 & 2009-09-23 21:30:00 & 113.80275 &  22.66913 & 66 & 157 & 0 \\ 
 B97U79 & 2009-09-23 21:30:20 & 113.80137 & 22.67106 & 18 &  157 & 1\\
 \hline
\end{tabular}
\captionof{table}{Example of the taxi trajectory data}
\label{table:snapshot}
\end{center}

\begin{figure}
    \centering
    \includegraphics[scale=0.4]{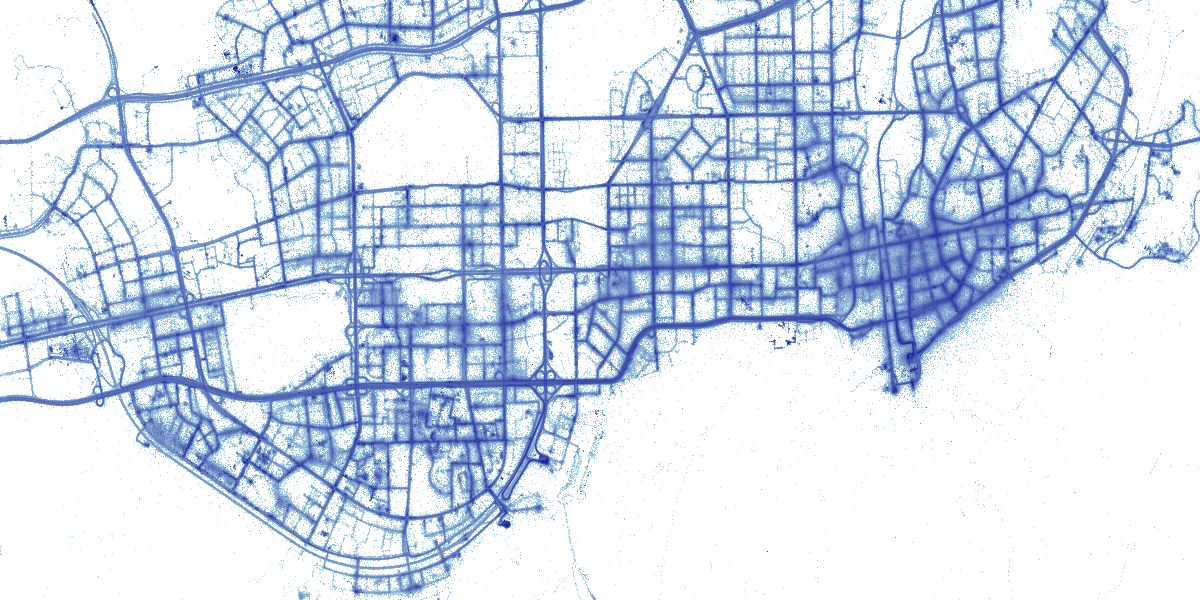}
    \caption{Union of the trajectories in Downtown Shenzhen on 09/17/2009}
    \label{fig:shenzen}
\end{figure}

\subsection{Experiment Setup}
We focus on the GPS records of downtown Shenzhen, with $|\mbox{longitude} -114.075| \leq 0.075$ and $|\mbox{latitude} - 22.54| \leq 0.03$ (See Figure \ref{fig:shenzen}). In a specific time range, we locate the riders by following taxis and observing when the occupied entry changes from 0 to 1. This change means that a pickup occurred and the rider's location is estimated to be the same as the taxi location at the time of pickup. For different days $d$ of the month, and different times $t$ of the day, we consider the pickups that were made in $[t,t+1\mbox{min}]$ to be the first-stage riders $R_1$. For the second-stage riders, we construct two scenarios $S_1, S_2$ using the pickups that occurred in $[t+1\mbox{min}, t+\mbox{2min}]$ in $d-7$ and $d-14$ respectively, which represent the same day as $d$ in the two previous weeks. We also construct the realized scenario $S^*$, which contains the pickups in $[t+\mbox{1min}, t+\mbox{2min}]$ of day $d$. We use the taxis of day $d$ that were not occupied in the past 5 minutes before $t$ to sample the set of drivers $D$. The edge weights correspond to the distances between drivers and riders. Note that in the data set, the number of all available drivers in the past 5 minutes is considerably higher than the number of pickups. Hence, to simulate a busier time,  we  randomly sample $2.5 \times |R_1|$ drivers in every instance. For every instance, we preform $10$ random driver samples, solve the problem for every sample, and report the average. We report results from different times of the 17th day of the month, with $S_1$ and $S_2$ constructed from the 3rd and the 10th day respectively.

{\color{black}
\subsection{Evaluation Metrics and Experimental Results}
We use Algorithm~\ref{alg:2scenarios} to solve the TSRMB problem with second-stage scenarios $S_1$ and $S_2$. Let $Alg(S_1, S_2)$ denote the cost of the solution returned by Algorithm~\ref{alg:2scenarios}, which, as defined in the TSRMB problem, is the sum of the average minimum-cost matching in the first stage and the worst-case minimum-cost bottleneck matching in the second-stage over scenarios ${S_1, S_2}$. We refer to this setting as the \textit{in-sample} case.
We denote by $OPT(S_1, S_2)$ the optimal cost of the TSRMB problem with the two scenarios ${S_1, S_2}$, computed using an integer programming formulation developed in Appendix~\ref{appendix:ip_tsrmb}. Let $Gr(S_1, S_2)$ denote the total worst-case cost of the greedy solution, which myopically solves the first-stage problem and uses the remaining drivers to serve either $S_1$ or $S_2$. Specifically, the greedy solution computes a minimum-cost matching between $R_1$ and $D$ to determine the first-stage decision $D_1$, whose cost is evaluated as the average minimum-cost matching. The second-stage cost is computed as the worst-case minimum bottleneck cost of a matching between $D \setminus D_1$ and the scenarios ${S_1, S_2}$. We compare the in-sample performance of Algorithm~\ref{alg:2scenarios} and the greedy algorithm to the optimal cost by reporting the ratios $Gr(S_1, S_2)/OPT(S_1, S_2)$ and $Alg(S_1, S_2)/OPT(S_1, S_2)$.

We also evaluate Algorithm~\ref{alg:2scenarios} on out-of-sample data. In particular, we take the first-stage decision $D_1$ returned by Algorithm~\ref{alg:2scenarios} and use the remaining drivers $D \setminus D_1$ to serve a realized second-stage scenario $S^*$. We refer to this setting as the \textit{out-of-sample} case, where $S_1$ and $S_2$ serve as predictions of $S^*$. Let $Alg(S^*)$ denote the total cost of our algorithm evaluated on the realized scenario $S^*$, comprising the average first-stage matching cost and the second-stage minimum bottleneck cost. Similarly, let $Gr(S^*)$ denote the total cost of the greedy solution evaluated on $S^*$. Both are computed using the same cost metric as in the TSRMB objective.
Additionally, we define $OPT^{S_1, S_2}(S^*)$ as the cost incurred by the optimal TSRMB solution for scenarios $S_1$ and $S_2$, evaluated on $S^*$. This is obtained by first solving TSRMB optimally on $S_1$ and $S_2$ (via the IP formulation) to obtain a first-stage decision $D_1$, and then using $D \setminus D_1$ to match $S^*$. 
Finally, $OPT(S^*)$ denotes the cost of the optimal solution that knows offline the scenario $S^*$, i.e., $OPT(S^*)$ is the cost of TSRMB problem with a single scenario $S^*$ which we compute using Algorithm \ref{alg:singlescenario}.
All evaluation metrics reported in Table~\ref{table:outofsample} are based on the TSRMB cost function, i.e., the sum of the average first-stage matching cost and the worst-case second-stage bottleneck cost. We compare the out-of-sample performance of Algorithm~\ref{alg:2scenarios}, the greedy algorithm, and the optimal two-scenario solution using the ratios $Alg(S^*)/OPT(S^*)$, $Gr(S^)/OPT(S^*)$, and $OPT^{S_1, S_2}(S^*)/OPT(S^*)$.
The in-sample and out-of-sample performances for different times on 09/17 are presented in Table \ref{table:outofsample}. The columns ``1st Stage'' and ``2nd Stage''  denote the time range of the first and second stage respectively.

In Table~\ref{table:matchingcomparison}, we further compare the second-stage bottleneck costs of the greedy algorithm and our algorithm on the out-of-sample scenario $S^*$. Specifically, we report the ratio $Gr_2(S^*)/Alg_2(S^*)$, where $Gr_2(S^*)$ and $Alg_2(S^*)$ denote the second-stage bottleneck costs of the greedy and Algorithm~\ref{alg:2scenarios} solutions, respectively. Finally, the column ``Sum weights Ratio ($Gr / Alg$)'' reports the ratio of total matching weights (i.e., the sum of edge weights) between the greedy and our algorithm’s solution when both are evaluated on $S^*$. Note that this total weight includes both the sum of  weights of the first-stage matching  and the sum of weights of the second-stage matching. Although, our algorithm does not optimize directly for this metric, we are computing it  to assess the generalization ability of our algorithm beyond the specific TSRMB objective.


\begin{table}[h]
\begin{center}
 {\color{black}
 \begin{tabular}{||c|c|c|c|c|c|c|c|c|c||}
 \hline
  \multirow{2}{*}{1st Stage} & \multirow{2}{*}{2nd Stage} & \multirow{2}{*}{$|D|$} & \multirow{2}{*}{$|R_1|$} & \multirow{2}{*}{$|S^*|$}& \multicolumn{3}{c|}{Out-of-sample} & \multicolumn{2}{c||}{In-sample} \\ \cline{6-10}
   &  &  &  &  & $\frac{Gr(S^*)}{OPT(S^*)}$  & $\frac{Alg(S^*)}{OPT(S^*)}$ & $\frac{OPT^{S_1S_2}(S*)}{OPT(S^*)}$ & $\frac{Gr(S_1,S_2)}{OPT(S_1,S_2)}$ & $\frac{Alg(S_1,S_2)}{OPT(S_1,S_2)}$\\ [1ex] 
 \hline\hline
 09:00-01 & 09:01-02  & 215 & 86 & 97 & 1.55 & 1.46 & 1.43 & 1.39 & 1.07\\
 \hline
10:00-01 &	10:01-02 & 187 & 75 & 54 & 1.50 & 1.48 & 1.45 & 1.21 & 1.11\\
\hline
11:00-01 &	11:01-02 & 210 & 84 & 78 & 1.17 & 1.12 & 1.11 & 1.12 & 1.08\\
\hline
12:00-01	& 12:01-02 & 215 & 86 & 91 & 1.20 & 1.11 & 1.05 & 1.31 & 1.01\\
\hline
13:00-01 & 13:01-02  & 205 & 82 & 93 & 1.20 & 1.22 & 1.18 & 1.15 & 1.03 \\
\hline
14:00-01 & 14:01-02  & 342 & 137 &138 & 1.42& 1.39 & 1.39 & 1.33 & 1.09\\
\hline
15:00-01 & 15:01-02  & 355 & 142 & 120 & 1.39 & 1.38 & 1.28 & 1.15& 1.12\\
\hline
16:00-01 & 16:01-02  & 345 & 138 & 132 & 1.38 & 1.33 & 1.32 & 1.29 & 1.19\\
\hline
17:00-01 & 17:01-02 	& 295 & 118 & 113 & 1.35 & 1.26 & 1.28& 1.39 & 1.11\\
\hline
18:00-01 & 18:01-02 	& 287 & 115 & 102 & 1.38 & 1.32 & 1.27 & 1.40& 1.09\\
\hline
19:00-01 & 19:01-02 	& 300 & 120 & 112 & 1.27 & 1.17 & 1.14 & 1.19 & 1.07 \\
\hline
20:00-01 & 20:01-02  & 307 & 123  & 134 & 1.58 & 1.37 & 1.34 & 1.57 &  1.17 \\ 
\hline
21:00-01 & 21:01-02  & 370 & 143  & 147 & 1.17 & 1.13& 1.13 & 1.06 & 1.25
 \\

 \hline
 \hline
\end{tabular}

}

\captionof{table}{\color{black}In-sample and out-of-sample performance evaluation of Greedy and Algorithm  \ref{alg:2scenarios}.}
\label{table:outofsample}
\end{center}
\end{table}

\begin{table}[h]
\begin{center}
 \color{black}{
 \begin{tabular}{||c|c|c|c|c|c|c||} 
 \hline
   1st Stage & 2nd Stage & $|D|$ & $|R_1|$ &  $|S^*|$ & $Gr_2(S^*)/Alg_2(S^*)$ & Sum weights Ratio ($Gr / Alg$)\\ [0.5ex] 
 \hline\hline
 09:00-01 & 09:01-02  & 215 & 86 & 97 & 1.32  & 0.98\\
 \hline
10:00-01 &	10:01-02 & 187 & 75 & 54 & 1.14 & 0.98\\
\hline
11:00-01 &	11:01-02 & 210 & 84 & 78 & 1.08 & 0.99 \\
\hline
12:00-01	& 12:01-02 & 215 & 86 & 91 & 1.16 & 0.99\\
\hline
13:00-01 & 13:01-02  & 205 & 82 & 93 & 1.18 & 0.98 \\
\hline
14:00-01 & 14:01-02  & 342 & 137 &138 & 1.15	& 0.99  \\
\hline
15:00-01 & 15:01-02  & 355 & 142 & 120 &  1.22 & 0.99  \\
\hline
16:00-01 & 16:01-02  & 345 & 138 & 132 & 1.17 & 0.98 \\
\hline
17:00-01 & 17:01-02 	& 295 & 118 & 113 & 1.08  & 0.99 \\
\hline
18:00-01 & 18:01-02 	& 287 & 115 & 102 & 1.40 & 1.01 \\ 
\hline
19:00-01 & 19:01-02 	& 300 & 120 & 112 & 1.41 & 0.98   \\
\hline
20:00-01 & 20:01-02  & 307 & 123  & 134 &  1.33 & 0.96  \\ 
\hline
21:00-01 & 21:01-02  & 370 & 143  & 147 & 1.28  & 0.97
 \\

 \hline
 \hline
\end{tabular}
}
\captionof{table}{\color{black}Comparison of Algorithm \ref{alg:2scenarios} and Greedy on out-of-sample w.r.t. the second-stage bottleneck cost and the total sum of weights.}
\label{table:matchingcomparison}
\end{center}
\end{table}

\subsection{Discussion}
We observe from Table~\ref{table:outofsample} that our Algorithm~\ref{alg:2scenarios} performs very well and achieves near-optimal costs. In particular, it significantly outperforms the greedy algorithm in both in-sample and out-of-sample evaluations.
In-sample, as shown in the last two columns of Table~\ref{table:outofsample}, our algorithm’s average cost is within 10\% of the optimal, while the greedy algorithm is, on average, about 27\% worse than the optimum.
Out-of-sample, we compare the performance of three methods: the greedy algorithm, Algorithm~\ref{alg:2scenarios}, and the first-stage solution obtained by optimally solving the TSRMB problem with scenarios $S_1$ and $S_2$. All three are evaluated on the realized scenario $S^*$ and benchmarked against the offline optimal that knows $S^*$ in advance. As shown in Table~\ref{table:outofsample}, the greedy algorithm incurs an average cost that is 35\% higher than the offline benchmark; Algorithm~\ref{alg:2scenarios} is within 28\%, and the TSRMB-based optimal solution is within 25\%. These results demonstrate that our algorithm closely tracks the performance of the  first-stage optimal solution of TSRMB and substantially improves over the greedy baseline.

Furthermore, since the greedy algorithm is by design optimal for the first-stage minimum cost matching, any improvement from Algorithm~\ref{alg:2scenarios} must come from a better second-stage bottleneck cost without significantly increasing the first-stage cost. Table~\ref{table:matchingcomparison} confirms this: the second-stage bottleneck cost of greedy  is on average 22\% higher than that of our algorithm. Meanwhile, the total matching weights across both stages are nearly identical between the two methods, as shown in the last column of Table~\ref{table:matchingcomparison}.
If the edge weights represent rider wait times, these results imply that our algorithm substantially reduces the maximum second-stage wait time while keeping the overall average wait time nearly unchanged (the average ratio of the last column is around 0.98\%). In this sense, our algorithm introduces more fairness in the wait time distribution by reducing extremes without sacrificing average performance.
}

\subsection{Evaluation on more than 2 Stages}

In practice, ride-hailing platforms apply their matching algorithms on a rolling horizon. That is, the planning interval is moved forward in time during each step, and then the matching problem is solved. To understand the long-term performance of our two-stage algorithm, we evaluate it on different horizons and compare its performance with that of a greedy algorithm that myopically solves each stage without considering the next batch. Specifically, we consider a number of consecutive stages \( T \in \{3, 5, 10\} \). The general framework we adopt to evaluate an algorithm on more than two stages is as follows: At time \( t \), we solve the two-stage problem (TSRMB) for stages \( t \) and \( t+1 \). This yields a set of drivers to be matched with the \( t \)-th stage scenario and a set of remaining available drivers. These drivers are then used to solve the next two-stage problem, with the realized scenario of \( t+1 \) as the first scenario, and the uncertain scenarios in \( t+2 \) as the second stage. The greedy algorithm operates as follows: At each time \( t \), it solves a minimum cost matching to match the realized scenario of time \( t \), without taking into account the next possible scenarios.

Similarly to the previous section's setup, we consider the pickups made in \([t, t + 1 \mbox{min}]\) as the riders of the \( t \)-th stage. In each stage, we construct two scenarios from days \( d-7 \) and \( d-14 \) respectively. We also construct the realized scenario, which contains the pickups of day \( d \). We use the taxis of day \( d \) that were not occupied in the past 5 minutes before the first pickup to sample the set of drivers \( D \). For each iteration, we sample a number of drivers proportional to the number of stages (\( D = 150 \cdot T \) for \( T \in \{3, 5, 10\} \)).

We evaluate our algorithm and the greedy algorithm on two metrics. The first metric is the sum of the bottleneck matching across all stages, reflecting the maximum distance/waiting time across all stages. The results are given in Table \ref{table:11}, where we report the ratio of the cost of the greedy algorithm over our algorithm {\color{black}that we denote \( \text{Greedy}_B/\text{ALG}_B \), where $B$ stands for bottleneck.}  We observe that our algorithm significantly improves over the greedy algorithm. {\color{black}  In particular, the cost of the greedy algorithm is 15\% higher in average than that of our algorithm.} This is because our algorithm hedges against future uncertainty, while the greedy approach is myopic.

The second metric is the sum of the average min cost matching across all stages, reflecting the average ride time across all stages. The results are given in Table \ref{table:22}, where we report the ratio of the cost of the greedy algorithm over our algorithm {\color{black}that we denote \( \text{Greedy}_A/\text{ALG}_A \), with $A$ referring to average matching.} We observe that our algorithm slightly improves over the greedy algorithm. In particular, the cost of the greedy algorithm can be 1\% to even 7\% higher than that of our algorithm.
Finally, we note that the improvement of our algorithm over the greedy one is more pronounced in the bottleneck metric than in the average ride time, which is expected, as the TSRMB model is designed to hedge against the bottleneck metric in its second stage.

\begin{table}[h]
\centering
\begin{minipage}{0.48\textwidth}
\centering
\color{black}
\begin{tabular}{||c|c||c|c||}
\hline
Instance & Starting time & \# Stages & $\frac{\text{Greedy}_B}{\text{ALG}_B}$ \\
\hline\hline
\multirow{3}{*}{3} & \multirow{3}{*}{15:00} & 3  & 1.17 \\
\cline{3-4}
& & 5 & 1.15 \\
\cline{3-4}
& & 10 & 1.10 \\
\hline
\multirow{3}{*}{2} & \multirow{3}{*}{18:00} & 3  & 1.21 \\
\cline{3-4}
& & 5 & 1.11 \\
\cline{3-4}
& & 10 & 1.08 \\
\hline
\multirow{3}{*}{1} & \multirow{3}{*}{21:00} & 3  & 1.23 \\
\cline{3-4}
& & 5 & 1.13 \\
\cline{3-4}
& & 10 & 1.18 \\
\hline\hline
\end{tabular}
\caption{Bottleneck matching for all stages}
\label{table:11}
\end{minipage}
\hfill
\begin{minipage}{0.48\textwidth}
\centering
\color{black}
\begin{tabular}{||c|c||c|c||}
\hline
Instance & Starting time & \# Stages & $\frac{\text{Greedy}_A}{\text{ALG}_A}$ \\
\hline\hline
\multirow{3}{*}{3} & \multirow{3}{*}{15:00} & 3  & 1.03 \\
\cline{3-4}
& & 5 & 1.02 \\
\cline{3-4}
& & 10 & 1.01 \\
\hline
\multirow{3}{*}{2} & \multirow{3}{*}{18:00} & 3  & 1.06 \\
\cline{3-4}
& & 5 & 1.07 \\
\cline{3-4}
& & 10 & 1.04 \\
\hline
\multirow{3}{*}{1} & \multirow{3}{*}{21:00} & 3  & 1.04 \\
\cline{3-4}
& & 5 & 1.02 \\
\cline{3-4}
& & 10 & 1.01 \\
\hline\hline
\end{tabular}
\caption{Average matching for all stages}
\label{table:22}
\end{minipage}
\end{table}

 \section{Other Cost Metrics}\label{sec:variants}
In this section, we initiate the study of other variants of two-stage matching problems, under both robust and stochastic models of uncertainty and for different cost functions. We define these problems, study their hardness of approximation and design  approximation algorithms in some specific cases. We summarize our results below and defer all the details to Appendix~\ref{appendix:stochastic} and Appendix~\ref{appendix:tsrm}.

\begin{enumerate}
    \item  \textbf{ {Two-Stage Robust Bottleneck Bottleneck Problem} (TSRBB). }The only difference from the TSRMB is that the first-stage cost is the bottleneck  of the first-stage matching. All the  approximation algorithms we have provided in Table \ref{table:results} for the TSRMB problem easily carry to this problem.
    {\color{black}In fact, the proofs of our algorithmic results do not rely on the specific structure of the first-stage cost function, as long as an optimal matching can be efficiently computed for the single-scenario case.}
    Specifically, the algorithms can be adapted to account for a bottleneck cost matching in place of the average minimum cost. For instance, in our algorithms, when we typically solve a first-stage minimum cost matching (as in Step 4 of Algorithm 1), it would now be replaced with a minimum bottleneck matching. This alteration does not impact the validity of our proofs, as they do not rely on the assumption that $OPT_1$ is derived from the cost function of minimum cost matching; it can be just as effectively applied to the cost function of bottleneck matching. {\color{black} For the hardness of approximation, Theorem \ref{thm:hardness2} and the first part of Theorem~\ref{thm:hardness_robust} still hold if we replace the average cost matching in the first stage with a bottleneck cost matching. In particular, TSRBB remains NP-hard with two scenarios and is NP-hard to approximate within a factor better than $3-\epsilon$ for general uncertainty sets. It is also $\Sigma_2^p$-hard under budget of uncertainty set. However, the second part of Theorem \ref{thm:hardness_robust} and Theorem \ref{thm:hardness3} rely on the fact that the first-stage cost is derived from an average matching, so these results do not carry over to the bottleneck case.
    
   }

    \item \textbf{Two-Stage Stochastic Matching Bottleneck Problem} (TSSMB). In this problem, the first-stage cost is the same as the TSRMB (e.g. the average matching weight). However, we assume that we have an explicit list of scenarios $\mathcal{S} = \{S_1, \ldots, S_p\}$. Scenario $S_i$ is realized with probability $p_i$. The objective is to minimize the function
    $$ \min\limits_{D_1 \subset D}\Big\{ cost_1(D_1,R_1) + \sum\limits_{i=1}^p p_i \cdot cost_2(D\setminus D_1, S_i)\Big\},$$
    where $ cost_2(D\setminus D_1, S_i)$ is the bottleneck matching cost between $D\setminus D_1$ and scenario $S_i$. In Appendix~\ref{appendix:stochastic}, we show this problem is NP-hard to approximate within a factor better than $5/3$. We also provide an algorithm that yields a 3-approximation when there is no surplus.
    \item \textbf{Two-Stage Robust Matching Problem} (TSRM). In this problem, the cost of the first stage is the total weight of the first-stage matching, and the second-stage cost is the total weight of the worst-case matching over scenarios. We present the formal definition of this problem in Appendix \ref{appendix:tsrm}, and show it is NP-hard even with two scenarios. Kalyanasundaram and Pruhs \cite{kalyanasundaram1993online} consider the online version of this problem, and show that the greedy algorithm is  $3$-competitive for two stages and therefore yields a 3-approximation in the worst case as well. We further improve this result and show a $7/3$-approximation when there is no surplus.
\end{enumerate}

\section{Conclusion}\label{sec:conclusion}

In this paper, we present a new two-stage robust optimization framework for matching problems under both explicit and implicit models of uncertainty. Our problem is motivated by real-life applications in the ride-hailing industry. We consider different cost functions under this model and study their theoretical hardness. We particularly focus on the Two-Stage Robust Matching Bottleneck variant and design approximation algorithms for implicit and explicit scenarios under different settings. Our algorithms give a constant approximation if the number of scenarios is fixed but require additional assumptions when there are polynomially or exponentially many scenarios to get a constant approximation. It is an interesting question whether there exists a constant approximation algorithm in the most general case that does not depend on the number of scenarios. Furthermore, we have tested our algorithms on a taxi dataset and showed that they improve significantly over the greedy approach, resulting in a reduction of the maximum wait time for taxi riders. 
{\color{black}Finally, we have compiled a list of open problems related to TSRMB and its variants in Appendix \ref{openproblems}, which we hope will guide future research on this topic.}

\vspace{2mm}
\color{black}
\noindent
{\bf Acknowledgements:} The authors thank the area editor, associate editor, and two anonymous referees whose comments helped us significantly strengthen both the technical results and the presentation of our paper.

\color{black}
\bibliographystyle{informs2014}
\setlength{\bibsep}{0.0pt}
\bibliography{bibliography}

\newpage
\begin{APPENDIX}{}

\section{Comparison to a Fully Adversarial Model} \label{newyear}

\red{In this section, we compare the TSRMB model to a {\em fully adversarial} model where we assume complete adversarial conditions without a predefined uncertainty set \(\mathcal{S}\) defining scenarios. The goal is to show that unlike the TSRMB model, the fully adversarial problem lacks a bounded approximation. {\color{black}Unlike the rest of the paper, which uses approximation guarantees relative to an optimal robust solution, this fully adversarial model is analyzed using competitive analysis, where the performance is measured against the hindsight optimum. }
Specifically, consider a scenario where there is a set of first-stage riders that need to be matched in the first stage with minimum average cost. Then, the adversary selects a second-stage scenario of riders that needs to be matched with the minimum bottleneck cost. If the adversary arbitrarily selects the second-stage scenario, not from a predefined uncertainty set, then the cost of any algorithm can be arbitrarily high compared to an offline algorithm that knows the scenario realization upfront.}

\red{To illustrate this, consider the following instance: we have a line with \(m+1\) points, labeled \(1, 2, \ldots, m+1\), each representing a driver. There is a distance of 1 between two consecutive points \(i\) and \(i+1\). Consider \(m\) first-stage riders located at points \(1\) through \(m\). The second-stage scenario is composed of a single rider, picked by an adversary. Suppose an algorithm match the initial set of first-stage riders, without any knowledge of the uncertainty set. This algorithm would leave one driver unused, say driver \(i\). If \(i \leq m/2\), let the adversary choose a second-stage rider located at point \(m+1\), so the second-stage cost is greater than \(m/2\). If \(i \geq m/2\), the adversary can choose a second-stage rider located at point \(1\), making the second-stage cost again greater than \(m/2\). Hence, the total cost of the algorithm can be \(\Omega(m)\) in the worst case. In contrast, an offline algorithm that knows the scenario realization upfront can optimize the matching. If the second-stage rider is located at point \(1\), then driver \(1\) is used to match it, resulting in a zero second-stage cost. The first-stage riders are matched such that driver \(i+1\) serves rider \(i\), resulting in an average cost of \(1\). Similarly, if the second-stage rider appears at point \(m+1\), each driver \(i\) serves rider \(i\),  resulting in a zero cost. This example demonstrates that achieving an online competitive ratio in the context of a fully adversarial two-stage model is not feasible.}

\section{NP-Hardness Proofs for TSRMB}\label{appendix:hardness}

\subsection{Proof of Theorem \ref{thm:hardness_robust}}

We start by presenting the 3-Dimensional Matching problem that we use in our reduction to show Theorem \ref{thm:hardness_robust}.  This problem is known to be  NP-complete \citep{karp2010reducibility}.

\noindent \textbf{3-Dimensional Matching (3-DM):} Given three sets $U$, $V$, and $W$ of equal cardinality $k$, and a subset $T$ of $U \times V \times W$, is there a subset $M$ of $T$ with $|M| = k$ such that whenever 
$(u, v, w)$ and $(u', v', w')$ are distinct triples in $M$, $u \neq u'$, $v\neq v'$, and $w \neq w'$ ? 


\vspace{2mm}
\noindent
\proof{\textit{Proof of Theorem \ref{thm:hardness_robust}}.}
\underline{\bf First part:} Consider an instance of the 3-DM problem. We  use it to construct (in polynomial time) an instance of TSRMB with 2 scenarios as follows:
\begin{itemize}
    \item Create two scenarios of size $k$: $S_1 = U$ and $S_2 = V$.
    \item Set $D = T$, every driver corresponds to a triple in $T$.
    \item For every $w \in W$, let $d_T(w)$ be the number of sets in $T$ that contain $w$. We create $d_T(w)-1$ first-stage riders, that are all copies of $w$. The total number of first-stage riders is therefore
    $|R_1| = |T|-k$.
    \item For $(w,e) \in R_1 \times D$, $d(w,e) = \left\{
    \begin{array}{ll}
       1 & \mbox{ if } w \in e \\
       3 & \mbox{ otherwise.}
    \end{array}
\right.$
    \item For $(u,e) \in S_1 \cup S_2 \times D$, $d(u,e) = \left\{
    \begin{array}{ll}
       1 & \mbox{ if } u \in e \\
       3 & \mbox{ otherwise.}
    \end{array}
\right.$
    \item For $u,v \in R_1\cup S_1\cup S_2, d(u,v) = \min\limits_{e\in D} d(u,e) + d(v,e)$.
    \item For $e,f \in D, d(e,f) = \min\limits_{u\in R_1\cup S_1\cup S_2} d(u,e) + d(u,f)$.
\end{itemize}
This choice of distances induces a metric graph. We claim that there exists a 3-dimensional matching if and only if there exists a solution to this TSRMB instance with total cost equal to $2$. Suppose that $M = \{e_1, \ldots, e_k\} \subset T$ is a 3-dimensional matching. Let $e_1, \ldots, e_k$ be the drivers that correspond to $M$ in the TSRMB instance. We show that by using $D_1 = D \setminus \{e_1, \ldots, e_k\}$ as a first-stage decision, we ensure that the total cost for the TSRMB instance is equal to 2. For any rider $u$ in scenario $S_1$, by definition of $M$, there exits a unique edge $e_i \in M$ that covers $u$. The corresponding driver $e_i \not\in D_1$ can be matched to $u$ with a distance equal to 1. Furthermore, $e_i$ cannot be matched to any other rider in $S_1$ with a cost \red{ at most 1}. Similarly, for any rider $v$ in scenario $S_2$, since there exits a unique edge $e_j \in M$ that covers $v$, the corresponding driver can be matched to $v$ with a cost of 1. The second-stage cost is therefore equal to 1. As for the first-stage cost, we know by definition of $M$, that every element $w \in W$ is covered exactly once. Therefore, for every $w \in W$, there exists $d_T(w) - 1$ edges that contain $w$ in $T \setminus M$. This means that every 1st stage rider can be matched to a driver in $D_1$ with a cost equal to 1. Hence the total cost of this two-stage matching is equal to 2.

Suppose now that there exists a solution to the TSRMB instance with a cost equal to 2. This means that the first and second-stage costs are both equal to 1. Let $M = \{e_1, \ldots, e_k\}$ be the set of drivers used in the second stage of this solution. We show that $M$ is a 3-dimensional matching. Let $e_i = (u, v, w)$ and $e_j = (u', v', w')$ be distinct triples in $M$. Since the second-stage cost is equal to $1$, the driver $e_i$ (resp. $e_j$) must be matched to $u$ (resp. $u'$) in $S_1$. Since we have exactly $n$ second-stage drivers and $k$ riders in $S_1$, this means that $e_i$ and $e_j$ have to be matched to different second-stage riders in $S_1$. Therefore we get $u' \neq u$. Similarly we see that $v' \neq v$. Assume now that $w = w'$,  this means that the TSRMB solution has used two drivers (triples) $e_i$ and $e_j$ that contain $w$ in the second stage. It is therefore impossible to match all the $d_T(w)-1$ copies of $w$ in the first stage with a cost equal to 1. Therefore $w \neq w'$. The above construction can be performed in polynomial time of the 3-DM input, and therefore shows that TSRMB with two scenarios is NP-hard.

\vspace{3mm}
\color{black}

\underline{\bf Second part:} Next, we show that TSRMB with three scenarios is NP-hard to approximate within a factor better than $3- \epsilon$ for any $\epsilon >0$. We employ a reduction from the 3-DM problem. Consider an instance of 3-DM with three sets $U, V, W$ of equal cardinality $k$ and a subset $T \subseteq U \times V \times W$. We construct an instance of TSRMB with three scenarios as follows:

\begin{itemize}
    \item Let $r $ be a very large integer. Create a \textit{dummy location}, where we place $r$ drivers and $r$ first-stage riders. This dummy location is very far from the rest of the graph.
    
    \item In addition to the $r$ drivers at the dummy location, create $|T|$ drivers, each corresponding to a triple in $T$. Denote this set of drivers by $\hat{D}$. Thus, the total number of drivers is $|D| = |\hat{D}| + r = |T| + r$.
    
    \item In addition to the $r$ first-stage riders at the dummy location, create $|T| - k$ first-stage riders, denoted by $\hat{R}_1$. Thus, the total number of first-stage riders is $|R_1| = |\hat{R}_1| + r = |T| - k + r$.
    
    \item Create three scenarios of second-stage riders, each of size $k$: $S_1 = U$, $S_2 = V$, and $S_3 = W$.\footnote{\color{black}Note in this TSRMB instance we have $|D|=|R_1|+k$, i.e., the surplus is zero.}
\end{itemize}

The dummy location, containing $r$ drivers and $r$ first-stage riders, is very far from the rest of the graph, i.e., very far from  $\hat{D}$, $\hat{R}_1$,   $S_1,S_2$ and $S_3$. This ensures that the $r$ drivers in the dummy location will match exclusively with the $r$ first-stage riders there, incurring a cost of $0$. The distances in the rest of the graph are defined as follows:

\begin{itemize}
    \item For $(w, e) \in \hat{R}_1 \times \hat{D}$, set $d(w, e) = 1$.
    \item For $(u, e) \in (S_1 \cup S_2 \cup S_3) \times \hat{D}$, set:
    $
    d(u, e) = 
    \begin{cases} 
       1 & \text{if } u \in e, \\
       3 & \text{otherwise}.
    \end{cases}
    $
    \item For $u, v \in \hat{R}_1 \cup S_1 \cup S_2 \cup S_3$, set $d(u, v) = \min_{e \in \hat{D}} \{ d(u, e) + d(v, e) \}$.
    \item For $e, f \in \hat{D}$, set $d(e, f) = \min_{u \in \hat{R}_1 \cup S_1 \cup S_2 \cup S_3} \{ d(u, e) + d(u, f) \}$.
\end{itemize}

\vspace{2mm}
This choice of distances induces a metric graph. First, as mentioned earlier, the $r$ first-stage riders in the dummy location are matched at a cost of $0$ since they are all at the same location as the $r$ drivers. The remaining first-stage riders in $\hat{R}_1$ are within a distance of $1$ from all drivers in $\hat{D}$. Thus, in any feasible solution, all riders in $\hat{R}_1$ will be matched at a distance cost of $1$. Therefore, the first-stage cost of any feasible solution to TSRMB, which is the average weight matching cost, is given by:
\[
\frac{1}{r + |\hat{R}_1|} ( |\hat{R}_1| \times 1+ r \times 0)= \frac{|\hat{R}_1|}{r + |\hat{R}_1|} = \frac{|T| - k}{r + |T| - k}.
\]
Fix any $\epsilon >0$ infinitely small. By choosing $r$ very large, we can ensure  that $\frac{|T| - k}{r + |T| - k} = \epsilon$ . Hence, the first-stage cost of any feasible solution is equal to $\epsilon$.

We show that there exists a 3-dimensional matching if and only if there is a TSRMB solution with a second-stage cost of $1$.  In fact,  if there exists a 3-dimensional matching $M = \{e_1, \dots, e_k\} \subseteq T$, use the drivers corresponding to $M$ to serve the second-stage riders. By the definition of a 3-dimensional matching, for each rider $u$ in $U=S_1$, there exists a distinct edge $e_i \in M$ that covers $u$ (i.e., there exists a distance driver $e_i$ such that the ditance $d(u,e_i)=1$). Similarly, this holds for riders in $S_2$ and $S_3$. Therefore, each second-stage scenario is matched with a bottleneck cost of $1$.

Conversely, suppose the second-stage cost of TSRMB is $1$. Consider the subset of drivers used in the second stage in an optimal solution, denoted by $M$. Note that $M$ has size $k$, and each driver in $M$ serves  three distinct riders: one from $S_1$, one from $S_2$, and one from $S_3$, all at a distance of $1$.  Hence, for each driver in $M$, the corresponding edge in $T$ covers a unique triple in $U \times V \times W$. Therefore, the subset of edges corresponding to $M$ forms a 3-dimensional matching in $T$.

If no 3-dimensional matching exists, at least one rider in the second stage must be matched with a driver at a distance of $3$. Consequently, the second-stage bottleneck cost of TSRMB has to be 3 and the total cost of TSRMB is $3 + \epsilon$. Therefore, we conclude that the total cost of TSRMB is $1+\epsilon$ if a 3-dimensional matching exists and $3+\epsilon$ if no 3-dimensional matching exists.

As the decision problem of finding a 3-dimensional matching is NP-complete, determining whether the cost of TSRMB is $1 + \epsilon$ or $3 + \epsilon$ is also NP-complete. Therefore, unless NP = P, it is NP-hard to approximate TSRMB within a factor better than $3 - \epsilon'$ for any $\epsilon' > 0$.

\hfill \Halmos
\endproof

\color{black}

\subsection{Proof of Theorem \ref{thm:hardness2}}

\red{In order to proof Theorem \ref{thm:hardness2}, the first part employs a reduction from the Clique problem, defined as follows:}

\red{
\noindent \textbf{Clique Problem:} Given an unweighted graph $G$ with edges $E$ and vertices $V$, a clique of size $p$ is a subset of $p$ vertices where each pair of vertices is adjacent. The Clique problem involves determining whether, for a given graph $G=(V,E)$ and an integer $p$, there exists a clique of size $p$. This problem is one of Karp's 21 NP-complete problems \citep{karp2010reducibility}.}

{\color{black} The second part  employs a reduction from the Clique Interdiction problem, defined as follows:

\noindent \textbf{Clique Interdiction Problem:} Given a graph $G=(V,E)$ and integers $q, p$, is there a subset $S \subseteq V$ of at most $q$ vertices such that the induced subgraph $G[V \setminus S]$ does not contain a clique of size $p$? This problem is known to be $\Sigma_2^P$-complete \citep{rutenburg1994propositional}.
}

\vspace{2mm}
\proof{\textit{Proof of Theorem \ref{thm:hardness2}.}}
\underline{\bf First part:}
Consider TSRMB under an implicit model of uncertainty where the sets of second-stage riders is given by the budgeted of uncertainty set:
$ \mathcal{S} = \{ S \subset R_2 \mid |S | \leq k \}.$ 
Consider an instance of the Clique problem with a graph $G=(V,E)$ and an integer $p$. We construct the corresponding TSRMB instance as follows. Let $k= {p \choose 2}$ be the maximum size of each scenario in $\mathcal{S}$.  
\begin{itemize}
    \item The set of first-stage riders $R_1=\emptyset$.
    \item The set $R_2$ consists of $|E|$ riders, each corresponding to an edge in $G$. Every scenario $S \subseteq R_2$ has a size at most $k$.
    \item The set of drivers $D$ comprises $|V| + k - p - 1$ drivers, divided into two categories: $D = A \cup B$. The set $A$ contains $|V|$ drivers corresponding to the vertices in $V$, and set $B$ includes the remaining $k - p - 1$ drivers. 
\end{itemize}
The distances between riders and drivers are defined as follows:
\begin{itemize}
    \item For $(i,e) \in B \times R_2$, $d(i,e) = 1$.
    \item For $(i,e) \in A \times R_2$, $d(i,e) = \left\{
    \begin{array}{ll}
       1 & \text{if vertex } i \text{ is {\color{black} incident} to edge } e \text{ in } G, \\
       3 & \text{otherwise.}
    \end{array}
\right.$
\item For $i, i' \in D,$ $ d(i, i') = \min_{e \in R_2} \{ d(i,e) + d(i',e) \}$.
    \item For $e, e' \in R_2,$ $ d(e, e') = \min_{i \in D} \{ d(i,e) + d(i,e') \}$.
\end{itemize}
This configuration of distances results in a metric graph. Note that all edges in the bipartite graph between riders $R_2$ and drivers $D$ have a distance cost of either $3$ or $1$. The first-stage cost is $0$ because $R_1=\emptyset$, so the objective value of TSRMB corresponds to the second-stage cost. Given that the cost of the objective function in the second stage is a bottleneck matching, the optimal objective value of problem TSRMB is either $1$ or $3$. In the next Lemma, we will demonstrate that the objective value is $3$ if and only if there exists a clique of size $p$ in the graph $G$.
\begin{lemma} \label{lem:cabo}
    The objective value of TSRMB is $3$ if and only if there exists a clique of size $p$ in the graph $G$.
\end{lemma}

\proof{\textit{Proof of Lemma \ref{lem:cabo}.}}
Suppose there exists a clique $C$ of size $p$ in $G=(V,E)$. All vertices in $C$ are adjacent, forming ${p \choose 2} = k$ edges. Let $S \subseteq R_2$ represent riders corresponding to these $k$ edges. Note that $S \in \mathcal{S}$ is a feasible scenario for TSRMB as its size equals $k$. These $k$ edges are {\color{black} incident} only to vertices of the clique $C$. Hence, any vertex in $V \setminus C$ is not {\color{black} incident} to any edge in $S$.
This implies that the distance between any rider $e \in S$ and a driver in $V \setminus C$ is $3$. 
The size of $B$ is $k - p - 1$, and the size of scenario $S$ is $k$, so at least $p + 1$ of the riders in $S$ need to be matched with drivers from $A$. Note that $A= C  \cup \{ V \setminus C\}$ and $|C|=p$. So, at least one rider in $S$ needs to {\color{black} be} matched to a driver in $V \setminus C$. We showed that {\color{black}the distance} between any rider in $S$ and any driver in $V \setminus C $ is equal to 3. Thus, any matching between $S$ and drivers from $R_2$ will have at least one edge with a distance of $3$. Consequently, the bottleneck cost of matching $S$ is $\text{cost}_2(D, S) = 3$. We have established the existence of one scenario with a cost of $3$. Therefore, the optimal objective value of TSRMB is $3$.

\red{
Conversely, suppose the optimal objective value 
is $3$. 
This means that if we consider only the edges of distance $1$ and remove the edges of distance $3$ in the bipartite graph between drivers and riders, we cannot find a perfect matching for every scenario $S \in \mathcal{S}$.  Consider $U$ the unweighted bipartite graph between drivers and riders, where we keep only the edges of distance $1$ and remove those with distance $3$. By Hall's Theorem, this implies the existence of $S \subseteq R_2$ such that $|S| \leq k$ and $|S| > |N(S)|$, where $N(S)$ denotes the neighbors of $S$ in the graph $U$. Let $\mathcal{E}$ be the set of edges in $G$ corresponding to the scenario $S$. Take a rider in $S$; it has a distance of $1$ with all drivers in $B$ and with the two drivers corresponding to vertices in $A$ that are {\color{black} incident} to the edge that represents the rider. {\color{black} Hence, the rider has $|B| + 2$ neighbors in the graph $U$} where $|B| + 2 = k - p + 1 = \frac{p(p-1)}{2} - p + 1 = {p-1 \choose 2}$. Therefore, $ |N(S)| \geq {p-1 \choose 2}$, implying that $|S| \geq {p-1 \choose 2} + 
1$. Note that, in general, the maximum number of edges that can be formed with $p - 1$ vertices is ${p-1 \choose 2}$. Since $|S| \geq {p-1 \choose 2} + 1$, it follows that the number of vertices that are used to form the edges in $\mathcal{E}$, is at least $p$. Therefore, the size of the neighborhood of $S$ in $U$ is at least $p + ( k - p - 1) = k - 1$. Hence, $|S| > |N(S)| \geq k - 1$, which implies $|S| \geq k$. Given that $|S| \leq k$, we conclude that $|S| = k$ and $N(S)=k-1$.  In the graph $U$, the neighbors of $S$ belonging to the set $C$ are of size $k-p-1$, so the  neighbors of $S$ belonging to the set $A$ are of size $p$. Recall, by definition, that the neighbors of $S$ belonging to the set $A$ correspond to vertices from $V$ that are {\color{black} incident} in the graph $G$ to the edges in  $\cal E$.
This means that we have $p$ vertices in $V$ that  form $k = {p \choose 2}$ edges. Consequently, there is a clique of size $p$ in the graph $G$.}

\hfill \Halmos
\endproof

To conclude, as the decision problem of finding a clique of a given size is NP-complete, determining whether the cost TSRMB is $1$ or $3$ is also NP-complete. Hence, unless NP=P, it is NP-hard to approximate TSRMB within a factor better than $3 - \epsilon$ for any $\epsilon > 0$.

\vspace{3mm}

{\color{black}
\underline{\bf Second part:}
Now we prove that TSRMB with budgeted uncertainty is $\Sigma_2^P$-hard. Consider an instance of the Clique Interdiction problem, where we are given a graph $G = (V, E)$ and integers $q$ and $p$. The Clique Interdiction problem asks whether there exists a subset $S \subseteq V$ of at most $q$ vertices such that the induced subgraph $G[V \setminus S]$ contains no clique of size $p$.

We construct a corresponding instance of TSRMB with budgeted uncertainty such that the optimal TSRMB cost is exactly 2
 if and only if the answer to the Clique Interdiction problem is YES.
Let $k = {p \choose 2}$ be the size of the largest scenario in the budgeted uncertainty set.
We define the TSRMB instance as follows:

\begin{itemize}
    \item The set of first-stage riders $R_1$ contains $|V| - q$ riders.
    \item The set of second-stage riders $R_2$ contains $|E|$ riders, one for each edge in $G$. Each scenario $S \subseteq R_2$ has size at most $k = {p \choose 2}$.
    \item The set of drivers $D$ comprises $2|V| + k - p - 1$ drivers, divided into three disjoint subsets:
    \begin{itemize}
        \item $A$, containing $|V|$ drivers corresponding to the vertices in $V$.
        \item $B$, also containing $|V|$ drivers corresponding to the vertices in $V$.
        \item $C$, a set of $k - p - 1$ dummy drivers.
    \end{itemize}
\end{itemize}

The distances between riders and drivers are defined as follows:
\begin{itemize}
    \item For $(i,e) \in A \times R_1$, $d(i,e) = 1$.
    \item For $(i,e) \in \{B \cup C\} \times R_1$, $d(i,e) = 3$.
    \item For $(i,e) \in \{A \cup B\} \times R_2$, $d(i,e) = \left\{
    \begin{array}{ll}
       1 & \text{if vertex } i \text{ is incident to edge } e \text{ in } G, \\
       3 & \text{otherwise.}
    \end{array}
    \right.$
    \item For $(i,e) \in C \times R_2$, $d(i,e) = 1$.  
    \item For $i, i' \in D$, $d(i, i') = \min_{e \in R_1 \cup R_2} \{ d(i,e) + d(i',e) \}$.
    \item For $e, e' \in R_1 \cup R_2$, $d(e, e') = \min_{i \in D} \{ d(i,e) + d(i,e') \}$.
\end{itemize}
The above distances define a metric graph. 
We now prove the following key equivalence:

\begin{lemma} \label{lem:interdiction}
    The TSRMB instance has cost exactly 2 if and only if the answer to the Clique Interdiction problem is YES.
\end{lemma}


\proof{\textit{Proof of Lemma \ref{lem:interdiction}.}}
Suppose the answer to the Clique Interdiction problem is YES. That is, there exists a subset of vertices \( X \subseteq V \), with \( |X| = q \), such that the induced subgraph \( G[V \setminus X] \) contains no clique of size \( p \). Let \( X \subseteq A \) be the drivers in \( A \) corresponding to the vertices in \( X \), and define \( D_1 = A \setminus X \). Since \( |R_1| = |V| - q = |D_1| \), this yields a valid first-stage matching. All edges between \( A \) and \( R_1 \) have distance 1, so the first-stage cost is 1.

We now argue that the second-stage cost is also 1. Suppose for contradiction that the second-stage bottleneck cost is 3. This means that there exists a scenario \( S' \subseteq R_2 \), with \( |S'| \leq k = {p \choose 2} \), such that no perfect matching from \( S' \) to the remaining drivers (i.e., \( B \cup C \cup X \)) exists using only distance-1 edges.
Let \( U \) be the unweighted bipartite graph induced by distance-1 edges between riders in \( S' \) and the drivers in \( B \cup C \cup X \). By Hall's theorem, there exists a subset \( S \subseteq S' \) such that \( |N(S)| < |S| \). Consider any rider \( e \in S \); it corresponds to an edge in \( G \), say \( (i, j) \). Then \( e \) is connected by distance-1 edges to:
\begin{itemize}
    \item the two drivers in \( B \) corresponding to \( i \) and \( j \),
    \item the drivers in \( X \subseteq A \) corresponding to \( i \) if $i \in X$ and/or \( j \) if $j \in X$,
    \item all drivers in \( C \) (since these are connected to every rider in \( R_2 \) with distance 1).
\end{itemize}




Hence, each rider in $S$ has at least $|C| + 2$ neighbors in the graph $U$ where
$
|C| + 2 = k - p + 1 = \frac{p(p-1)}{2} - p + 1 = {p-1 \choose 2}.
$
Therefore, $|N(S)| \geq {p-1 \choose 2}$, which implies that 
$
|S| \geq {p-1 \choose 2} + 1.
$
Note that, in general, the maximum number of edges that can be formed with $p - 1$ vertices is ${p-1 \choose 2}$. Since $|S| \geq {p-1 \choose 2} + 1$, it follows that the number of vertices used to form the edges in $S$ must be at least $p$. Thus, the size of the neighborhood of $S$ in $U$ is at least 
$
p + (k - p - 1) = k - 1,
$
where the $p$ neighbors come from the set $B$ and $k - p - 1$ from the set $C$. So we have $|N(S)| \geq k-1$. We also note that the neighbors of $S$ in $U$ cannot include any node from the set $X$, because otherwise we would have $|N(S)| \geq k$, which would imply   $|S| \geq k + 1$, contradicting the fact that $|S| \leq k$. Therefore, we conclude that $|N(S)| = k - 1$ and $|S| = k$.
In the graph $U$, the neighbors of $S$ in the set $C$ have size $k - p - 1$, so the neighbors of $S$ in the set $B$ must have size $p$. Recall that the neighbors of $S$ in $B$ correspond to vertices from $V$ that are incident in the graph $G$ to the edges in $S$. This means we have $p$ vertices in $V$ that form 
$
k = {p \choose 2}
$
edges. Moreover, since $S$ has no neighbors in $X$, it follows that the vertices used in $S$ have no intersection with $X$. Consequently, there exists a clique of size $p$ in the graph $G[V \setminus X]$, contradicting our assumption that no $p$-clique exists in $G[V \setminus X]$.
Therefore, the second-stage cost of our solution is $1$, and the total cost is $2$. It is easy to see that any feasible solution to TSRMB incurs a total cost of at least $2$: the first-stage cost is at least $1$ (since all first-stage edges have cost at least $1$), and similarly, the second-stage cost is also at least $1$. Hence, we conclude that the optimal cost of TSRMB is $2$.

\medskip
Conversely, suppose the optimal cost of TSRMB is 2. Then both the first-stage and second-stage costs must be 1. Since only drivers in \( A \) are connected to \( R_1 \) via distance-1 edges, the first-stage matching must use only drivers in \( A \). Let \( D_1 \subseteq A \) be the drivers used in the first stage, with \( |D_1| = |R_1| = |V| - q \), and define \( X = A \setminus D_1 \). Then \( |X| = q \), and \( X \) is our candidate clique interdiction set.
Suppose for the sake of contradiction that \( G[V \setminus X] \) contains a clique \( K \) of size \( p \). Then \( K \) induces \( k = {p \choose 2} \) edges, and we define \( S \subseteq R_2 \) as the set of riders corresponding to these edges. This is a valid scenario.
Each edge in \( S \) connects only to vertices in \( K \subseteq V \setminus X \). The drivers in \( A \) corresponding to \( K \) are not available in the second stage since they were used in the first-stage matching. Thus, the riders in \( S \) must be matched to drivers in \( B \cup C \cup (A \setminus D_1) = B \cup C \cup X \).
Since \( |C| = k - p - 1 \), at most \( k - p - 1 \) riders in \( S \) can be matched to \( C \) with cost 1. The remaining \( p + 1 \) riders must be matched to drivers in \( B \cup X \). But \( |K| = p \), so at most \( p \) drivers in \( B \) are incident to edges in \( S \), and the vertices in \( X \) are not incident to any edge in \( S \). Thus, at least one (edge) rider must be matched with a driver (vertex) that is not incident to the edge, thus it has a cost distance 3, contradicting the assumption that the second-stage bottleneck cost is 1.
Hence, \( G[V \setminus X] \) contains no clique of size \( p \), and the answer to the Clique Interdiction problem is YES.
\hfill \Halmos
\endproof

We conclude that deciding whether the cost of TSRMB is exactly 2 is equivalent to solving the Clique Interdiction problem, which is $\Sigma_2^P$-complete. Hence, TSRMB with budgeted uncertainty is $\Sigma_2^P$-hard.
\hfill \Halmos
\endproof

}


\color{black}

\subsection{Proof of Theorem \ref{thm:hardness3}}

In order to show Theorem \ref{thm:hardness3}, we employ a reduction from the set cover problem, defined as follows:

\noindent\textbf{Set Cover Problem:} Given a set of elements $\mathcal{U} = \{1,2,...,n\}$ (called the universe), a collection $S_1, \ldots, S_p$ of $p$ sets whose union equals the universe and an integer $q$.\\
Question: Is there a set $C \subset \{1, \ldots, p\}$ such that $|C| \leq q$ and $\bigcup\limits_{i \in C} S_i = \mathcal{U}$ ?
This decision problem is known to be NP-complete \citep{feige1998threshold}.

\vspace{2mm}
\proof{\textit{Proof of Theorem \ref{thm:hardness3}.}}
Consider an instance of the decision problem of set cover with $n$ elements, $p$ sets and an integer $q$. We construct an instance of the TSRMB problem as follows.

\begin{itemize}
    \item Let $r $ be a very large integer. Create a \textit{dummy location}, where we place $r$ drivers and $r$ first-stage riders. This dummy location is very far from the rest of the graph.
    \item In addition to the $r$ drivers at the dummy location, create $p$ drivers $D = \{1, \ldots, p\}$. For each $j \in \{1, \ldots, p\}$, driver $j$ corresponds to  set $S_j$. Thus, the total number of drivers is $|D| =  r +p $.
    \item  In addition to the $r$ first-stage riders at the dummy location, create $p - q$ first-stage riders, denoted by $\hat{R}_1= \{1, \ldots, p-q\}$. Thus, the total number of first-stage riders is $|R_1| = |\hat{R}_1| + r = p - q + r$.
    \item Create $n$ second-stage riders, $R_2 = \{1, \ldots, n\}$ and set $\mathcal{S} = \{ \{1\}, \ldots, \{n\}\}$. Every second-stage scenario is of size $k=1$.
\end{itemize}

The dummy location, containing $r$ drivers and $r$ first-stage riders, is very far from the rest of the graph, i.e., very far from  $\hat{D}$, $\hat{R}_1$ and $R_2$. This ensures that the $r$ drivers in the dummy location will match exclusively with the $r$ first-stage riders there, incurring a cost of $0$. The distances in the rest of the graph are defined as follows:

\begin{itemize}
    \item  For $(i,j) \in \hat{R}_1 \times \hat{D}$, $d(i,j) = 1$.
    \item For $(i,j) \in R_2 \times \hat{D}$, $d(i,j) = \left\{
    \begin{array}{ll}
       1 & \mbox{ if } i \in S_j \\
       3 & \mbox{ otherwise.}
    \end{array}
\right.$
\item For $i,i' \in \hat{R}_1\cup R_2,$ $ d(i,i') = \min\limits_{j\in \hat{D}} d(i,j) + d(i',j)$.
    \item For $j,j' \in \hat{D},$ $ d(j,j') = \min\limits_{i\in \hat{R}_1\cup R_2} d(i,j) + d(i,j')$.
\end{itemize}

This choice of distances induces a metric graph. First, as mentioned earlier, the $r$ first-stage riders in the dummy location are matched at a cost of $0$ since they are all at the same location as the $r$ drivers. The remaining first-stage riders in $\hat{R}_1$ are within a distance of $1$ from all drivers in $\hat{D}$. Thus, in any feasible solution, all riders in $\hat{R}_1$ will be matched at a distance cost of $1$. Therefore, the first-stage cost of any feasible solution to TSRMB, which is the average weight matching cost, is given by:
\[
\frac{1}{r + |\hat{R}_1|} ( |\hat{R}_1| \times 1+ r \times 0)= \frac{|\hat{R}_1|}{r + |\hat{R}_1|} = \frac{p - q}{r + p-q}.
\]
Fix any $\epsilon >0$ infinitely small. By choosing $r$ very large, we can ensure  that $\frac{p - q}{r + p - q} = \epsilon$ . Hence, the first-stage cost of any feasible solution is equal to $\epsilon$.

Next, we show that a set cover of size $ q$ exists if and only if there is a TSRMB solution with second-stage cost equal to $1$. Suppose without loss of generality that $S_1, \ldots, S_q$ is a set cover. Then by using the drivers $\{1, \ldots, q\}$ in the second stage, we ensure that every scenario is matched with a cost of $1$. This implies the existence of a solution with second-stage cost equal to $1$. Now suppose there is a solution to the TSRMB problem with second-stage cost equal to 1. Let $D_2$ be the set of second-stage drivers of this solution, then we have $|D_2| = q$. We claim that the sets corresponding to drivers in $D_2$ form a set cover. In fact, since the second-stage cost of the TSRMB solution is equal to 1. this means that for every scenario $i \in \{1, \ldots, n\}$, there is a driver $j \in D_2$ within a distance  1 from $i$. Therefore $i \in S_j$ and $\{S_j :  j \in D_2\}$ is a set cover of size $q$.

If no set cover of size $q$ exists, at least one rider in the second stage must be matched with a driver at a distance of $3$. Consequently, the second-stage bottleneck cost of TSRMB has to be 3 and the total cost of TSRMB is $3 + \epsilon$. Therefore, we conclude that the total cost of TSRMB is $1+\epsilon$ if a set cover of size $q$ exists and $3+\epsilon$ if no set cover of size $q$ exists.

As the decision problem of finding a set cover of a given size is NP-complete, determining whether the cost of TSRMB is $1 + \epsilon$ or $3 + \epsilon$ is also NP-complete. Therefore, unless NP = P, it is NP-hard to approximate TSRMB within a factor better than $3 - \epsilon'$ for any $\epsilon' > 0$.
\hfill \Halmos
\endproof

\color{black}

{\color{black}

\section{Connection between    Chromatic \(k\)-Supplier Problem and TSRMB with Explicit Scenarios} \label{apx:tamim}

In this section, we formalize the connection between the \emph{chromatic \(k\)-supplier problem} and TSRMB with \(p\) explicit scenarios. In particular, we present a reduction from the chromatic \(k\)-supplier problem to TSRMB, showing that any constant-factor approximation algorithm for TSRMB would imply a constant-factor approximation for the chromatic \(k\)-supplier problem, under the assumption that the ratio between the maximum and minimum  distances in the metric is bounded by a polynomial in the input size.
The chromatic \(k\)-supplier problem was introduced recently by~\cite{goyal2023tight}, who listed it as an open problem. Currently, no constant-factor approximation guarantees are known for this problem. This reduction provides further evidence that designing constant-factor approximation algorithms for TSRMB is a challenging task, as TSRMB is at least as hard as the chromatic \(k\)-supplier problem. 
Let us define the chromatic \(k\)-supplier problem.

\vspace{2mm}
\noindent
{\bf Chromatic \(k\)-Supplier Problem (Soft Assignment Version)}.
We consider the soft-assignment version of the chromatic \(k\)-supplier problem, defined by ~\cite{goyal2023tight}, in which suppliers can be assigned to multiple clients.
The input consists of a set of clients and suppliers located in a metric space. Each client belongs to a unique color class, and there are \(p\) color classes in total. A set of \(m\) potential supplier locations is given, along with an integer \(k\) specifying how many suppliers can be opened (possibly with repetitions, i.e., the same location may be selected multiple times).
The goal is to open at most \(k\) suppliers and assign each client to one of the opened suppliers such that the \emph{chromatic constraint} is satisfied: no two clients of the same color are assigned to the same supplier. The objective is to minimize the maximum distance between any client and the supplier to which they are assigned.

\vspace{2mm}
\noindent
{\bf Reduction to TSRMB with Explicit Scenarios.}
 Given a chromatic \(k\)-supplier instance with \(m\) supplier locations and $n$ clients. Define the aspect ratio as $D_{\max}/D_{\min}$ where $D_{\max}$ is the maximum distance in the metric and $D_{\min}$ is the minimum non-zero distance in the metric. We assume that $D_{\max}/D_{\min} $ is upper bounded by a polynomial in the input, i.e., a polynomial of $n$ and $m$. We construct the corresponding TSRMB instance as follows.

Let \(r\) be a large integer, chosen such that \(r \gg 1\) (will be specified later). For each supplier location \(j \in [m]\), we introduce:
\begin{itemize}
    \item \(r\) drivers, all located at \(j\);
    \item \(r - 1\) first-stage riders, also located at \(j\).
\end{itemize}

Additionally, choose an arbitrary supplier location \(j_0 \in [m]\), and place \(m - k\) additional first-stage riders at location \(j_0\). Thus, the total number of first-stage riders is \(m(r - 1) + (m - k) = mr - k\), and the total number of drivers is \(mr\).

For each color \(i \in [p]\), introduce a second-stage scenario \(S_i\) consisting of all clients (riders) of color \(i\). These form the explicit scenario set \(\mathcal{S} = \{S_1, \dots, S_p\}\).

We now explain why the above construction reduces the chromatic \(k\)-supplier problem to an instance of TSRMB.
In the constructed TSRMB instance, we have a total of \(mr\) drivers and \(mr - k\) first-stage riders. To serve all first-stage riders, we will use exactly \(mr - k\) drivers, leaving \(k\) drivers unused for the second stage.
Due to the nature of the TSRMB objective, which combines an average-cost matching in the first stage and a bottleneck cost matching in the second stage, the \(k\) drivers reserved for the second stage can be any $k$ drivers from the full pool. Among the \(mr - k\) riders  in the first stage, most will be matched locally to a driver at the same location, incurring zero cost. The only exceptions are the \(m - k\) riders added at the extra location \(j_0\), which do not have enough co-located drivers; and up to \(k\) additional riders, where the corresponding local drivers were reserved for the second stage. Thus, at most \(m\) riders may be matched to non-local drivers. The first-stage average matching cost is therefore bounded by
$
m \cdot D_{\max} /(mr - k).
$
By taking \(r\) sufficiently large, this quantity can be made arbitrarily small. In particular, for $\epsilon>0$, we  want $r$ such that $m \cdot D_{\max} /(mr - k) \leq \epsilon \cdot D_{\min}$, so we choose  $r = \frac{k}{m}+ \frac{1}{\epsilon}\frac{D_{\max}}{D_{\min} }$. Note that $r$ is polynomial in the input of the problem because of our assumption on the aspect ratio, so our reduction is a approximation preserving polynomial time reduction.
As a result, the first-stage cost of TSRMB is bounded by $\epsilon D_{\min} \leq \epsilon \cdot OPT_2$, where $OPT_2$ denotes the optimal second-stage cost. This implies that the total TSRMB objective is dominated by the second-stage term.
Hence, the TSRMB instance effectively reduces to the problem of selecting \(k\) drivers (i.e., \(k\) supplier locations) to be reserved for the second stage, in order to minimize the worst-case bottleneck matching cost. In each scenario \(S_i\), which corresponds to all clients of color \(i\), each rider must be matched to one of the \(k\) available drivers. Moreover, since these riders all share the same color, the TSRMB constraint that no two riders in a scenario are matched to the same driver enforces the chromatic constraint of the original problem: no two clients of the same color may be assigned to the same supplier.

Therefore, solving the TSRMB instance yields a feasible solution to the chromatic \(k\)-supplier problem: the \(k\) drivers reserved for the second stage define the set of supplier locations, and the bottleneck cost incurred in the worst-case scenario defines the chromatic clustering cost. Let $OPT$ denote the optimal value of the chromatic \(k\)-supplier instance, and $OPT_1 + OPT_2$ be the optimal value of the corresponding TSRMB instance. Then,
\[
OPT_2 \leq OPT_1 + OPT_2 \leq \epsilon D_{\min} + OPT \leq (1 + \epsilon) \cdot OPT,
\]
since $OPT \geq D_{\min}$ and the first-stage cost is at most $\epsilon D_{\min}$. Thus, the second-stage cost of the TSRMB solution, which defines the clustering cost, is at most a \((1 + \epsilon)\)-factor approximation to the optimal chromatic \(k\)-supplier cost.
In conclusion, any \(\rho\)-approximate solution to TSRMB yields a \((1 + \epsilon)\rho\)-approximate solution to chromatic \(k\)-supplier. Therefore, an \(O(1)\)-approximation algorithm for TSRMB implies an \(O(1)\)-approximation algorithm for the chromatic \(k\)-supplier problem.

}

\section{Proofs of Section \ref{subsection:mathilde} }\label{appendix:implicit}


\proof{\textit{Proof of Claim \ref{lemma:ball}}.} $ $
First, in the optimal solution of the original problem, 
    $R_1$ is matched to a subset $D_1^*$ of drivers. The scenario $S_1$ is matched to a set of drivers $D_{S_1}$ where $D_1^* \cap D_{S_1} = \emptyset$. Let $D_o$ be the set of drivers that are matched to $o_1, \ldots, o_j^*$ in a scenario that contains $o_1, \ldots, o_j^*$. It is clear that $D_1^* \cap D_o = \emptyset.$  We claim that $D_o \cap D_{S_1} = \emptyset$. In fact, suppose there is a driver $\rho \in D_o \cap D_{S_1}$. This implies the existence of some $o_j$ with $j \leq j^*$ and some rider $r \in S_1$ such that $d(\rho, o_j) \leq OPT_2$ and $d(\rho, r) \leq OPT_2$. But then $d(\delta, o_j) \leq d(\delta,r) + d(\rho, r) + d(\rho, o_j) \leq 3OPT_2$ which contradicts the fact the $o_j$ is an outlier. Therefore $D_o \cap D_{S_1} = \emptyset$. We show that $D_1^*$ is a feasible first-stage solution to the single scenario problem of $S_1 \cup \{o_1, \ldots o_j^*\}$ with a cost {\color{black}at most} $OPT_1 + OPT_2$. In fact, $D_1^*$ can be matched to $R_1$ with a cost {\color{black}at most} $OPT_1$, $D_{S_1}$ to $S_1$ and $D_o$ to  $\{o_1, \ldots, o_j^*\}$ with a cost {\color{black}at most} $OPT_2$. Therefore $\Omega_{j^*} + \Delta_{j^*} \leq OPT_1 + OPT_2.$

Second, recall that $cost_1\big(D_1(j^*),R_1\big)  = \Omega_{j^*} $. Consider a scenario $S$ and a rider $r \in S$. Let $\mathcal{B'}$ be the set of the $n-\ell$ closest second-stage riders to $\delta$. Let $D_{S_1}(j^*)$ be set of second-stage drivers matched to $S_1$ in the single scenario problem for scenario $S_1 \cup \{o_1, \ldots, o_{j^*}\}$. Let $D_o(j^*)$ be the set of second-stage drivers matched to $\{o_1, \ldots, o_{j^*}\}$ in the single scenario problem for scenario $S_1 \cup \{o_1, \ldots, o_{j^*}\}$. Recall that the second-stage cost for this single scenario problem is $\Delta_{j^*}$. We distinguish three cases:
    \begin{enumerate}
        \item If $r \in \mathcal{B'}$, then by Lemma \ref{lemma:center}, $r$ is connected to every driver in $D_{S_1}(j^*)$ within a distance {\color{black}at most} $\Delta_{j^*} + 4OPT_2$.
         \item If $r \in \{o_{j^*+1}, \ldots, o_{\ell}\}$, then $r$ is connected to every driver in $D_{S_1}(j^*)$ within a distance {\color{black}at most}
         $3OPT_2$ + $OPT_2 + \Delta_j^*$.
        \item If $r \in \{o_1, \ldots, o_{j^*}\}$ (i.e., $r$ an outlier), then $r$ can be matched to a different driver in $D_o(j^*)$ within a distance  {\color{black}at most} $OPT_2$.
    \end{enumerate}
    This means that in every case, we can match $r$ to a driver in $D\setminus D_1(j^*)$ with a cost {\color{black}at most} $4OPT_2 + \Delta_{j^*}$. This implies that 
    $$  \max\limits_{S \in \mathcal{S}} cost_2\big(D\setminus D_1(j^*), S\big) \leq 4OPT_2 + \Delta_{j^*}$$
    and therefore $$ \Omega_{j^*} + \max\limits_{S \in \mathcal{S}} cost_2\big(D\setminus D_1(j^*), S\big)  \leq \Omega_{j^*} + \Delta_{j^*} + 4OPT_2 \leq OPT_1 + 5OPT_2. $$
\hfill \Halmos
\endproof

\proof{\textit{Proof of Claim \ref{lemma:betaj}}.} Let $\alpha_j$ be the second-stage cost of $D_1(j)$ on the TSRBM instance with scenarios $S_1$ and $S_2$. Formally, $\alpha_j = \max\limits_{S \in \{S_1,S_2\}} cost_2\big(D\setminus D_1(j), S\big)$.  Therefore $\beta_j = \Omega_j + \alpha_j$. Let's consider the two sets
\begin{alignat*}{3}
    O_1 & =  \{ r \in \{o_1, \ldots, o_{\ell}\} \ | \ d(r,\delta) > 2\alpha_j + OPT_2\}.\\
    O_2 & =  \{o_1, \ldots, o_{\ell}\} \setminus O_1.
\end{alignat*}

Consider  $D_1(j)$ as a first-stage decision to TSRMB with scenarios $S_1$ and $S_2$. Let $\Tilde{D}_1 \subset D \setminus D_1(j)$ be the set of drivers that are matched to $O_1$ when the scenario $S_2 = \{o_1, \ldots, o_{\ell}\}$ is realized. Similarly, let $\Tilde{D}_2 \subset D \setminus D_1(j)$ be the drivers matched to scenario $S_1$. We claim that $\Tilde{D}_1 \cap \Tilde{D}_2 = \emptyset$. Suppose that there exists some driver $\rho \in \Tilde{D}_1 \cap \Tilde{D}_2$, this implies the existence of some $o \in O_1$ and $r \in S_1$ such that $d(\rho, o) \leq \alpha_j$ and $d(\rho, r) \leq \alpha_j$. And since $d(r,\delta)\leq OPT_2$ by definition of $\delta$ we would have
\begin{equation*}
d(o, \delta) \leq d(\rho, o) + d(\rho, r) + d(r,\delta) \leq 2\alpha_j + OPT_2,\end{equation*}
which contradicts the definition of $O_1$. Therefore $\Tilde{D}_1 \cap \Tilde{D}_2 = \emptyset$.

Now consider a scenario $S \in \mathcal{S}$. The riders of $S\cap O_1$ can be matched to $\Tilde{D}_1$ with a bottleneck cost at most $\alpha_j$. Recall that by Lemma \ref{lemma:center}, any rider in $R_2 \setminus  \{o_1, \ldots, o_{\ell}\}$ is within a distance at most $4OPT_2$ from any rider in $S_1$. The riders $r \in S \setminus \{o_1, \ldots, o_{\ell}\}$ can therefore be matched to any driver  $\rho \in \Tilde{D}_2$ within a distance at most 
\begin{equation*}
    d(r, \rho) \leq d(r, S_1) + d(S_1, \rho) \leq 4 OPT_2 + \alpha_j.
\end{equation*}
As for riders $r \in S \cap O_2$, they can also be matched to any driver $\rho$ of $\Tilde{D}_2$ within a distance at most 
\begin{equation*}
    d(r, \rho) \leq d(r, \delta) + d(\delta, S_1) + d(S_1, \rho) \leq 2\alpha_j + OPT_2 + OPT_2 +  \alpha_j = 3\alpha_j + 2OPT_2.
\end{equation*}
Therefore we can bound the second-stage cost \begin{equation*}
    \max\limits_{S \in \mathcal{S}} cost_2\big(D\setminus D_1(j), S\big) \leq \max\{\alpha_j + 4OPT_2, 3\alpha_j + 2OPT_2\} 
\end{equation*}
and we get that
\begin{equation*}
    cost_1\big(D_1(j),R_1\big) + \max\limits_{S \in \mathcal{S}} cost_2\big(D\setminus D_1(j), S\big) \leq \max\{\beta_j + 4OPT_2, \ 3\beta_j + 2OPT_2\}.
\end{equation*}
The other inequality $\beta_j \leq cost_1\big(D_1(j),R_1\big) + \max\limits_{S \in \mathcal{S}} cost_2\big(D\setminus D_1(j)\big)$ is trivial. 
\hfill \Halmos
\endproof

\color{black}

\section{Arbitrary Surplus with $k=1$}\label{appx:kone}

\proof{\textit{Proof of Theorem \ref{thm:koneimplicit}}.}
Fix an optimal solution to the TSRMB problem, with the first-stage cost denoted by $OPT_1$ and the second-stage cost denoted by $OPT_2$.
In Algorithm~\ref{alg:k1}, we constructed a maximal set $S$ of second-stage riders from $R_2$ such that the riders in $S$ are at a distance greater than $2 OPT_2$ from one another, and for each rider in $R_2 \setminus S$, there exists a rider in $S$ within a distance of at most $2 OPT_2$.

Each rider in $S$ constitutes a second-stage scenario in our TSRMB problem with $k=1$. The key observation is that these second-stage riders in $S$ cannot be served by the the same drivers in the optimal solution. That is, the optimal solution to the TSRMB problem must assign a distinct driver to each rider in $S$ to be used in the second stage. Otherwise, if two riders are served by the same driver, then by the triangle inequality, the distance between them would be less than or equal to $2 OPT_2$, which contradicts the definition of $S$.

Now, consider the problem TSRMB-1-Scenario($R_1, S, D$). We claim that the optimal solution $D_1$ to this problem has a cost (in this one scenario problem) no greater than $OPT_1 + OPT_2$. To see this, we construct a feasible  solution with cost $OPT_1 + OPT_2$. In particular, let $D_1^*$ be an optimal solution for our original TSRMB problem with $k=1$. Let  use $D_1^*$ to serve  the first-stage riders in $R_1$. As shown above, we know that there exist $|S|$ distinct drivers in $D \setminus D_1^*$, each of whom can serve a rider in $S$ within a distance of at most $OPT_2$. That is, we can construct a matching for the scenario $S$ using the drivers $D \setminus D_1^*$ with a bottleneck cost of at most $OPT_2$, and a matching for $R_1$ using the drivers $D_1^*$ with an average matching cost of $OPT_1$. Therefore, we have
$$ cost_1(D_1, R_1) + cost_2(D \setminus D_1, S)  \leq OPT_1 + OPT_2.$$
Finally, note that the second-stage riders in the original TSRMB problem are either in $S$ or within a distance of at most $2 OPT_2$ from some rider in $S$. Therefore, by the triangle inequality, each rider in $R_2$ can be matched to a driver in $D \setminus D_1$ with a cost of at most: $cost_2(D \setminus D_1, S) + 2 OPT_2.$
Thus, the total cost of our solution for TSRMB is at most:
$$cost_1(D_1, R_1) + cost_2(D \setminus D_1, S) + 2 OPT_2 \leq OPT_1 + 3 OPT_2. $$
This concludes the proof.

\hfill \Halmos
\endproof

\color{black}

\section{Two-Stage Stochastic Bottleneck Matching Problem (TSSMB)}\label{appendix:stochastic}
\subsection{Problem Formulation}
In this section, we consider a variant of the TSRMB  problem with an expected second-stage cost over scenarios instead of a worst-case cost. In particular, we consider a set $R_1$ of first-stage riders which is given first, and must immediately and irrevocably be matched to a subset of drivers $D_1 \subset D$. Once $R_1$ is matched, a scenario $S_i \subset R_2$ is revealed  from a list ${\cal S}= \{ S_1, \ldots, S_q \}
$ with probability $p_i$ and needs to be matched using the remaining drivers. The expected second-stage cost is $\sum\limits_{i=1}^q p_i \cdot cost_2(D\setminus D_1, S_i)$. The objective function is given by
$$ \min\limits_{D_1 \subset D}\Big\{ cost_1(D_1,R_1) + \sum\limits_{i=1}^q p_i \cdot cost_2(D\setminus D_1, S_i)\Big\},$$
where $cost_1(D_1,R_1)$ and $cost_2(D\setminus D_1, S_i)$ are defined similarly to the TSRMB problem.
For brevity of notation, we set $f(D_1) = cost_1(D_1,R_1) + \sum\limits_{i=1}^q p_i \cdot cost_2(D\setminus D_1, S_i)$. Given an optimal first-stage solution $D_1^*$, we denote $OPT_1 = cost_1(D_1^*,R_1)$, $OPT_2 = \sum\limits_{i=1}^q p_i \cdot cost_2(D\setminus D_1^*, S_i)$ and $OPT = OPT_1 + OPT_2$.

{\color{black}
\subsection{Complexity}
\begin{corollary}
There is no $(\frac{5}{3}-\epsilon)$-approximation algorithm for TSSMB for any fixed $\epsilon > 0$, unless $P = NP$.
\end{corollary} 

\proof{\textit{Proof}.} 
The proof follows the same reduction as in the second part of Theorem \ref{thm:hardness_robust}, with the same  instance constructed from a given 3-Dimensional Matching (3-DM) instance. Consider the three second-stage scenarios: \( S_1 = U \), \( S_2 = V \), and \( S_3 = W \), each of size \( k \), and assign equal probabilities \( p_1 = p_2 = p_3 = \frac{1}{3} \). As shown in the proof of Theorem~1, the first-stage cost is \( \epsilon \). 

If a valid 3-dimensional matching exists, each second-stage scenario can be served at bottleneck cost 1, yielding a total expected cost of TSSMB equal to \( \epsilon + \frac{1}{3}\cdot 1 + \frac{1}{3}\cdot 1 + \frac{1}{3}\cdot 1 =1 +\epsilon  \). If no such matching exists, at least one scenario must incur bottleneck cost 3, and the total expected cost of TSSMB becomes at least 
\[
\epsilon + \frac{1}{3} \cdot 1 + \frac{1}{3} \cdot 1 + \frac{1}{3} \cdot 3 = \epsilon + \frac{5}{3}.
\] 
Hence, any algorithm achieving an approximation ratio strictly better than \( \frac{5}{3} - \epsilon \), for some \( \epsilon> 0 \), would distinguish between these two cases and decide 3-DM, implying P = NP. 
\hfill
\Halmos
\endproof
}

\subsection{Algorithm for No Surplus}
Consider the case where there is no surplus of drivers, i.e., the total number of drivers is equal to $|R_1|$ plus the size of the maximum scenario.
We assume for the the sake of simplicity that all scenarios have the same size. The proof follows as well if the sizes are different. We show in this case that we can have a $3$-approximation  by considering every scenario independently to get different first-stage decisions, and then picking the best first-stage decision among them. For every scenario $S_i$, we solve the following problem $$ \min\limits_{D_1 \subset D}\Big\{ cost_1(D_1,R_1) +  cost_2(D\setminus D_1, S_i)\Big\}.$$

When there is no surplus, we know that in the optimal solution, the same set of drivers is matched to every scenario. Therefore if we have a solution for one single scenario, we can use the \red{triangle inequality} to bound the cost of this solution for any scenario.
\begin{algorithm}
	\caption{}
	\label{stochasticnosurplusalgo}
	\begin{algorithmic}[1]

	\FOR{$i \in \{1, \ldots,q\}$}
        \STATE{$D^1_i:=$ TSRMB-1-Scenario$(R_1, S_i, D)$.}	   
	\ENDFOR
	\RETURN{$D_1 = \argmin\limits_{i=1,\ldots,q}f(D^1_i)$}
	\end{algorithmic} 
\end{algorithm}

\begin{theorem}\label{stochasticnosuplusproof}
Algorithm \ref{stochasticnosurplusalgo} yields a solution with total cost at most $OPT_1 + 3 OPT_2$ for TSSMB with no surplus.
\end{theorem}

\proof{\textit{Proof}.}
Let $OPT_1$ and $OPT_2$ be the first-stage and second-stage costs of an optimal solution, and $b_1, \ldots, b_q$ be the bottleneck edge weights in the optimal second-stage matchings for $S_1, \ldots, S_q$. Therefore $OPT_2 = \sum\limits_{i=1}^q p_i \cdot b_i$. Consider $D^1_i$ for $i \in \{1,\ldots,q\}$ obtained from Algorithm \ref{stochasticnosurplusalgo}. We claim that $\min\limits_{i=1,\ldots,q}f(D^1_i) \leq OPT_1 + 3 OPT_2$. For every $i \in \{1,\ldots,q\}$, let $\alpha_i$ and $\beta_i$ be the first-stage and second-stage costs of $D^1_i$ when we consider only scenario $S_i$, that is 
$$cost_1(D^1_i,R_1) = \alpha_i \quad \text{and} \quad   cost_2(D\setminus D^1_i, S_i) =   \beta_i.$$
It is clear that $\alpha_i + \beta_i \leq OPT_1 + b_i$. Furthermore, when a scenario $S_j$ ($j\neq i$) is realized, we can bound the cost of matching $S_j$ to $D\setminus D^1_i$ by using the \red{ triangle inequality},
$$ cost_2(D\setminus D^1_i, S_j) \leq \beta_i + b_i + b_j.$$
Therefore we get that, 
\begin{center}
\begin{tabular}{lll}
    \\
     $f(D^1_i)$ & $\leq$ & $\alpha_i + p_i\cdot \beta_i + \sum\limits_{j\neq i} p_j (\beta_i + b_j + b_i)$\\
     & = &  $\alpha_i + \beta_i + \sum\limits_{j\neq i} p_j (b_j + b_i)$\\
     & $\leq$ & $OPT_1 + p_i \cdot b_i + \sum\limits_{j\neq i} p_j (b_j + 2b_i)$ \\
     & $\leq$ & $OPT_1 + OPT_2 + 2(1-p_i)b_i$.\\
     & & 
\end{tabular}
\end{center}

Next we show that $\min\limits_{1\leq i\leq q} (1-p_i)\cdot b_i \leq OPT_2$. Suppose, in the contrary, that for all $i \in \{1, \ldots, q\}$, we have $(1-p_i)\cdot b_i > OPT_2$, then $p_i \cdot b_i > \frac{p_i}{1-p_i} OPT_2$ and by summing we get that 
\begin{equation*}
    OPT_2 > OPT_2 \sum\limits_{i=1}^q \frac{p_i}{1-p_i}.
\end{equation*}
We can assume without loss of generality that $OPT_2 >0$ and therefore we get that 
\begin{equation*}
    1 > \sum\limits_{i=1}^q \frac{p_i}{1-p_i} = \sum\limits_{i=1}^q \frac{p_i + 1 - 1}{1-p_i} = \sum\limits_{i=1}^q \frac{1}{1-p_i} - q,
\end{equation*}
which implies that
\begin{equation}\label{tocontradict}
    1+q > \sum\limits_{i=1}^q \frac{1}{1-p_i}
\end{equation}
By the Cauchy-Schwarz inequality, we know that
\begin{equation*}
     (q-1)\cdot \sum\limits_{i=1}^q \frac{1}{1-p_i} = \sum\limits_{i=1}^q (1-p_i) \cdot  \sum\limits_{i=1}^q \frac{1}{1-p_i} \geq q^2.
\end{equation*}
Therefore, \begin{equation*}
    \sum\limits_{i=1}^q \frac{1}{1-p_i} \geq \frac{q^2}{q-1} > q+1,
\end{equation*}
which contradicts \eqref{tocontradict}. Hence $\min\limits_{1\leq i\leq q} (1-p_i)\cdot b_i \leq OPT_2$ and therefore $ \min\limits_{1\leq i\leq q}  f(D^1_i) \leq OPT_1 + 3OPT_2.$ 

\hfill \Halmos
\endproof

\section{Two-Stage Robust Matching Problem (TSRM)}\label{appendix:tsrm}
\subsection{Problem Formulation}
In this section, we consider a variant of the TSRMB problem in which \emph{both the first and second-stage costs} are defined as the \emph{total weight} of minimum-cost matchings, in contrast to TSRMB where the first-stage cost is based on the average cost per rider and the second stage uses bottleneck cost. As before, we are given a set \( R_1 \) of first-stage riders, which must be immediately and irrevocably matched to a subset of drivers \( D_1 \subset D \). After this first-stage matching, a scenario \( S_i \subset R_2 \) is revealed from a given list \( \mathcal{S} = \{ S_1, \ldots, S_q \} \).
The first-stage cost, \( \text{cost}_1(D_1, R_1) \), is the total weight of the matching between \( D_1 \) and \( R_1 \), while the second-stage cost, \( \text{cost}_2(D \setminus D_1, S) \), is the total weight of the minimum-cost matching between the realized scenario \( S \) and the available drivers \( D \setminus D_1 \).
Formally, let \( M_1 \) be the minimum-weight perfect matching between \( R_1 \) and \( D_1 \), and for any scenario \( S \in \mathcal{S} \), let \( M_2^S \) be the minimum-weight perfect matching between \( S \) and \( D \setminus D_1 \). Then the cost functions for TSRM are given by:

$$
    cost_1(D_1,R_1)  =  \sum\limits_{(i,j) \in M_1} d(i,j), \quad \mbox{ and } \quad
    cost_2(D\setminus D_1,S)  =  \sum\limits_{(i,j) \in M_2^S} d(i,j).
$$

Given an optimal first-stage solution $D_1^*$, we denote $OPT_1 = cost_1(D_1^*,R_1)$, $OPT_2 = \max \{cost_2(D\setminus D_1^*, S)\; | \; S \in \mathcal{S}\}$ and $OPT = OPT_1 + OPT_2$. In this variant we consider an explicit model of scenarios and optimize over the worst-case scenario. We show that the problem is NP-hard even with two scenarios. We restate a result from the literature that gives a 3-approximation using a greedy approach. We further improve over this approximation in the specific case of no surplus by providing a $7/3$-approximation. We assume for the sake of simplicity that all the scenarios have the same size $k$.

\subsection{Complexity}
\begin{theorem}
TSRM is NP-hard even when the number of scenarios is equal to 2.
\end{theorem}

\proof{\textit{Proof.}}
We construct (in polynomial time) a reduction from the 2-partition
problem. Let $\big(I, (s_i)_{i \in I}\big)$ be an instance of the 2-partition problem. 

\vspace{3mm}

\noindent \textbf{The 2-partition problem:}\\
Instance: Finite set $I$ and number $s_i \in \mathbb{Z}^+$ for $i \in I$.\\
Question: Is there a subset $I' \subset I$ such that $\sum\limits_{i \in I'} s_i = \sum\limits_{i \in I \setminus I'} s_i$?\\ 
It is well known that the 2-partition problem is weakly NP-hard even when $|I'| = |I|/2$.

\vspace{3mm}

Without loss of generality, suppose that $I = \{1, \ldots, n\}$ for some integer $n$. We construct the following instance of TSRM with two scenarios. Let $R_1 = \{r_1, \ldots, r_n\}$, $S_1 = \{r_{n+1}, \ldots, r_{2n}\}$ and $S_2= \{r_{2n+1}, \ldots, r_{3n}\}$. Note that every scenario is of size $n$. Let $D = \{\delta_1, \ldots, \delta_{2n}\}$. Let $P$ be a sufficiently big constant such that $P \geq \sum\limits_{i \in I} s_i$. We define the distances between drivers and riders as follows:
\begin{itemize}
    \item For $j \in \{1,\ldots,n\}$:
$d(\delta_j, r_j) = P$, $d(\delta_{n+j},r_j) = P$ and  $d(\delta_i,r_j) = \infty$ otherwise.
    \item For $j \in \{n+1, \ldots, 2n\}$, $d(\delta_{j-n},r_j) = P$, $d(\delta_j,r_j) = s_{j-n}$ and $d(\delta_i,r_j) = \infty$ otherwise.
    \item For $j \in \{2n+1, \ldots, 3n\}$, $d(\delta_{j-2n},r_j) = s_{j-2n}$, $d(\delta_{j-n},r_j) = P$ and $d(\delta_i,r_j) = \infty$ otherwise.
\end{itemize}
This choice of distances induces a metric bipartite graph on $R_1$ and $S_1 \cup S_2$. A feasible solution to this TSRM instance with bounded cost has two possibilities to match rider $r_j\in R_1$ ($j\leq n$): either to driver $\delta_j$ or driver $\delta_{n+j}$. Consider a feasible bounded cost first-stage solution $D_1$, and let $I'$ be the set of indices $j \leq n$ such that the first-stage rider $r_j$ is matched to driver $\delta_j$ in the first stage. Then $I \setminus I'$ is the set of
elements $j\leq n$ such that the first-stage rider $r_j$ is matched to driver $\delta_{j+n}$. In both cases, the cost of matching $r_j \in R_1$ is equal to $P$. When the scenario $S_1$ is realized, the rider $r_{n+j} \in S_1$ ($j \leq n)$ needs to be matched to $\delta_{j+n}$ if $j \in I'$, with a cost $P$ and to $\delta_{j}$ if $j \in I\setminus I'$, with a cost $s_j$. Similarly, when the scenario $S_2$ is realized, the rider $r_{2n+j} \in S_2$ ($j \leq n)$ needs to be matched to $\delta_{j+n}$ if $j \in I'$, with a cost $s_j$ and to $\delta_{j}$ if $j \in I\setminus I'$, with a cost $P$. The first and second-stage costs are therefore:
$$     cost_1(D_1,R_1)  = P|I|, $$
$$ cost_2(D \setminus D_1, S_1)  = P|I'| + \sum\limits_{j \in I \setminus I'} s_j, \quad \quad \text{and} \quad \quad
    cost_2(D \setminus D_1, S_2)  = P|I \setminus I'| + \sum\limits_{j \in I'} s_j.
$$

We claim that there exists
a 2-partition $I'$ such that $|I'| = |I \setminus I'|$ if and only if there is a solution with total cost equal to
$ \frac{1}{2}\big(3P|I| + \sum\limits_{j \in I} s_j\big)$. 

Suppose there exist a 2-partition $I'$ with $|I'| = |I \setminus I'|$. This implies that \begin{equation}
    \sum\limits_{j \in I'}s_j + P|I\setminus I'| = \sum\limits_{j \in I\setminus I'}s_j + P|I'| = \frac{1}{2}\big(P|I| + \sum\limits_{j \in I} s_j\big).
\end{equation}
Let $D_1$ be the first-stage decision that for every $j \leq n$, matches $r_j$ to $\delta_j$ if $j \in I'$, and $r_j$ to $\delta_{n+j}$ otherwise. The costs of this first-stage decision on scenarios $S_1$ and $S_2$ are:
\begin{flalign*}
    cost_1(D_1,R_1) + cost_2(D \setminus D_1, S_1) & = P|I| + P|I'| + \sum\limits_{j \in I \setminus I'} s_j = \frac{1}{2}\big(3P|I| + \sum\limits_{j \in I} s_j\big),\\
     cost_1(D_1,R_1) +  cost_2(D \setminus D_1, S_2) & = P|I| + P|I \setminus I'| + \sum\limits_{j \in I'} s_j = \frac{1}{2}\big(3P|I| + \sum\limits_{j \in I} s_j\big).
\end{flalign*}
Therefore the total cost of $D_1$ is equal to  $$cost_1(D_1,R_1) + \max\limits_{S \in \{S_1,S_2\}} cost_2(D \setminus D_1, S) = \frac{1}{2}\big(3P|I| + \sum\limits_{j \in I} s_j\big).$$

Suppose now that there is a first-stage decision $D_1$ with bounded total cost equal to $\frac{1}{2}\big(3P|I| + \sum\limits_{j \in I} s_j\big)$. Let $I'$ be the set of indices $j \leq n$ such that, in the first-stage matching of $D_1$, $r_j$ is matched to driver $\delta_j$ for $j \leq n$. We know that  
\begin{flalign*}
     cost_1(D_1,R_1) + cost_2(D \setminus D_1, S_1) & = P|I| + P|I'| + \sum\limits_{j \in I \setminus I'} s_j \leq \frac{1}{2}\big(3P|I| + \sum\limits_{j \in I} s_j\big).\\
      cost_1(D_1,R_1) +  cost_2(D \setminus D_1, S_2) & = P|I| + P|I \setminus I'| + \sum\limits_{j \in I'} s_j \leq \frac{1}{2}\big(3P|I| + \sum\limits_{j \in I} s_j\big).
\end{flalign*}
This implies the following inequalities
\begin{gather}
     P|I'| + \sum\limits_{j \in I \setminus I'} s_j \leq  \frac{1}{2}\big(P|I| + \sum\limits_{j \in I} s_j\big)\label{ineq1}.\\
     P|I \setminus I'| + \sum\limits_{j \in I'} s_j \leq  \frac{1}{2}\big(P|I| + \sum\limits_{j \in I} s_j\big)\label{ineq2}.
\end{gather}
The only way \eqref{ineq1} and \eqref{ineq2} can hold is if we have 
\begin{equation}\label{finish}
    P|I'| + \sum\limits_{j \in I \setminus I'} s_j  = P|I \setminus I'| + \sum\limits_{j \in I'} s_j = \frac{1}{2}\big(P|I| + \sum\limits_{j \in I} s_j\big).
\end{equation} 
Now suppose that $|I\setminus I'| > |I'|$, since we can make $P$ as big as needed, then equation \eqref{finish} cannot hold. Therefore $|I\setminus I'| \leq |I'|$. Similarly, we get that $|I'| \leq |I \setminus I|$, Therefore, $|I'| = |I \setminus I'|$ and equation \eqref{finish} becomes 
\begin{equation*}\label{PIPrime}
 \sum\limits_{j \in I \setminus I'} s_j  = \sum\limits_{j \in I'} s_j.
     \end{equation*}
This shows that $I'$ is a 2-Partition with $|I'| = |I \setminus I'|$. 
\hfill
\Halmos
\endproof

\subsection{Approximation Algorithsm}
In order to show our approximation guarantee.
We will present two algorithms and then balance their performances. The first algorithm (greedy) in Lemma \ref{greedytsrm} works for any setting (meaning with any surplus). The second algorithm (Algorithm \ref{alg:73}) in Lemma \ref{lemma73} assumes that the surplus is 0. The main result for no surplus is given in  Theorem \ref{thm:tamim}.

\begin{lemma}\label{greedytsrm}
The greedy algorithm that minimizes only the first-stage cost yields a solution with total cost at most $3OPT_1 + OPT_2$ to the TSRM.
\end{lemma}

\proof{\textit{Proof.}}
Special case of Theorem 2.4 in \cite{kalyanasundaram1993online}. \hfill \Halmos
\endproof

\begin{lemma}\label{lemma73}
If  the surplus $\ell = |D| - |R_1| - k$ is equal to zero, Algorithm \ref{alg:73} yields a solution to  TSRM with a total cost  at most
 $ OPT_1 + 5 OPT_2 $.
\end{lemma}

\begin{algorithm}
	\caption{}
	\label{alg:73}
	\begin{algorithmic}[1]
	\STATE{Pick a scenario \( S \in \mathcal{S} \).}
	\STATE{Find a subset of drivers \( D_2 \subseteq D \) of size \( |S| \) that minimizes the cost of a perfect matching with \( S \), and compute this matching.}
	\STATE{Compute a minimum weight perfect matching between \( D \setminus D_2 \) and \( R_1 \). Let \( D_1 \) be the set of drivers used in this matching.}
	\RETURN{\( D_1 \)}.
	\end{algorithmic} 
\end{algorithm}

\proof{\textit{Proof of Lemma \ref{lemma73}.}}
Consider Algorithm \ref{alg:73}. In the remaining of the proof, we will refer to the total cost of the solution given by Algorithm \ref{alg:73} as $ALG$, and to its first (resp. second) stage cost as $ALG_1$ (resp. $ALG_2$). Let us show the following two claims.

\begin{claim}
$ALG_2 \leq 3 OPT_2$.
\end{claim}

\proof{\textit{Proof.}}
We use the notation $M(A,B)$ to refer to the total weight of the minimum weight perfect matching between a set of drivers $A$ and a set of riders $B$. If the scenario $S$ that was picked by the algorithm is realized, then in this case we know that its second-stage cost is at most $OPT_2$. Now, suppose a different scenario $S' \neq S$ is realized. Let $D_2^*$ be the set of $k$ drivers that the optimal solution saves for the second stage. We use the \red{ triangle inequality} to establish that :
\begin{equation*}\label{triangular}
    M(D_2, S') \leq M(D_2,S) + M(D_2^*,S) + M(D_2^*,S').
\end{equation*}
Let us bound the right hand side terms of the above equation. We have $M(D_2,S) \leq OPT_2$ because by definition of $D_2$. Now since $|D| = |R_1| + k$, thiss means that the optimal solution saves exactly $k$ drivers to be matched with any scenario realization. This implies that $ M(D_2^*,S) \leq OPT_2$ and $M(D_2^*, S') \leq OPT_2$. The claim follows immediately.
\hfill
\Halmos
\endproof

\begin{claim}
$ALG_1 \leq OPT_1 + 2OPT_2$.
\end{claim}
\proof{\textit{Proof.}}
We construct a matching between $D \setminus D_2$ and $R_1$ with a total weight cost at most $OPT_1 + 2 OPT_2$. Let $r_1 \in R_1$, and $\delta_1(r_1)$ be the driver matched to $r_1$ in the optimal solution. If $\delta_1(r_1) \not\in D_2$, then just match $\delta_1(r_1)$ with $r_1$. Therefore we can assume without loss of generality that all the drivers $\delta_1(r_1)$ are used in $D_2$. This means that exactly $|R_1|$ drivers of  $D\setminus D_2$ are used in second stage of the optimal solution. We can match $R_1$ with $D\setminus D_2$ and bound the cost of this matching as follows:
\begin{equation*}
    M(D\setminus D_2, R_1) \leq M(D\setminus D_2, S) + M(D_2, S) + M(D_2, R_1).
\end{equation*}
We have $ M(D\setminus D_2, S) \leq OPT_2$ because exactly $|R_1|$ drivers from $D\setminus D_2$ are used in the second stage of the optimal solution and $|R_1| = |D \setminus D_2|$. We have $M(D_2, S) \leq OPT_2$ by definition of $D_2$. Finally, $M(D_2, R_1) \leq OPT_1$ because $D_2$ includes all the drivers that were used in the first stage of the optimal matching. Therefore, we get $ ALG_1 = M(D\setminus D_2, R_1) \leq OPT_1 + 2OPT_2. $
\hfill \Halmos
\endproof
The proof of Lemma \ref{lemma73}  follows immediately by combining the above two claims.
\hfill \Halmos
\endproof

The main result for No surplus is given in the following theorem.
\begin{theorem} \label{thm:tamim}
If the surplus $\ell = |D| - |R_1| - k$ is equal to zero, there exists a polynomial time algorithm with a $\frac{7}{3}$-approximation to the TSRM problem.
\end{theorem}
\proof{\textit{Proof.}}
We show the theorem by balancing the results of Lemma \ref{greedytsrm} and Lemma \ref{lemma73}.
Let $Greedy$ denote the total cost of the greedy algorithm and $ALG$ the one of the solution obtained by Algorithm \ref{alg:73}.  From Lemma \ref{greedytsrm} and Lemma \ref{lemma73}, we have that $Greedy \leq 3OPT_1+ OPT_2$ and $ALG \leq OPT_1 + 5OPT_2$. By taking the best of the two algorithms we get:
\begin{align*}
\min\{Greedy, \ ALG\} &= \min\{3OPT_1+ OPT_2, \ OPT_1 + 5OPT_2\} \\
&=  OPT \cdot \min \left\{ \frac{3OPT_1+ OPT_2}{OPT},  \frac{OPT_1 + 5OPT_2}{OPT} \right\}    
\end{align*} 
Since $OPT = OPT_1 + OPT_2$, we get that
$$ \min\{Greedy, \ ALG\} \leq OPT \cdot \max\limits_{x \in [0,1]} \min\big\{3x + (1-x), \ x + 5(1-x)\big\} = \frac{7}{3}\cdot OPT.$$
\hfill \Halmos
\endproof

{\color{black}

\section{IP Formulation for TSRMB with Two Scenarios}\label{appendix:ip_tsrmb}

We present below an integer programming (IP) formulation for the TSRMB problem with two scenarios, $S_1$ and $S_2$. The problem is defined as follows:
\[
\min_{D_1 \subset D} \left\{ \text{cost}_1(D_1, R_1) + \max_{S \in \{S_1, S_2\}} \text{cost}_2(D \setminus D_1, S) \right\},
\]
where $R_1$ is the set of first-stage riders, $D$ is the set of all available drivers, and $D_1$ is the decision variable denoting the subset of drivers assigned to the first stage.

We define the following variables. For each $i \in R_1$ and $j \in D$: $x_{ij} = 1$ if driver $j$ is matched to rider $i$ in the first stage, and $0$ otherwise. For each $\ell \in \{1,2\}$, $i \in S_\ell$, and $j \in D$: $y_{ij\ell} = 1$ if driver $j$ is matched to rider $i$ in scenario $S_\ell$ during the second stage, and $0$ otherwise. Finally, $z \in \mathbb{R}$ denotes the worst-case bottleneck second-stage cost (i.e., the maximum cost over both scenarios).
The resulting IP formulation is:
\begin{equation*}
\begin{array}{ll@{}ll}
\min \quad \quad & \displaystyle \frac{1}{|R_1|}\sum\limits_{i \in R_1, j \in D} d_{ij}x_{ij} + z &\\ \\
\text{s.t.} \quad 
& \displaystyle\sum\limits_{j \in D} x_{ij} = 1,  & \forall \ i \in R_1,  \\
& \displaystyle\sum\limits_{j \in D} y_{ij\ell} = 1,  & \forall \ \ell \in \{1,2\}, \ i \in S_{\ell}, \\
& \displaystyle\sum\limits_{i \in R_1} x_{ij} + \sum\limits_{i \in S_{\ell}} y_{ij\ell} \leq 1, \   & \forall \ j \in D,\ \ell \in \{1,2\},  \\
& z \geq d_{ij} \cdot y_{ij\ell},   & \forall \ \ell \in \{1,2\}, \ i \in S_{\ell},\ j \in D, \\ 
& x_{ij} \in \{0,1\}, & \forall \ i \in R_1,\ j \in D, \\ 
& y_{ij\ell} \in \{0,1\}, & \forall \ \ell \in \{1,2\},\ i \in S_{\ell},\ j \in D. \\ 
\end{array}
\end{equation*}
The first constraint ensures that each first-stage rider is matched to exactly one driver from the set $D$. The second constraint ensures that each second-stage rider, belonging to either scenario $S_1$ or $S_2$, is also matched to exactly one driver. The third constraint enforces that each driver can be used at most once, either in the first stage or in a given second-stage scenario. Note that this is an inequality, as a driver may remain unused. Additionally, because this constraint is written per scenario, a driver not used in the first stage may appear in both scenarios in the second stage (i.e., be reused across scenarios). The fourth constraint ensures that $z$ captures the worst-case bottleneck second-stage cost across all matched pairs in both scenarios. The objective function is the sum of $z$ and the average first-stage matching cost.

\section{Open Problems}\label{openproblems}

While our main focus is on the TSRMB problem, we also initiated the study of three other variants. Below we summarize several open questions, both for TSRMB and for these variants, that remain open for future work.

\vspace{2mm}
\noindent
\textbf{TSRMB.} In this problem, the first-stage cost is the average cost of a minimum-cost perfect matching, and the second-stage cost is the bottleneck cost of a perfect matching. Our results are summarized in Table~\ref{table:results}. There are several compelling open questions:
\begin{itemize}
    \item For the case of two scenarios, we give a 5-approximation and show NP-hardness. Can we close the gap and obtain a better approximation or stronger lower bound?
    \item For general $p$ scenarios, our approximation guarantee is $O(p^{1.59})$. Can we design a constant-factor approximation for large $p$?
    \item Under budgeted uncertainty with small surplus, our current approximation ratio is 17. Can we obtain a constant-factor approximation for arbitrary surplus?
\end{itemize}

\vspace{2mm}
\noindent
\textbf{TSRBB.} In this variant, both the first and second-stage costs correspond to the bottleneck cost in a perfect matching. All our algorithmic results for TSRMB extend to TSRBB, and several hardness results carry over as well. The same open questions as above apply: can the approximation factors be improved or tight lower bounds be shown?

\vspace{2mm}
\noindent
\textbf{TSSMB.} In this stochastic variant, the first-stage cost is the average cost in a minimum-cost perfect matching, and the second-stage cost is the \emph{expected} bottleneck cost across all scenarios. We show a 3-approximation for the case with no surplus and a lower bound of $5/3 - \epsilon$ for general instances. Open questions include:
\begin{itemize}
    \item Can we obtain a constant-factor approximation in the presence of surplus?
    \item Can we close the gap between the 3-approximation and the $5/3$ lower bound?
\end{itemize}

\vspace{2mm}
\noindent
\textbf{TSRM.} In this variant, both the first and second-stage costs are the total weight of a minimum-cost perfect matching, i.e., the sum of edge weights in the matching. There is a greedy 3-approximation algorithm for all uncertainty models and we show NP-hardness. Additionally, we obtain a 7/3-approximation in the special case with no surplus. Open questions include:
\begin{itemize}
    \item Can the 3-approximation be improved in the general case?  or can we strengthen the lower bound for this variant?
    \item Can the 7/3-approximation be improved for the no-surplus case, or is it tight?
\end{itemize}

}
\end{APPENDIX}
%
%






\end{document}